\documentclass[a4paper,12pt,reqno]{amsart}

\usepackage{amsmath, amsfonts, amssymb, amsthm, mathrsfs, mathtools, mathalfa}
\usepackage{comment}
\usepackage{yfonts}
\usepackage{enumerate}
\usepackage{enumitem}
\usepackage{stmaryrd}
\usepackage{fancyhdr}
\usepackage{bm}
\usepackage[utf8]{inputenc}
\usepackage[pdfencoding=auto, hyperfootnotes=false]{hyperref}
\usepackage{tikz,tikz-cd}
\usepackage[hang,flushmargin]{footmisc}
\usepackage{graphicx}
\usepackage{xr}
\usepackage{extarrows}
\usepackage{todonotes}

\usepackage{tikz}
\usetikzlibrary{matrix}
\tikzset{stretch/.initial=1}
\usetikzlibrary{calc}

\newcommand\drawloop[4][]%
   {\draw[shorten <=0pt, shorten >=0pt,#1]
      ($(#2)!\pgfkeysvalueof{/tikz/stretch}!(#2.#3)$)
      let \p1=($(#2.center)!\pgfkeysvalueof{/tikz/stretch}!(#2.north)-(#2)$),
          \n1= {veclen(\x1,\y1)*sin(0.5*(#4-#3))/sin(0.5*(180-#4+#3))}
      in arc [start angle={#3-90}, end angle={#4+90}, radius=\n1]%
   }

\makeatletter
\def\@tocline#1#2#3#4#5#6#7{\relax
  \ifnum #1>\c@tocdepth 
  \else
    \par \addpenalty\@secpenalty\addvspace{#2}%
    \begingroup \hyphenpenalty\@M
    \@ifempty{#4}{%
      \@tempdima\csname r@tocindent\number#1\endcsname\relax
    }{%
      \@tempdima#4\relax
    }%
    \parindent\z@ \leftskip#3\relax \advance\leftskip\@tempdima\relax
    \rightskip\@pnumwidth plus4em \parfillskip-\@pnumwidth
    #5\leavevmode\hskip-\@tempdima
      \ifcase #1
       \or\or \hskip 1em \or \hskip 2em \else \hskip 3em \fi%
      #6\nobreak\relax
    \dotfill\hbox to\@pnumwidth{\@tocpagenum{#7}}\par
    \nobreak
    \endgroup
  \fi}
\makeatother

\newtheorem{theorem}{Theorem}[section]
\newtheorem{lemma}[theorem]{Lemma}

\newtheorem{corollary}[theorem]{Corollary}
\newtheorem{proposition}[theorem]{Proposition}

\theoremstyle{definition}
\newtheorem{defn}[theorem]{Definition}
\newtheorem{remark}[theorem]{Remark}
\newtheorem{example}[theorem]{Example}

\newcommand{\mc}{\mathcal}

\newcommand{\mf}{\mathbf}
\newcommand{\mb}{\mathbb}

\newcommand{\wh}{\widehat}
\newcommand{\wt}{\widetilde}

\newcommand{\ud}{\,\mathrm{d}}
\newcommand{\id}{\mathrm{id}}

\DeclareMathOperator{\aut}{Aut}
\DeclareMathOperator{\ab}{Z}
\DeclareMathOperator{\Bo}{B}
\DeclareMathOperator{\tran}{\Theta}

\DeclareMathOperator{\poly}{poly}

\DeclareMathOperator{\Supp}{Supp}

\DeclareMathOperator{\q}{c}
\DeclareMathOperator{\ns}{X}
\DeclareMathOperator{\nss}{Y}
\DeclareMathOperator{\co}{\circ\hspace{-0.02 cm}}
\DeclareMathOperator{\cu}{C}
\DeclareMathOperator{\cor}{Cor}

\DeclareMathOperator{\Aff}{Aff}

\DeclareMathOperator{\abph}{\mc{U}}

\newcommand*{\sbr}[1]{\scalebox{0.8}{$(#1)$}}
\newcommand*{\db}[1]{\llbracket #1\rrbracket}

\DeclareMathOperator{\tnss}{\tilde{\nss}}

\begin{document}

\title[On $\mb{F}_2^\omega$-affine-exchangeable probability measures]{On $\mb{F}_2^\omega$-affine-exchangeable probability measures}

\author{Pablo Candela}
\address{Universidad Aut\'onoma de Madrid and ICMAT\\ Ciudad Universitaria de Cantoblanco\\ Madrid 28049\\ Spain}
\email{pablo.candela@uam.es}

\author{Diego Gonz\'alez-S\'anchez}
\address{MTA Alfr\'ed R\'enyi Institute of Mathematics\\ 
Re\'altanoda utca 13-15\\
Budapest, Hungary, H-1053}
\email{diegogs@renyi.hu}

\author{Bal\'azs Szegedy}
\address{MTA Alfr\'ed R\'enyi Institute of Mathematics\\ 
Re\'altanoda utca 13-15\\
Budapest, Hungary, H-1053}
\email{szegedyb@gmail.com}

\begin{abstract} 
For any standard Borel space $\Bo$, let $\mc{P}(\Bo)$ denote the space of Borel probability measures on $\Bo$. 
In relation to a difficult problem of Aldous in exchangeability theory, and in connection with  arithmetic combinatorics, Austin raised the question of describing the structure of \emph{affine-exchangeable} probability measures on product spaces indexed by the vector space $\mb{F}_2^\omega$, i.e., the measures in $\mc{P}(\Bo^{\mb{F}_2^\omega})$ that are invariant under the coordinate permutations on $\Bo^{\mb{F}_2^\omega}$ induced by all affine automorphisms of $\mb{F}_2^{\omega}$. We answer this question by describing the extreme points of the space of such affine-exchangeable measures. We prove that there is a single structure underlying every such measure, namely, a random infinite-dimensional cube (sampled using Haar measure adapted to a specific filtration) on a group that is a countable power of the 2-adic integers. Indeed, every extreme affine-exchangeable measure in $\mc{P}(\Bo^{\mb{F}_2^\omega})$ is obtained from a $\mc{P}(\Bo)$-valued function on this group, by a vertex-wise composition with this random cube. The consequences of this result include a description of the convex set of affine-exchangeable measures in $\mc{P}(\Bo^{\mb{F}_2^\omega})$ equipped with the vague topology (when $\Bo$ is a compact metric space), showing that this convex set is a Bauer simplex. We also obtain a correspondence between affine-exchangeability and limits of convergent sequences of (compact-metric-space valued) functions on vector spaces $\mb{F}_2^n$ as $n\to\infty$. Via this correspondence, we establish the above-mentioned group as a general limit domain valid for any such sequence.
\end{abstract}

\maketitle

\section{Introduction}

\noindent Let $\Bo$ be a standard Borel space, with $\sigma$-algebra $\mc{B}$. Let us denote the standard Borel space of probability measures on $\mc{B}$ by $\mc{P}(\Bo,\mc{B})$, or by $\mc{P}(\Bo)$ when the $\sigma$-algebra is clear \cite[p.\ 113]{Ke}. For a countably infinite set  $T$, we denote by $(\Bo^T,\mc{B}^{\otimes T})$ the corresponding (standard Borel) product measurable space.

In the study of exchangeability in probability theory, a central problem consists in describing the structure of measures in $\mc{P}(\Bo^T,\mc{B}^{\otimes T})$ that are invariant under prescribed symmetries, these symmetries being usually the coordinate permutations on $\Bo^T$ induced by a group $\Gamma$ of permutations of $T$. Denoting the set of such measures by $\textrm{Pr}^{\Gamma}(\Bo^T)$ (following the notation in \cite{Austin2}), the aim of such descriptions is to represent any such measure as an image (typically a \emph{mixture}; see \cite[(2.3)]{Aldous}) of simpler measures in $\textrm{Pr}^{\Gamma}(\Bo^T)$. 

The most classical result in this topic is de Finetti's theorem \cite{dF}, which describes $\textrm{Pr}^{\Gamma}(\Bo^T)$ for $\Gamma$ the group of all finitely-supported permutations of $T$. More precisely, in this classical context $\mu$ is said to be an exchangeable measure on $B^T$ if $\mu\co \theta_\gamma^{-1} = \mu$ for every transformation  $\theta_\gamma:B^T\to B^T$ of the form $(v_t)_{t\in T}\mapsto (v_{\gamma(t)})_{t\in T}$, where $\gamma:T\to T$ is a permutation that fixes all but finitely many elements of $T$. De Finetti's theorem states that any such measure $\mu$ is a mixture of product measures $\prod_{t\in T}\lambda$, $\lambda\in \mc{P}(\Bo)$. 

Different notions of exchangeability, involving other groups $\Gamma$, lead to various analogues or extensions of de Finetti's theorem. Principal examples of such extensions include the results of Aldous \cite{A2,A3}, Hoover \cite{H1,H2}, Kallenberg \cite{K} and Kingman \cite{Kingman}. 

Exchangeability theory has various connections beyond probability theory, with combinatorics and ergodic theory among other areas; see \cite{Austin} for a comprehensive survey, bringing to light in particular the relations with the topic of graph limits; see also \cite{D&J,Fra2}. 

This paper concerns notions of exchangeability where $T$ is the infinite discrete cube, which we denote by $\db{\mb{N}}$, i.e.\ the subset of $\{0,1\}^{\mb{N}}$ consisting of those sequences with only finitely many non-zero coordinates. We often identify $\db{\mb{N}}$ as a set with the vector space $\mb{F}_2^\omega:=\bigoplus_{i\in \mb{N}} \mb{F}_2$. This direction originates in a problem posed in the 1980s by Aldous \cite[\S 16]{Aldous}, which asked for a structural description of $\textrm{Pr}^{\Gamma}(\Bo^{\db{\mb{N}}})$ for $\Gamma=\aut(\db{\mb{N}})$, the group of isometries\footnote{This group is isomorphic to $S_\infty\ltimes \mb{Z}_2^\omega$, where $S_\infty$ denotes the group of finitely-supported permutations of $\mb{N}$. Throughout this paper, for any $m\in\mb{N}$ we denote by $\mb{Z}_m$ the abelian group $\mb{Z}/m\mb{Z}$.}  (or discrete-cube automorphisms) of $\db{\mb{N}}$. As explained in \cite[\S 16]{Aldous}, this problem of Aldous did not seem to yield to previous mainstream  methods in this area (see for instance \cite[Example (16.20)]{Aldous}). More than two decades later, Austin shed light on the difficulty of the problem by showing in \cite{Austin2} that the space of $\aut(\db{\mb{N}})$-exchangeable measures on $\Bo^{\db{\mb{N}}}$ (or \emph{cube-exchangeable} measures, as they are called in \cite{Austin2}) has a markedly less tractable structure than those corresponding to other, more classical, exchangeability notions. More precisely, Austin showed that for any (non-trivial) compact metric space $\Bo$, the convex set $\textrm{Pr}^{\aut(\db{\mb{N}})}(\Bo^{\db{\mb{N}}})$ equipped with the \emph{vague} (a.k.a.\ \emph{weak}) topology is a so-called \emph{Poulsen simplex} (meaning that its extreme points form a dense subset), whereas for other more classical exchangeability notions, the convex set $\textrm{Pr}^{\Gamma}(\Bo^T)$ with the vague topology has a more stable structure, in that it is a \emph{Bauer simplex} (meaning that its extreme points form a closed subset). This motivated the study of other natural notions of exchangeability involving the cube $\db{\mb{N}}$, especially notions that are stronger than $\aut(\db{\mb{N}})$-exchangeability and are thus more likely to have simpler descriptions. In particular, motivated by connections with arithmetic combinatorics, in \cite[\S 5.3]{Austin2} Austin highlighted as an interesting object of study the notion of what we will call \emph{affine-exchangeable} measures on $\Bo^{\db{\mb{N}}}$, whose corresponding group of symmetries $\Gamma$ is the group of (invertible) affine transformations on the cube $\db{\mb{N}}$ (identified with $\mb{F}_2^\omega$). Let us denote this group by $\textrm{Aff}(\mb{F}_2^\omega)$ (following \cite{Austin2}), and note that $\textrm{Aff}(\mb{F}_2^\omega)\cong \textrm{GL}(\mb{F}_2^\omega)\ltimes \mb{Z}_2^\omega$. 

The main results of this paper are, firstly, a solution of the description problem for affine-exchangeability, in the form of a representation theorem for affine-exchangeable measures (Theorem \ref{thm:main} below); secondly, an application concerning limit objects for certain convergent sequences of functions, in the spirit of the recent limit theories for graphs and hypergraphs \cite{BCLSV, L12, LS06, LS07}. To introduce this application, let us begin by mentioning that in \cite{S} the third named author proved an analogous result for limits of functions defined on abelian groups, with a notion of convergence involving linear forms of (true) complexity 1 (in the sense of \cite{GW}). In this paper we consider limits of sequences of functions $(f_n:\mb{F}_2^n\to \Bo)_{n\in \mb{N}}$ where $\Bo$ is a compact metric space and where the notion of convergence involves systems of linear forms of arbitrary finite complexity. Variants of this notion have been considered previously (e.g.\ for Boolean functions in  \cite{HHH}); we shall define it in detail in the sequel (in Section \ref{sec:exch-limits}) but let us briefly outline the notion here. Given a sequence $(f_n:\mb{F}_2^n\to \Bo)_{n\in \mb{N}}$ for $\Bo$ a compact metric space, we can identify any given system of finitely many linear forms over $\mb{F}_2$ as a set $\mc{L}\subset \mb{F}_2^k\times \{0^{\mb{N}\setminus[k]}\}\subset \mb{F}_2^\omega$ for some $k\in\mb{N}$ (where $0^{\mb{N}\setminus[k]}$ denotes the constant 0 sequence with coordinates indexed by $\mb{N}\setminus[k]$). Then, denoting by $\mu_{\cu^k(\mb{F}_2^n)}$ the Haar probability measure on the group of standard $k$-dimensional cubes in $\mb{F}_2^n$, and noting that any such cube can be viewed as an affine-linear map $A:\mb{F}_2^k\to\mb{F}_2^n$, we can define the measure $\mu_{\mc{L},f_n}$ on $\Bo^{\mc{L}}$ as the pushforward of $\mu_{\cu^k(\mb{F}_2^n)}$ under the map $A\mapsto (f_n\co A(L))_{L\in \mc{L}}$ (see \eqref{eq:samplemeas}). We then say that the sequence $(f_n)_{n\in\mb{N}}$ converges if for every finite $\mc{L}\subset \mb{F}_2^\omega$ the measures $\mu_{\mc{L},f_n}$ converge vaguely as $n\to\infty$. Our second main result (Theorem \ref{thm:mainlim} below) establishes that every such sequence $(f_n)$ converges to a limit object closely related to our representation theorem for affine-exchangeable measures.

Before discussing these results in more detail, let us give some basic examples of affine-exchangeable measures.

\begin{example}\label{ex:1} 
Let $\Bo$ be a standard Borel space and fix any function $m:\mb{Z}_2\to \mc{P}(\Bo)$. Let  $G$ denote the closed subgroup of the compact abelian group $\mb{Z}_2^{\db{\mb{N}}}$ consisting of what we can view as infinite-dimensional cubes. More precisely, the elements of $G$ are the functions $\q:\db{\mb{N}}\to \mb{Z}_2$ of the form $\q(v)=x_0+\sum_{i=1}^{\infty} v\sbr{i} h_i$ where $x_0\in\mb{Z}_2$, $h_i\in \mb{Z}_2$ for all $i\in \mb{N}$ and $v\in \db{\mb{N}}$. Note that, since every $v\in \db{\mb{N}}$ has at most finitely many non-zero coordinates, the  last sum above is always well-defined. Letting $\mu_G$ denote the Haar probability measure on $G$, we can then define a  measure $\mu\in \mc{P}(\Bo^{\db{\mb{N}}})$ by the formula
\begin{equation}\label{eq:basicex1}
\mu =\int_G \prod_{v\in \db{\mb{N}}} m(\q(v)) \ud\mu_G(\q),
\end{equation}
where the product here denotes a countable product of measures. Thus, for any cylinder set $S\subset \Bo^{\db{\mb{N}}}$, i.e.\ a set of the form $S=\prod_{v\in \db{\mb{N}}} A_v\subset B^{\db{\mb{N}}}$ with $A_v\not=\Bo$ only for $v$ in some \emph{finite} set $V\subset \db{\mb{N}}$ (recall that these cylinder sets generate the product $\sigma$-algebra $\mc{B}^{\otimes T}$), we have $\mu(S)= \int_G  \prod_{v\in V}m(\q(v))(A_v)\ud\mu_G(\q)$. It is readily seen that $\mu$ is an affine-exchangeable measure, using the fact that for every $\gamma\in \textrm{Aff}(\mb{F}_2^\omega)$ the transformation $G\to G$, $\q\mapsto \q\co\gamma$ preserves the measure $\mu_G$. Note also that, since the map sending $(x_0,h_1,h_2,\ldots)$ to $\q(v)=x_0+\sum_{i=1}^{\infty} v\sbr{i} h_i$ is an isomorphism $\mb{Z}_2^{\mb{N}}\to G$, we can rewrite formula \eqref{eq:basicex1} as follows (where $\mu_{\mb{Z}_2^{\mb{N}}}$ is the Haar measure on $\mb{Z}_2^{\mb{N}}$):
\begin{equation}\label{eq:basicex1-2}
\mu =\int_{\mb{Z}_2^{\mb{N}}} \prod_{v\in \db{\mb{N}}} m(x_0+v\sbr{1}h_1+v\sbr{2}h_2+\cdots)\ud\mu_{\mb{Z}_2^{\mb{N}}}(x_0,h_1,h_2,\ldots).
\end{equation}
\end{example}
\noindent As mentioned above, the elements $\q\in G$ can be viewed as infinite-dimensional combinatorial cubes on $\mb{Z}_2$ (the integral in \eqref{eq:basicex1} or \eqref{eq:basicex1-2} could thereby be viewed as an infinite-dimensional, measure-valued, Gowers inner-product of the function $m$ with itself). By generalizing this underlying notion of infinite-dimensional cubes, we soon obtain a large family of   examples of affine-exchangeable measures. In particular, we have the following construction where the group $\mb{Z}_2$ is replaced by more general filtered abelian groups.

\begin{example}\label{ex:2} 
Let $Z$ be a compact abelian group, and let  $Z_{\bullet}=(Z_{(j)})_{j\ge 0}$ be a filtration on $Z$, i.e., a sequence of closed subgroups of $Z$ with $Z_{(j)}\supseteq Z_{(j+1)}$ for all $j\ge 0$ and $Z_{(0)}=Z_{(1)}=Z$. We can then generalize the group $G$ from Example \ref{ex:1} to obtain the group of infinite-dimensional cubes \emph{adapted to the filtration} $Z_\bullet$. This group, which we shall denote by $\cu^{\omega}(Z_\bullet)$, is the closed subgroup of $Z^{\db{\mb{N}}}$ generated by elements $f\in Z^{\db{\mb{N}}}$ of the form $f(v_1,v_2,\ldots) = av_{j_1}\cdots v_{j_s}$ where $s\in \mb{N}$,  $a\in Z_{\sbr{s}}$ and $j_1,\ldots,j_s\in \mb{N}$. In standard terminology related to cubic structures, this construction of the group $\cu^{\omega}(Z_\bullet)$ essentially follows that of the \emph{Host--Kra cubes} associated with a filtered group, except that here the construction is carried out in a setting of countably infinite dimension (we refer to Part 2 of the book \cite{HKbook} for more background on cubic structures, and to \cite[\S 2.2.1]{Cand:Notes1} for a treatment of Host--Kra cubes in a nilspace theoretic context as used in this paper). Let $\mu_{\cu^{\omega}(Z_\bullet)}$ be the Haar measure on $\cu^{\omega}(Z_\bullet)$. Then for any standard Borel space $\Bo$ and Borel map $m:\ab\to \mc{P}(\Bo)$, we can define the following measure in $\mc{P}(\Bo^{\db{\mb{N}}})$ similarly to \eqref{eq:basicex1}:
\begin{equation}\label{eq:basicex2}
\mu =\int_{\cu^{\omega}(Z_\bullet)} \prod_{v\in \db{\mb{N}}} m(\q(v))\ud\mu_{\cu^{\omega}(Z_\bullet)}(\q).
\end{equation}
It can be checked that $\mu$ is affine-exchangeable provided that the filtration $Z_\bullet$ satisfies a simple property that we call \emph{2-homogeneity}, which states that for every $i\in \mb{N}$ we have $2g\in Z_{(i+1)}$ for every $g\in Z_{(i)}$ (see Corollary \ref{cor:2-hom-aff-exch}).
\end{example}
\noindent More examples can be obtained by further generalizing the underlying cubic structures; we shall see such constructions in detail in the sequel (e.g.\ in Section \ref{sec:prelims}). In fact, affine-exchangeability is a special case of a more general notion of exchangeability involving $\db{\mb{N}}$, called \emph{cubic exchangeability}, which was introduced in \cite{CScouplings} (and is itself strictly stronger than the notion of Aldous from \cite[\S 16]{Aldous}). We shall now briefly discuss this connection, and the main results concerning cubic exchangeability from \cite{CScouplings}, as they provide useful ingredients and motivations for the main results in this paper. 

A measure on $\Bo^{\db{\mb{N}}}$ is \emph{cubic-exchangeable} if for each integer $k\geq 0$, for any pair of injective discrete-cube morphisms\footnote{A map $\phi:\{0,1\}^k\to\db{\mb{N}}$ is a \emph{morphism} if it extends to an affine homomorphism $\mb{Z}^k\to \bigoplus_{i\in \mb{N}}\mb{Z}$; see \cite[Definition 6.1]{CScouplings}. If $\phi$ is injective then its image can be viewed as a $k$-dimensional affine subcube of $\db{\mb{N}}$.}  $\phi_1,\phi_2:\{0,1\}^k\to \db{\mb{N}}$ , the two images of the measure under the coordinate projections to $\phi_1(\{0,1\}^k)$ and $\phi_2(\{0,1\}^k)$ are equal; see \cite[Definition 6.3]{CScouplings}.\footnote{As such, cubic exchangeability does not involve invariance under a \emph{group} $\Gamma$ like other exchangeability notions, although it can be formulated in terms of invariance under a \emph{semigroup}.} The main result in \cite{CScouplings} concerning cubic exchangeability is a representation theorem for cubic-exchangeable measures on $\Bo^{\db{\mb{N}}}$. This theorem describes the extreme points in the space of such measures in terms of a specific construction which involves \emph{compact nilspaces}; see \cite[Theorem 6.7]{CScouplings}. Nilspaces are structures in which combinatorial cubes can be defined in great generality, so their connection with exchangeability notions involving $\db{\mb{N}}$ is natural. In particular, nilmanifolds (central examples of compact nilspaces) were already related to certain notions of exchangeability in the work of Frantzikinakis in \cite{Fra2}. Originating in work of Host and Kra (notably \cite{HK, HKparas}), nilspaces were introduced in \cite{CamSzeg}, and by now there is a fair amount of literature on these objects \cite{Cand:Notes1, Cand:Notes2, GMV1, GMV2, GMV3}. We shall recall some basic nilspace theory in this paper, but we defer this to Section \ref{sec:prelims}. On the other hand, the construction of cubic-exchangeable measures involving compact nilspaces is more recent and plays a key role in this paper's main results, so let us recall the construction here.
\begin{defn}\label{def:extreme-aff-ex}[Nilspace construction of extreme cubic-exchangeable measures]\\
 Let $(\Bo,\mc{B})$ be a standard Borel space, let $\ns$ be a compact nilspace and let $m:\ns\to \mc{P}(\Bo)$ be a Borel map. Let $\cu^\omega(\ns)$ denote the space of infinite-dimensional cubes $\db{\mb{N}}\to \ns$, equipped\footnote{The set $\cu^\omega(\ns)$ is a natural  generalization for compact nilspaces of the group $\cu^\omega(Z_\bullet)$ from Example \ref{ex:2}. The related notions, in particular the Haar measure on $\cu^\omega(\ns)$, will be explained in Subsection \ref{subsec:CinfHaar}.} with its Haar probability measure $\mu_{\cu^\omega(\ns)}$. Then we define the associated cubic-exchangeable measure $\zeta_{\ns,m}$ by the following formula (where the product denotes a countable product of measures):
\begin{equation}\label{eq:zeta}
\zeta_{\ns,m} = \int_{\cu^\omega(\ns)} \prod_{v\in \db{\mb{N}}} m(\q(v))\; \ud\mu_{\cu^\omega(\ns)}(\q).
\end{equation}
\end{defn}
\noindent In \cite{CScouplings}, this construction was explained from the more \emph{statistical} point of view on exchangeability, in terms of joint distributions of sequences of random variables. The present paper adopts a more \emph{ergodic-theoretic} viewpoint on exchangeability, working  directly with $\Gamma$-invariant measures on product spaces (see e.g.\ \cite{Austin2}); this viewpoint motivates formula \eqref{eq:zeta}, expressing the construction directly as a measure on $\Bo^{\db{\mb{N}}}$ (both viewpoints are useful).

Note that the basic example in \eqref{eq:basicex1} is a special case of \eqref{eq:zeta}, indeed the measure $\mu$ in \eqref{eq:basicex1} is $\zeta_{\ns,m}$ where the underlying nilspace $\ns$ consists of the group $\mb{Z}_2$ (with the standard abelian cube structure determined by the lower-central series on $\mb{Z}_2$), and where $\cu^\omega(\ns)$ is the group $G$ in that example. Similarly, the measure in \eqref{eq:basicex2} is $\zeta_{\ns,m}$ where $\ns$ is the so-called \emph{group nilspace} associated with the filtered group $(Z,Z_\bullet)$, i.e.\ the nilspace consisting of $Z$ equipped with the Host--Kra cubes determined by the filtration $Z_\bullet$ (see e.g.\ \cite[\S 2.2.1]{Cand:Notes1}).

To summarize, from previous works we have three exchangeability notions involving the cube  $\db{\mb{N}}$, which are related as follows:\footnote{A more detailed explanation of these inclusions is given in Lemma \ref{lem:clarifinclu} below.
}\vspace{0.2cm}
\begin{equation}\label{eq:inclu}
\textrm{Pr}^{\textrm{Aff}(\mb{F}_2^\omega)}(\Bo^{\db{\mb{N}}})  
\; \subset \; \{\textrm{cubic-exchangeable measures in }\mc{P}(\Bo^{\db{\mb{N}}})\} 
\;  \subset \; \textrm{Pr}^{\aut(\db{\mb{N}})}(\Bo^{\db{\mb{N}}}).
\end{equation}
Moreover, while $\textrm{Pr}^{\aut(\db{\mb{N}})}(\Bo^{\db{\mb{N}}})$ has no representation akin to the classical ones (as shown in \cite{Austin2}), by contrast cubic exchangeability does admit such a representation, which tells us that the extreme cubic-exchangeable measures are exactly the measures of the form $\zeta_{\ns,m}$ (namely  \cite[Theorem 1.3 or Theorem 6.7]{CScouplings}).

An observation leading to the main results of this paper is that the exactness of the above description of cubic exchangeability fails for the stronger notion of affine exchangeability. More precisely, for a measure of the form $\zeta_{\ns,m}$ to be affine-exchangeable, the underlying nilspace $\ns$ must have a more specific structure than in the general cubic-exchangeable setting. These more specific structures are the so-called \emph{2-homogeneous nilspaces}, introduced in \cite{CGSS} (and recalled in more detail below). As a consequence, when we replace cubic exchangeability with affine exchangeability, the representation theorem from \cite{CScouplings} can be significantly refined, leading to one of the main results of this paper, Theorem \ref{thm:main} below. In particular, the potentially varying compact nilspaces $\ns$ in the representation in \cite[Theorem 6.7]{CScouplings} can all be replaced here by a \emph{single} compact nilspace, defined using the group of $2$-adic integers. To formalize this, we use the following notation. 

Throughout most of the paper (except in Section \ref{sec:p-hom}), we shall denote the group of 2-adic integers by $\mf{Z}$ (instead of $\mf{Z}_2$)  to avoid an overload of sub-indices in the subsequent notations. For any integer $\ell\ge 1$ we denote by $\mf{Z}_{\bullet,\ell}$ the filtration on $\mf{Z}$ with $i$-th term $\mf{Z}_{(i)} = \mf{Z}$ if $i=0,\ldots,\ell$ and $\mf{Z}_{(i)}=2^{i-\ell}\mf{Z}=\{2^{i-\ell}x:x\in \mf{Z}\}$ for $i>\ell$.
\begin{defn}\label{def:H}
We denote by $\mc{H}$ the group nilspace consisting of the compact abelian group\footnote{By a slight abuse of notation we shall often denote this group also by $\mc{H}$.}   $\prod_{\ell=1}^{\infty} \mf{Z}^{\mb{N}}$ equipped with the cube structure determined by the product filtration $\mc{H}_{\bullet}=\prod_{\ell=1}^{\infty} \mf{Z}_{\bullet,\ell}^{\mb{N}}$, that is, the filtration with $i$-th term $\mc{H}_{(i)}=\mc{H}$ for $i=0,1$ and $\mc{H}_{(i)}= (2^{i-1} \mf{Z}^{\mb{N}}) \times (2^{i-2} \mf{Z}^{\mb{N}}) \times \cdots$ for $i\geq 2$.
\end{defn}

\noindent Recall that a measure $\mu$ on  a $\sigma$-algebra $\mc{A}$ is said to be \emph{concentrated} on a set $S\in \mc{A}$ if for every $B\in \mc{A}$ we have $\mu(B)=\mu(B\cap S)$. We can now state our main result.

\begin{theorem}\label{thm:main}
Let $\Bo$ be a standard Borel space and let $\mu\in\textup{Pr}^{\textup{Aff}(\mb{F}_2^\omega)}(\Bo^{\db{\mb{N}}})$. Then there is a Borel probability measure $\kappa$ on $\mc{P}(\Bo^{\db{\mb{N}}})$, which is concentrated on the set\footnote{The set of such measures is shown to be Borel in Lemma \ref{lem:indepclosed}.} of measures $\{\nu\in \mc{P}(\Bo^{\db{\mb{N}}}) ~|~ \nu =\zeta_{\mc{H},m} \textup{ for some Borel map } m:\mc{H}\to\mc{P}(\Bo)\}$, such that $\mu=\int_ {\mc{P}(\Bo^{\db{\mb{N}}})} \nu\ud\kappa$.
\end{theorem}
\noindent In other words, every affine-exchangeable measure is a mixture of measures $\zeta_{\mc{H},m}$. 

There are several equivalent ways to describe these measures $\zeta_{\mc{H},m}$, apart from formula \eqref{eq:zeta}. Let us give here a particularly simple description that does not involve any nilspace theory: we have $\zeta_{\mc{H},m}=\int_G \prod_{v\in \db{\mb{N}}} m(\q(v))\ud\mu_G(\q)$ where $G$ is a closed subgroup of $\mc{H}^{\mb{F}_2^\omega}$ determined by a countable set of linear equations. More precisely, if for any map $f:\mb{F}_2^k\to\mc{H}$ (for any $k\in \mb{N}$) we denote by $\sigma_k(f)$ the alternating sum $\sum_{v\in \mb{F}_2^k} (-1)^{v\sbr{1}+\cdots+v\sbr{k}}f(v)$, then $G$ is the compact abelian group consisting of all maps $\q: \mb{F}_2^\omega\to \mc{H}$ such that for every $k$ and every injective affine homomorphism $\phi:\mb{F}_2^k\to\mb{F}_2^\omega$ we have $\sigma_k(\q\co\phi)=0\!\mod \mc{H}_{(k)}$. This description of $G$ using linear equations (modulo subgroups) is equivalent to the construction in \eqref{eq:basicex2} (in particular $G=\cu^\omega(\mc{H}_\bullet)$), by basic results on Host--Kra cubes (see e.g.\ \cite[\S 2.2.3]{Cand:Notes1}).

Theorem \ref{thm:main} yields a description of the geometry of the set $\textup{Pr}^{\textup{Aff}(\mb{F}_2^\omega)}(\Bo^{\db{\mb{N}}})$ for any standard Borel space $\Bo$. Indeed, as detailed in Section \ref{sec:geom}, the extreme points of this convex set are exactly the measures of the form $\zeta_{\mc{H},m}$. Moreover, the following holds.
\begin{theorem}\label{thm:BauerProp}
For any compact metric space $\Bo$, the convex set $\textup{Pr}^{\textup{Aff}(\mb{F}_2^\omega)}(\Bo^{\db{\mb{N}}})$ equipped with the vague topology is a Bauer simplex.
\end{theorem}
\noindent This \emph{Bauer property} was identified in \cite[\S 3.1]{Austin2} to be a common feature of other mainstream exchangeability notions. Theorem \ref{thm:BauerProp} shows that affine exchangeability also follows this pattern.

As mentioned above, Theorem \ref{thm:main} has applications concerning limit objects for a notion of convergence for sequences of functions $(f_n:\mb{F}_2^n\to \Bo)_{n\in \mb{N}}$ for a fixed compact metric space $\Bo$. This notion is analogous to well-known and much studied notions of convergence for sequences of graphs and hypergraphs. These applications rely on a general connection between representation theorems for exchangeable measures in probability theory, and various limit theories in combinatorics; for more background we refer to \cite[\S 2.3]{Austin}, which discusses this connection in the setting of exchangeable random hypergraph colorings. In the present paper, we take the opportunity to develop this connection in the setting of affine exchangeability. To explain this, let us say that a 2-homogeneous compact nilspace $\ns$ \emph{represents affine exchangeability} if it can replace $\mc{H}$ in Theorem \ref{thm:main}. Given a compact metric space $\Bo$, let us also say that a compact 2-homogeneous nilspace $\ns$ is a \emph{limit domain} for convergent sequences $(f_n:\mb{F}_2^n\to \Bo)_{n\in \mb{N}}$, if for any such sequence there exists a measurable function $m:\ns\to \mc{P}(\Bo)$ such that for any finite $\mc{L}\subset \mb{F}_2^\omega$ the measures $\mu_{\mc{L},f_n}$ converge vaguely to\footnote{ $p_{\mc{L}}:\Bo^{\db{\mb{N}}}\to \Bo^{\mc{L}}$ is the projection to the coordinates indexed by $\mc{L}\subset \mb{F}_2^\omega$ (identifying $\mb{F}_2^\omega$ with $\db{\mb{N}}$).} $\zeta_{\ns,m}\co p_{\mc{L}}^{-1}$. In Section \ref{sec:exch-limits}, we prove the following result, which connects these concepts.
\begin{theorem}\label{thm:repre-with-2-hom}
Let $\ns$ be a compact profinite-step\footnote{Meaning that $\ns$ is the inverse limit of compact nilspaces of finite step; see Subsection \ref{subsec:profin}.} 2-homogeneous nilspace. The following statements are equivalent.
\begin{enumerate}[leftmargin=0.8cm]
\item $\ns$ represents affine-exchangeability.
\item For every compact 2-homogeneous profinite-step nilspace $\nss$ there is a \textup{(}continuous\textup{)} fibration $\varphi:\ns\to \nss$.
\item $\ns$ is a limit domain for convergent sequences $(f_n:\mb{F}_2^n\to\Bo)_{n\in \mb{N}}$, for every compact metric space $\Bo$.
\end{enumerate}
\end{theorem}
\noindent Combining this with Theorem \ref{thm:main}, we deduce immediately the following result.
\begin{theorem}\label{thm:mainlim}
For any compact metric space $\Bo$, the nilspace $\mc{H}$ is a limit domain for convergent sequences of functions $(f_n:\mb{F}_2^n\to\Bo)_{n\in\mb{N}}$. 
\end{theorem}
\noindent The group nilspace $\mc{H}$, involved in Theorems \ref{thm:main} and \ref{thm:mainlim}, may at first seem quite complicated, and one may wonder whether a simpler object could still yield such representation results. However, statement $(ii)$ in Theorem \ref{thm:repre-with-2-hom} constitutes a strong requirement for a nilspace to be able to represent affine exchangeability, putting constraints on how simple such a nilspace can be. In Section \ref{sec:repvsnils} we use this fact to rule out certain other natural candidates. In particular, we were interested in clarifying the relation between our results concerning limit domains and the results in \cite{HHH}. The latter paper studies a convergence notion for sequences of \emph{Boolean} functions $(f_n:\mb{F}_p^n\to \{0,1\})_{n\in \mb{N}}$, for any prime $p$, and for $p=2$ that notion is equivalent to the one studied in this paper (we detail this equivalence in Appendix \ref{app:correspondence}). However, in Subsection \ref{subsec:restrictions-on-lim-dom} we prove that the group $\mb{G}_\infty$ used to define the \emph{limit objects} in \cite{HHH} cannot be used as a limit domain in the sense of statement $(iii)$ in Theorem \ref{thm:repre-with-2-hom}.

The paper has the following outline. In Section \ref{sec:prelims} we give some background on nilspaces and set up some measure-theoretic tools needed in the sequel, especially the machinery involving the infinite-dimensional cube set $\cu^\omega(\ns)$ on a compact nilspace $\ns$, and the associated construction and basic properties of the measures $\zeta_{\ns,m}$. We take the opportunity, in Subsection \ref{subsec:basicex}, to give further basic examples illustrating affine-exchangeability, which also help to motivate the subsequent material. In Section \ref{sec:2-homCC}, we begin proving Theorem \ref{thm:main}, by first refining the representation theorem \cite[Theorem 1.3]{CScouplings} in light of more recent results on higher-order Fourier analysis in characteristic $p$ from \cite{CGSS-p-hom}. This relies especially on the results concerning \emph{$p$-homogeneous nilspaces}. More precisely, we prove in Section \ref{sec:2-homCC} that the additional strength of affine exchangeability (compared to cubic exchangeability) implies that the nilspaces obtained by applying \cite[Theorem 6.7]{CScouplings} to an affine-exchangeable measure must be \emph{2-homogeneous} nilspaces. This enables us to apply the structure theorem for 2-homogeneous nilspaces obtained in \cite{CGSS-p-hom}, which tells us that any finite 2-homogeneous nilspace is the image, under some nilspace fibration, of a certain finite filtered abelian 2-group. The next part of the argument, carried out in Section \ref{sec:p-hom},  involves setting up adequate inverse systems of such filtered 2-groups, in order to prove that for any compact 2-homogeneous nilspace $\ns$ there is a fibration $\mc{H}\to\ns$. These ingredients are then combined in Section \ref{sec:mainproof} to complete the proof of Theorem \ref{thm:main}. Section \ref{sec:geom} is devoted to proving Theorem \ref{thm:BauerProp}. In Section \ref{sec:exch-limits} we prove a version of Theorem \ref{thm:repre-with-2-hom} focusing on the relation between limit domains and affine-exchangeability; see Theorem \ref{thm:corresp}. This is then used  in Section \ref{sec:repvsnils} to complete the proof of Theorem \ref{thm:repre-with-2-hom}, and thus obtain also Theorem \ref{thm:mainlim}. In Subsection \ref{subsec:restrictions-on-lim-dom} we study the aforementioned relation with \cite{HHH} concerning alternative limit domains.

Finally, let us mention the possibility of extending the methods in this paper to any prime $p>2$. Combined with corresponding extensions of prior work (notably the results from \cite[\S 6]{CScouplings}) this would yield representation theorems for analogues of affine exchangeability for $p>2$. We do not pursue this extension in this paper; see Remark \ref{rem:extensionto-p}. Another  natural direction would be to restrict the true complexity of the linear forms involved in the definition of convergence of sequences $(f_n:\mb{F}_2^n\to \Bo)_{n\in \mb{N}}$, i.e., by requiring vague convergence of $\mu_{\mc{L},f_n}$ only for systems of linear forms $\mc{L}$ of true complexity \emph{at most} some prescribed finite bound. It can be seen that $\mc{H}$ is a valid limit domain also to describe the limits of such sequences. Moreover, with more work it is possible to use a simpler version of $\mc{H}$ in that case. We outline how this can be done in Remark \ref{rem:finite-complexity-limit}.

\smallskip

\noindent \textbf{Acknowledgements.} 
We thank Hamed Hatami, Pooya Hatami, and Bryna Kra for valuable comments on this paper. All authors received funding from Spain’s MICINN project PID2020-113350GB-I00. The second-named author received funding from project Momentum (Lend\"ulet) 30003 of the Hungarian Government. The research was also supported partially by the NKFIH ``Élvonal'' KKP 133921 grant and partially by the Hungarian Ministry of Innovation and Technology NRDI Office within the framework
of the Artificial Intelligence National Laboratory Program.

\section{Nilspace and measure-theoretic preliminaries}\label{sec:prelims}

\noindent In this section we provide a brief introduction to nilspaces and related measure-theoretic aspects needed for the sequel. In particular we introduce the space of infinite-dimensional cubes $\cu^{\omega}(\ns)$ on a compact nilspace $\ns$, and define what we shall call the \emph{Haar measure} on this space; this underpins the construction of the measures $\zeta_{\ns,m}$. Finally, we give more background and examples on cubic and affine exchangeability.

\subsection{Brief introduction to nilspaces}\hfill\\
For any integer $n\ge 1$ we denote by $[n]$ the set $\{1,2,\ldots,n\}$. We define $\db{n}:=\{0,1\}^n$,  and define $\db{0}:=\{0\}$. We denote by $0^n$ (resp.\ $1^n$) the element of $\db{n}$ with all entries equal to 0 (resp.\ 1). Following \cite[Definition 1.0.1]{Cand:Notes2}, by \emph{compact space} we shall mean by default a compact, second-countable, Hausdorff topological space. 

A map $\phi:\db{n}\to\db{m}$ is a \emph{discrete-cube morphism} \cite{CamSzeg},\cite[Definition 1.1.1]{Cand:Notes1} if for every $j\in[m]$ the $j$-th coordinate of $\phi(v)$ is either a constant function of $v$ (equal to $0$ or $1$) or is equal to $v\sbr{i_j}$ or $1-v\sbr{i_j}$ for some $i_j\in [n]$. For $n\le m$ we say that a discrete-cube morphism  $\phi:\db{n}\to\db{m}$ is a \emph{face map} if $\phi$ is injective and exactly $m-n$ coordinates of $\phi(v)$ remain constant as $v$ varies through $\db{n}$.
\begin{defn}[Nilspaces \cite{CamSzeg}] 
A \emph{nilspace} is a set $\ns$ together with a collection of sets $\cu^n(\ns)\subset \ns^{\db{n}}$, $n\in\mb{N}$, such that the following axioms are satisfied.
\begin{enumerate}[leftmargin=0.7cm]
    \item (Composition)\; For any discrete-cube morphism $\phi:\db{n}\to\db{m}$ and any $\q\in \cu^{m}(\ns)$, we have $\q\co\phi\in\cu^n(\ns)$.
    \item (Ergodicity)\; $\cu^1(\ns)=\ns^{\db{1}}$.
    \item (Corner completion)\; For any $n$ let $\q': \db{n}\setminus\{1^n\} \to  \ns$ be such that for all face maps $\phi:\db{n-1}\to\db{n}$ with $1^n\notin \phi(\db{n-1})$ we have $\q'\co\phi\in\cu^{n-1}(\ns)$. Then there exists $\q\in\cu^n(\ns)$ such that $\q(v)=\q'(v)$ for all $v\in\db{n}\setminus\{1^n\}$.
\end{enumerate}
\noindent The elements of $\cu^n(\ns)$ are called the $n$-\emph{cubes} on $\ns$. The maps $\q': \db{n}\setminus\{1^n\} \to  \ns$ satisfying the assumption in the completion axiom are called the $n$\emph{-corners} on $\ns$, and a cube $\q$ satisfying the conclusion of this axiom is a \emph{completion} of $\q'$. We denote the set of all $n$-corners on $\ns$ by $\cor^n(\ns)$. We say that $\ns$ is a \emph{$k$-step} nilspace if each $\q\in\cor^{k+1}(\ns)$ has a \emph{unique} completion.

If $\ns$ is endowed with a topology making it a compact space, and  for every $n\in\mb{N}$ the cube set $\cu^n(\ns)$ is closed relative to the product topology on $\ns^{\db{n}}$, then we say that $\ns$ is a \emph{compact nilspace}.
\end{defn}
\noindent Next we recall the definition of morphisms in the nilspace category, and the special type of morphisms called fibrations (or fiber-surjective morphisms), which are nilspace analogues of surjective homomorphisms between abelian groups.
\begin{defn}[Morphisms and fibrations] 
Let $\ns,\nss$ be nilspaces. A map $\varphi:\ns\to\nss$ is a morphism if $\varphi\co\q\in\cu^n(\nss)$ for every $\q\in\cu^n(\ns)$. We denote the set of such morphisms by $\hom(\ns,\nss)$. We say that $\varphi\in \hom(\ns,\nss)$ is a \emph{fibration} if for every $\q'\in\cor^n(\ns)$, for every completion $\q\in \cu^n(\nss)$ of the corner $\varphi\co\q'$, there exists a completion $\tilde{\q}\in\cu^n(\ns)$ of $\q'$ such that  $\varphi\co\tilde{\q}=\q$. If $\ns$ and $\nss$ are compact nilspaces then morphisms in $\hom(\ns,\nss)$ are required to be also continuous maps.
\end{defn}
The next concept is crucial for the analysis of the structure of nilspaces.
\begin{defn}[Characteristic nilspace factors]\label{def:char-factors}
Let $\ns$ be a nilspace. For every $k\ge 0$ and $x,y\in\ns$ we write $x\sim_k y$ if there exist $\q_1,\q_2\in\cu^{k+1}(\ns)$ such that $\q_1(v)=\q_2(v)$ for all $v\not=1^{k+1}$, $\q_1(1^{k+1})=x$, and $\q_2(1^{k+1})=y$. We define the $k$-th \emph{characteristic factor} $\ns_k:=\ns/\sim_k$ and denote by $\pi_k$ the canonical projection $\ns\to \ns/\sim_k$. Then $\ns_k$ equipped with the cube sets $\cu^n(\ns_k):=\{\pi_k\co\q:\q\in \cu^n(\ns)\}$ is a $k$-step nilspace. If $\ns$ is a compact nilspace then $\ns_k$ is also a compact nilspace and the map $\pi_k$ is continuous and open.
\end{defn}
\noindent For the proofs of the claims in this definition, see \cite[\S 3.2]{Cand:Notes1} and \cite[\S 2.1]{Cand:Notes2}. These characteristic nilspace factors enable us to view nilspaces as iterated abelian bundles, in the sense of the following construction (see \cite[Definition 3.2.17]{Cand:Notes1}).
\begin{defn}[Abelian bundle] 
Let $\ab$ be an abelian group and $S$ and $B$ be sets. We say that $B$ is a \emph{$\ab$-bundle over} $S$ if there exists an action $\alpha:B\times \ab\to B$, $(b,z)\mapsto b+z$ and a (projection) map $\pi:B\to S$ such that the following holds.
\begin{enumerate}
    \item The action of $\ab$ is free, i.e., for any $b\in B$ we have $\{z\in \ab:b+z=b\}=\{0_{\ab}\}$.
    \item The map $s\mapsto\pi^{-1}(s)$ is a bijection between $S$ and the set of orbits of $\ab$ in $B$.
\end{enumerate}
\noindent For any integer $k\ge 0$, we say that $B$ is a \emph{$k$-fold abelian bundle} if there exists a sequence of sets $B_0,\ldots,B_k$ with $B_k=B$, and abelian groups $\ab_1,\ldots,\ab_k$, such that $B_0$ is a singleton and $B_{i}$ is a $\ab_i$-bundle over $B_{i-1}$ for all $i\in [k]$. Denoting by $\pi_{i-1,i}$ the projection $B_i\to B_{i-1}$, for each $i\in [k]$, we then denote by $\pi_{i,j}$ the projection $\pi_{i,i+1}\co\cdots\co \pi_{j-1,j}:B_j\to B_i$ for any $i\leq j$ in $[k]$, and we define $\pi_i:=\pi_{i,k}$. We call the bundles $B_i$ the \emph{factors} of the bundle $B$.

If $\ab$ is a compact abelian group and $B$, $S$ and $\ab$ are compact spaces, we say that $B$ is a \emph{compact $\ab$-bundle} over $S$ if in addition to the previous assumptions we have that $\alpha$ is continuous and $U\subset S$ is open if and only if $\pi^{-1}(U)\subset B$ is open. We say that $B$ is a \emph{compact $k$-fold abelian bundle} if for every $i\in[k]$ the factor $B_i$ is a compact $\ab_i$-abelian bundle over $B_{i-1}$.
\end{defn}
\begin{lemma}[Compact nilspaces as abelian bundles]
Let $\ns$ be a $k$-step compact nilspace. Then $\ns$ along with the maps $\pi_{i,j}:\ns_j\to\ns_i$ for $i\le j$ is a $k$-fold compact abelian bundle. The compact abelian groups $\ab_1,\ldots,\ab_k$ such that $\ns_i$ is a $\ab_i$-bundle over $\ns_{i-1}$ for each $i\in[k]$ are called the \emph{structure groups} of $\ns$.
\end{lemma}
\begin{proof} 
See \cite[Section 3.2.3]{Cand:Notes1} and \cite[Section 2.1]{Cand:Notes2}.
\end{proof}
\noindent Note that each cube-set $\cu^n(\ns)$ can also be seen as a $k$-fold compact abelian bundle. To detail this we recall from \cite[Definition 2.2.30]{Cand:Notes1} that for any abelian group $\ab$ and $k\in \mb{N}$, the $k$-step nilspace $\mc{D}_k(\ab)$ is $\ab$ together with the cube sets
\[
\cu^n(\mc{D}_k(\ab))=\{\q:\db{n}\to\ab ~| \textrm{ for every face map }\phi:\db{k+1}\to\db{n}, \sigma_{k+1}(\q\co\phi)=0\},
\] 
where recall that for any $f:\db{k}\to\mb{C}$ we define  $\sigma_k(f)=\sum_{v\in \db{k}} (-1)^{v\sbr{1}+\cdots+v\sbr{k}}f(v)$.
\begin{lemma} 
Let $\ns$ be a $k$-step compact nilspace with structure groups $\ab_1,\ldots,\ab_k$ and let $n\in \mb{N}$. Then $\cu^n(\ns)$ is a $k$-fold compact abelian bundle, with factors $\cu^n(\ns_i)$, factor maps $\pi_{i,j}^{\db{n}}:\cu^n(\ns_j)\to\cu^n(\ns_i)$, and structure groups $\cu^n(\mc{D}_1(\ab_1)),\ldots,\cu^n(\mc{D}_k(\ab_k))$.
\end{lemma}
\noindent Every iterated compact abelian bundle $B$ can be endowed with a Borel probability measure, called the \emph{Haar measure} on $B$, that generalizes the Haar measure on compact abelian groups. In particular, given a $k$-step compact nilspace $\ns$, the Haar measure on $\cu^n(\ns_i)$ can be defined as follows (for more details we refer to \cite[\S 2.2.2]{Cand:Notes2}).

\begin{proposition}
Let $\ns$ be a compact $k$-step nilspace. Then for any $n\ge 0$ and $i\in[k]$ there exists a unique Borel probability measure $\mu_{\cu^n(\ns_i)}$ on $\cu^n(\ns_i)$ with the following properties: we have $\mu_{\cu^n(\ns_i)}=\mu_{\cu^n(\ns)}\co (\pi_i^{\db{n}})^{-1}$ and $\mu_{\cu^n(\ns_i)}$ is invariant under the action of $\cu^n(\mc{D}_i(\ab_i(\ns)))$.
\end{proposition}

For the proof we refer to \cite[Proposition 2.2.5]{Cand:Notes2} and \cite[Proposition 3.6]{CScouplings}.

\begin{remark}\label{rem:ext-haar-x-db-n}
We can define a measure $\mu$ on $\ns^{\db{n}}$ concentrated on $\cu^n(\ns)$ by setting  $\mu(A):=\mu_{\cu^n(\ns)}(A\cap \cu^n(\ns))$ for any Borel set $A\subset \ns^{\db{n}}$. We shall sometimes abuse the notation by denoting $\mu$ also by $\mu_{\cu^n(\ns)}$ (i.e.\ viewing $\mu_{\cu^n(\ns)}$ as a measure on $\ns^{\db{n}}$).
\end{remark}
\noindent Finally, let us recall the definition of a specific class of nilspaces which will be used extensively in this paper. Here and throughout the sequel, the group $\mb{Z}_p^n$ will be identified with $[0,p-1]^n$ equipped with addition mod $p$, the usual way.
\begin{defn}[$p$-homogeneous nilspaces \cite{CGSS-p-hom}]\label{def:p-hom}
Let $\ns$ be a nilspace and let $p\in \mb{N}$ be a prime. We say that $\ns$ is a \emph{$p$-homogeneous nilspace} if for every positive integer $n$, for every $f\in \hom(\mc{D}_1(\mb{Z}^n),\ns)$ the restriction $f|_{[0,p-1]^n}$ is in $\hom(\mc{D}_1(\mb{Z}_p^n),\ns)$.
\end{defn}
\noindent We recall more background on these nilspaces in Section \ref{sec:p-hom}. For now let us mention that the most basic examples of $p$-homogeneous nilspaces are given by the elementary abelian $p$-groups $\mb{Z}_p^n$, and that more examples can be constructed easily (see \cite{CGSS-p-hom}).  These include the nilspace $\mc{H}$ introduced in Definition \ref{def:H}, which is $2$-homogeneous. This nilspace is also an example of a compact nilspace that is not of finite step, but rather of what we shall call \emph{profinite-step}. This is a useful type of infinite-step nilspaces, to which we now turn.

\subsection{Profinite-step nilspaces}\label{subsec:profin}\hfill\\
In the sequel we often have to deal with compact nilspaces that are not necessarily of finite step. However, we shall always be able to assume that these nilspaces have the following useful property.
\begin{defn}
We say that a compact nilspace $\ns$ is a \emph{profinite-step} nilspace if it is the inverse limit of compact nilspaces of finite step.
\end{defn}
\noindent See \cite[Section 2.7]{Cand:Notes2} for the definition of inverse limits of compact nilspaces; from this and Definition \ref{def:char-factors}, it is not hard to see that a compact nilspace is profinite-step if and only if it is the inverse limit of its characteristic factors.

Every finite-step compact nilspace is trivially profinite-step, but not all compact nilspaces are profinite-step. For instance, given a filtered group $(G,G_\bullet)$, the associated group nilspace is profinite-step if and only if the filtration $G_\bullet$ has the following property.
\begin{defn}\label{def:non-deg-filt}
We say that a filtration $G_\bullet=(G_{(i)})_{i\geq 0}$ on a group $G$ is \emph{non-degenerate} if for any pair of distinct elements $g, g'\in G$  there exists $i\in \mb{N}$ such that $g\neq g'\!\mod G_{(i)}$. Equivalently, the filtration $G_\bullet$ is non-degenerate if $\bigcap_{i=0}^{\infty} G_{\sbr{i}} = \{\id_G\}$.
\end{defn}
\noindent Every profinite-step nilspace can be endowed with a unique Borel probability measure, which we shall call its \emph{Haar measure}, such that for each $k\in \mb{N}$ the pushforward of this measure to the  $k$-th characteristic factor equals the Haar measure on this factor. In order to establish this rigorously, we need to extend the definition of $k$-fold compact abelian bundles, to define $\infty$-fold compact abelian bundles. Let us state the main result here and defer the technical (but relatively routine) proofs to Appendix \ref{app:cubic-coupling-infinite-step}.

\begin{proposition}[$n$-cubic Haar measures on compact profinite-step nilspaces]\label{prop:haar-mes}
Let $\ns$ be a compact profinite-step nilspace. Then for every $n\ge 0$ there exists a unique measure $\mu_{\cu^n(\ns)}$ on $\cu^n(\ns)$ such that for every $i\in\mb{N}$ we have $\mu_{\cu^n(\ns)} \co (\pi_i^{\db{n}})^{-1} = \mu_{\cu^n(\ns_i)}$. We say that $\mu_{\cu^n(\ns)}$ is the \emph{$n$-cubic Haar measure} on $\cu^n(\ns)$ for every $n\ge 0$.
\end{proposition}
\begin{proof} 
We combine Lemma \ref{lem:haar-mes-infty-fold} with Lemma \ref{lem:hom-spaces-infty-fold-ab-bundles}, applying the latter with $P=\db{n}$ and $S=\emptyset$.
\end{proof}
\noindent Finally, we need to establish that a fibration between profinite-step compact nilspaces always preserves the Haar measures.
\begin{lemma}\label{prop:char-inv-nil}
Let $\ns$, $\nss$ be compact profinite-step nilspaces, and let $\varphi:\ns\to \nss$ be a fibration. Then, for every $n\ge 0$, the map $\varphi^{\db{n}}:\cu^n(\ns)\to \cu^n(\nss)$, $\q\mapsto \varphi\co\q$ preserves the $n$-cubic Haar measures, i.e.\ $\mu_{\cu^n(\ns)} \co (\varphi^{\db{n}})^{-1} = \mu_{\cu^n(\nss)}$.
\end{lemma}

\begin{proof} 
A fibration between nilspaces induces a totally-surjective bundle morphism between the corresponding sets of cubes \cite{GMV1}, \cite[Lemma 3.3.12]{Cand:Notes1}, and it therefore preserves the Haar measures as claimed, by Lemma \ref{lem:mes-pre-1}.
\end{proof}

\subsection{The infinite-dimensional cube set \texorpdfstring{$\cu^\omega(\ns)$}{} and its Haar measure}\label{subsec:CinfHaar}\hfill\smallskip\\
\noindent As mentioned in the introduction, an important object in this paper is the set of infinite-dimensional cubes on a nilspace. Let us define this structure formally.
\begin{defn}\label{def:omega-cubes}
Let $\ns$ be a nilspace. For each $n\in \mb{N}$ let $\phi_n:\db{n}\to\db{\mb{N}}$,  $v\mapsto (v, 0^{\mb{N}\setminus[n]})=(v\sbr{1},\ldots,v\sbr{n},0,\ldots)$, and let $p_n:\ns^{\db{\mb{N}}}\to \ns^{\db{n}}$ be the projection induced by $\phi_n$, i.e.\ the map $(x\sbr{v})_{v\in\db{\mb{N}}}\mapsto (x\sbr{\phi_n(v)})_{v\in\db{n}}$. We then define 
\begin{equation}\label{eq:DefCinf}
\cu^\omega(\ns):=\bigcap_{n=1}^\infty p_n^{-1}(\cu^n(\ns)).
\end{equation}
If $\ns$ is a compact nilspace, then $\cu^\omega(\ns)$ is a compact subset of $\ns^{\db{\mb{N}}}$ in the product topology. (Note that $\cu^\omega(\ns)$ is non-empty, containing  in particular all constant maps $\db{\mb{N}}\to\ns$.)
\end{defn}
\noindent Recall that a map $\phi:\db{n}\to\db{\mb{N}}$ is a (discrete cube) morphism if it extends to an affine homomorphism $\mb{Z}^n\to \bigoplus_{i\in \mb{N}}\mb{Z}$ (see \cite[Definition 6.1]{CScouplings}). Another natural way to think of $\cu^\omega(\ns)$ is that its elements are precisely the functions $\q:\db{\mb{N}}\to\ns$ such that for every morphism $\phi:\db{n}\to\db{\mb{N}}$ (for every $n$) we have $\q\co\phi\in\cu^n(\ns)$.

\begin{example}\label{ex:omega-cubes-d-k} 
Let $\ab$ be an abelian group and let $k\in \mb{N}$. Then
\[
\cu^\omega(\mc{D}_k(\ab))= \Big\{\q:v\mapsto \sum_{i=0}^k\sum_{S=\{s_1,\ldots,s_i\}\in \binom{\mb{N}}{i}} z_{i,S}v\sbr{s_1}\cdots v\sbr{s_i} ~\Big|~ z_{i,S}\in\ab\Big\},
\]
where note that the sum over $S$ here is always well-defined since $v\in \db{\mb{N}}$ has only finitely many non-zero coordinates. The proof of this equality follows from applying \cite[Lemma 2.2.5]{Cand:Notes1} to $\q\co\phi_n$ for each $n\in \mb{N}$, and an induction using that the sets $D_n:=\db{n}\times \{0^{\mb{N}\setminus[n]} \}\subset\db{\mb{N}}$ form an increasing sequence with $\bigcup_{n=1}^\infty D_n = \db{\mb{N}}$. We omit the details.
\end{example}

\begin{defn}
Let $\ns$ be a nilspace. An \emph{$\omega$-corner} on $\ns$ (\emph{rooted} at some vertex $v\in \db{\mb{N}}$) is a map $\q':\db{\mb{N}}\setminus\{v\}\to \ns$ such that for every face map\footnote{A face map $\phi:\db{n}\to \db{\mb{N}}$ is an injective morphism such that for some $S\in \binom{\mb{N}}{n}$, the coordinate $\phi(v)\sbr{i}$ is a constant function of $v$ for every $i\not\in S$.} $\phi:\db{n}\to \db{\mb{N}}$ with $\phi(\db{n})\not\ni v$, we have $\q'\co \phi \in \cu^n(\ns)$. We denote the set of such $\omega$-corners rooted at $v$ by $\cor_v^\omega(\ns)$.
\end{defn}
\noindent A useful fact about $\cu^\omega(\ns)$ is that if $\ns$ is profinite-step then this infinite-dimensional cube set satisfies the following form of unique corner-completion.
\begin{lemma}\label{lem:uniqueomegacomp}
Let $\ns$ be a compact profinite-step nilspace. Then every $\omega$-corner on $\ns$ has a unique completion in $\cu^\omega(\ns)$.
\end{lemma}
\begin{proof}
Let $\q':\db{\mb{N}}\setminus\{v\}\to \ns$ be an $\omega$-corner. We need to prove the existence and uniqueness of some $\q\in \cu^{\omega}(\ns)$ such that $\q(w)=\q'(w)$ for every $w\neq v$. Composing with a suitable morphism $\phi:\db{\mb{N}}\to\db{\mb{N}}$ if necessary, we can assume without loss of generality that $v=0^{\mb{N}}$ (if the non-zero coordinates of $v$ are indexed by some finite $V\subset \mb{N}$, we can set $\phi(w):=w\sbr{i}$ if $i\notin V$ and $\phi(w):=1-w\sbr{i}$ otherwise).

Recall that $\ns = \varprojlim \ns_k$ where $\ns_k$ are the characteristic nilspace factors of $\ns$. For every $k$ and any $(k+1)$-corner $q$ on $\ns_k$ rooted at $0^{k+1}$, let $\overline{q}\in \cu^{k+1}(\ns_k)$ be its unique completion. Note that $\pi_k\co \q'|_{\db{k+1}}$ is a $(k+1)$-corner on $\ns_k$ rooted at $0^{k+1}$. For every $k$ let $C_k:=\{q'\in \cu^{k+1}(\ns): \pi_k\co q' = \overline{\pi_k\co \q'|_{\db{k+1}}}\}$. We claim that this is a non-empty compact set. The fact that $C_k\neq \emptyset$ follows from the fact that $\pi_k:\ns\to\ns_k$ is a fibration. To see that $C_k$ is compact, note that $C_k=\cu^{k+1}(\ns)\cap [(\pi_k^{\db{k+1}})^{-1}(\{\overline{\pi_k\co\q'|_{\db{k+1}}}\})]$. It follows that the set $D_k:=\{\q''(0^{k+1})\in \ns:\q''\in C_k\}$ is a non-empty compact set for every $k$.

The sets $D_k$ form a decreasing sequence relative to inclusion, so by compactness of $\ns$ (the finite intersection property) we have $\bigcap_{k=1}^{\infty}D_k\not=\emptyset$. Letting $x$ be any point in $\bigcap_{k=1}^{\infty}D_k$, we see by construction that $x$ completes the corner $\q'$.

To prove uniqueness, let $x,y\in \ns$ be two points completing $\q'$, and denote the corresponding completions in $\cu^\omega(\ns)$ by $\q_x$ and $\q_y$. For every $k\ge 1$ we have that $\pi_k\co\q_x|_{\db{k+1}}$ and $\pi_k\co \q_y|_{\db{k+1}}$ are cubes in $\cu^{k+1}(\ns_k)$ with equal values at every vertex $w\in \db{k+1}\setminus\{0^{k+1}\}$. By uniqueness of corner-completion  in $\ns_k$, we have $\pi_k(x)=\pi_k(y)$. As this holds for every $k\ge 1$ and $\ns$ is profinite-step, it follows that $x=y$.
\end{proof}

\begin{lemma}\label{lem:Cinfbund} 
Let $\ns$ be a compact profinite-step nilspace and for each $k\in\mb{N}$ let $\ab_k$ be the $k$-th structure group of $\ns$. Then $\cu^\omega(\ns)$ is an $\infty$-fold compact abelian bundle with structure groups $\cu^\omega(\mc{D}_k(\ab_k))$ and projections $\pi_{k,i}^{\db{\mb{N}}}:\cu^\omega(\ns_i)\to \cu^\omega(\ns_k)$ for $0\le k\le i$.
\end{lemma}
\begin{proof} 
The first thing that we need to check is that $\cu^\omega(\ns_k)$ is a $\cu^\omega(\mc{D}_k(\ab_k))$-compact abelian bundle over $\cu^\omega(\ns_{k-1})$ for every $k\in \mb{N}$. Thus we need to prove Properties 1 and 2 of \cite[Definition 3.2.17]{Cand:Notes1}. To prove Property 1, take any $\q\in \cu^\omega(\ns_k)$ and $f\in \cu^\omega(\mc{D}_k(\ab_k))$. If $\q+f=\q$ then coordinate-wise we have $\q\sbr{v}+f\sbr{v}=\q\sbr{v}$ for any $v\in \db{\mb{N}}$. But the action of $\ab_k$ on $\ns_k$ is free by \cite[Theorem 3.2.19]{Cand:Notes1} so $f=0$. For Property 2, let $\q\in \cu^\omega(\ns_k)$. We need to show that if $\q'\in \cu^\omega(\ns_k)$ is any other element such that $\pi_{k-1,k}^{\db{\mb{N}}}\co \q = \pi_{k-1,k}^{\db{\mb{N}}}\co \q'$ then there exists $f\in \cu^\omega(\mc{D}_k(\ab_k))$ such that $\q=\q'+f$. Recall from Definition \ref{def:omega-cubes} the map $\phi_n:\db{n}\to\db{\mb{N}}$. Note that for any fixed $n\in\mb{N}$ we have $\pi_{k-1,k}\co\q\co\phi_n = \pi_{k-1,k}\co\q'\co\phi_n$ and thus, by \cite[Theorem 3.2.19]{Cand:Notes1}, we have $f_n:=\q\co\phi_n-\q'\co\phi_n\in \cu^n(\mc{D}_k(\ab_k))$. By definition the maps $(f_n)_{n\ge 1}$ are consistent (identifying $\db{1}\subset \db{2}\subset \cdots \subset \db{\mb{N}}$ according to $\phi_1,\phi_2,\ldots$) and thus if we define $f:\db{\mb{N}}\to \ab_k$ as $v\mapsto f_n\sbr{v}$ where $v\in \db{n}\times \{0^{\mb{N}\setminus[n]}\}$ (for $n$ large enough) we have $f\in \cu^\omega(\mc{D}_k(\ab_k))$ and $\q=\q'+f$. 

Next we need to check the conditions of \cite[Definition 2.1.6]{Cand:Notes2}. Recall that on $\cu^\omega(\ns_k)\subset \ns_k^{\db{\mb{N}}}$ we are using the product topology, and by \cite[Proposition 2.1.9]{Cand:Notes2}, the analogous conditions hold for $\cu^n(\ns_k)$ instead of $\cu^\omega(\ns_k)$. Then conditions $(i)$ and $(ii)$ of \cite[Definition 2.1.6]{Cand:Notes2} follow directly from the definitions of $\cu^\omega(\ns_k)$ and $\cu^\omega(\mc{D}_k(\ab_k))$. The action of $\cu^\omega(\mc{D}_k(\ab_k))$ on $\cu^\omega(\ns_k)$ is continuous as it is coordinate-wise continuous and we are using the product topology on $\cu^\omega(\ns_k)\subset \ns_k^{\db{\mb{N}}}$. Hence condition $(iii)$ follows. In order to prove $(iv)$ it suffices to check that $\pi_{k-1,k}^{\db{\mb{N}}}$ is open and continuous. The continuity follows from the continuity of $\pi_{k-1,k}$ and the definition of the product topology on $\cu^\omega(\ns_k)\subset \ns_k^{\db{\mb{N}}}$. To prove that it is an open map, note that it suffices to prove that images under $\pi_{k-1,k}^{\db{\mb{N}}}$ of subsets of the form $(\prod_{v\in \db{n}} U_v)\times \ns_k^{\db{\mb{N}}\setminus\db{n}}$ are open (where $U_v\subset \ns_k$ are open), as these product sets form a base of the product topology. But for these product sets the openness is already known from the theory for finite-dimensional cubes (see \cite[Lemma 2.1.10]{Cand:Notes2}).

Finally, we need to show that $\cu^\omega(\ns)=\varprojlim \cu^\omega(\ns_k)$ (as per Definitions \ref{def:infty-fold-ab-bundle} and \ref{def:infty-fold-compact-ab-bundle}). Since $\ns$ is profinite-step we have $\ns = \varprojlim \ns_k$ with limit maps (fibrations) $\pi_k:\ns\to\ns_k$. For any $\q\in \cu^\omega(\ns)$, for each integer $k\geq 0$ let $\q_k:=\pi_k\co \q:\db{\mb{N}}\to\ns_k$. We need to show that $\q_k\in \cu^\omega(\ns_k)$ and that $\pi_{k-1,k}\co \q_k=\q_{k-1}$, for each $k\in \mb{N}$. The latter equality follows from the fact that $\ns$ is the inverse limit of the $\ns_k$ (which gives the equality coordinate-wise). To see the former claim, note that for each $n\in \mb{N}$ we have $(\pi^{\db{\mb{N}}}_k(\q))\co \phi_n = \pi_k^{\db{n}}(\q\co\phi_n)$, and this last map is in $\cu^n(\ns_k)$ since $\q\co\phi_n\in \cu^n(\ns)$ and $\pi_k$ is a morphism, so $(\pi^{\db{\mb{N}}}_k(\q))\co \phi_n$ is an $n$-cube on $\ns_k$ for every $n$, which means that $\pi^{\db{\mb{N}}}_k(\q)\in  \cu^\omega(\ns_k)$ as required. This proves the inclusion of $\cu^\omega(\ns)$ in $\varprojlim \cu^\omega(\ns_k)$, and the opposite inclusion is similar.
\end{proof}
\begin{corollary}\label{cor:haar-mes-omega-cube}
Let $\ns$ be a profinite-step compact nilspace. Then there exists a unique Borel probability measure on $\cu^\omega(\ns)$, which we call the \emph{Haar measure} and denote by $\mu_{\cu^\omega(\ns)}$, determined by the property that for each integer $k\geq 0$ we have $\mu_{\cu^\omega(\ns)}\co(\pi_k^{\db{\mb{N}}})^{-1}=\mu_{\cu^\omega(\ns_k)}$.
\end{corollary}
\begin{proof}
This follows by combining Lemma \ref{lem:Cinfbund} with Lemma \ref{lem:haar-mes-infty-fold}.
\end{proof}
\noindent As in Remark \ref{rem:ext-haar-x-db-n}, let us mention similarly here that we shall sometimes abuse the notation and view $\mu_{\cu^\omega(\ns)}$ as a measure on $\ns^{\db{\mb{N}}}$ (concentrated on $\cu^\omega(\ns)$).
\begin{remark}\label{rem:co-omega-equals-cor-omega}
Note that if $\ns$ is a compact profinite-step nilspace, the set $\cor_v^\omega(\ns)$ of $\omega$-corners rooted at $v$  (for any fixed $v\in \db{\mb{N}}$) is also an $\infty$-fold compact abelian bundle and thus it has a Haar measure. Furthermore, by Lemma \ref{lem:uniqueomegacomp} the map $\psi:\cu^\omega(\ns)\to\cor_v^\omega(\ns)$, $\q\mapsto \q|_{\db{\mb{N}}\setminus\{v\}}$ is a bijective bicontinuous bundle-morphism, so $\psi$ and $\psi^{-1}$ both preserve the Haar measures.
\end{remark}
\begin{proposition}\label{prop:mes-pre-inf-cube}
Let $\ns$ and $\nss$ be compact profinite-step nilspaces. Let $\varphi:\nss\to \ns$ be a fibration. Then $\varphi^{\db{\mb{N}}}:\cu^\omega(\nss)\to \cu^\omega(\ns)$ preserves the Haar measures. 
\end{proposition}
\begin{proof} 
It suffices to check that $\varphi^{\db{\mb{N}}}$ is a totally-surjective bundle morphism between the corresponding $\infty$-fold compact abelian bundles. As in the proof of Lemma \ref{lem:Cinfbund}, the arguments are routine analogues of existing ones in the literature and we shall not detail them beyond the following outline. First one needs to check the conditions $(i)$ and $(ii)$ of Definition \ref{def:infty-fold-bundle-morphism}. Condition $(i)$ follows from the fact that given $\q,\q'\in\cu^\omega(\ns)$ such that $\pi_k^{\db{\mb{N}}}(\q)=\pi_k^{\db{\mb{N}}}(\q')$ we have that for each fixed $v\in \db{\mb{N}}$, $\pi_k(\q\sbr{v})=(\pi_k^{\db{\mb{N}}}(\q))\sbr{v}=(\pi_k^{\db{\mb{N}}}(\q'))\sbr{v}=\pi_k(\q'\sbr{v})$. Then as $\varphi$ is a bundle morphism by \cite[Proposition 3.3.2]{Cand:Notes1} we have $\pi_k(\varphi(\q\sbr{v}))=\pi_k(\varphi(\q'\sbr{v}))$. Hence $\pi_k^{\db{\mb{N}}}(\varphi^{\db{\mb{N}}}(\q))=\pi_k^{\db{\mb{N}}}(\varphi^{\db{\mb{N}}}(\q'))$. Condition $(ii)$ follows similarly.

Continuity follows from that of $\varphi$. To prove the surjectivity between the $i$-factors, let $f\in \cu^\omega(\mc{D}_i(\ab_i))$ where $\ab_i$ is the $i$-th structure group of $\nss$ and $\q\in\cu^\omega(\nss)$. It can be checked that $\varphi^{\db{\mb{N}}}(\q+f)=\varphi^{\db{\mb{N}}}(\q)+\alpha_i^{\db{\mb{N}}}(f)$ where $\alpha_i:\ab_i\to\ab_i'$ is the $i$-structure homomorphism of $\varphi$ and $\ab_i'$ the $i$-th structure group of $\ns$. It is readily seen using Example \ref{ex:omega-cubes-d-k} that the map $\alpha_i^{\db{\mb{N}}}:\cu^\omega(\mc{D}_i(\ab_i))\to \cu^\omega(\mc{D}_i(\ab'_i))$ is surjective.
\end{proof}
\noindent Let us record one last feature of the cube set $\cu^\omega(\ns)$ in this subsection. If $\ns$ is the group nilspace associated with a compact filtered abelian group $(\ab,\ab_\bullet=(\ab_{\sbr{k}})_{k\geq 0})$, then $\cu^\omega(\ns)$ is a compact abelian subgroup of $\ab^{\db{\mb{N}}}$. Indeed this follows from the definition \eqref{eq:DefCinf} and the standard fact that the Host--Kra cube-set $\cu^n(\ns)$ is a closed subgroup of $\ab^{\db{n}}$. We then have the following fact.
\begin{lemma}\label{lem:nsabHaarequiv}
Let $\ns$ be a compact profinite-step nilspace, and suppose that $\ns$ is the group nilspace associated with a filtered compact abelian group $(\ab,\ab_\bullet)$ \textup{(}thus the filtration $\ab_\bullet$ is non-degenerate\textup{)}. Then the Haar measure on $\cu^\omega(\ns)$ defined in Corollary \ref{cor:haar-mes-omega-cube} equals the Haar measure of $\cu^\omega(\ns)$ as a compact abelian group.
\end{lemma}

\begin{proof}
By definition and uniqueness of the (group) Haar measure on $\cu^\omega(\ns)$, it suffices to prove that the measure $\mu_{\cu^\omega(\ns)}$ given by Corollary \ref{cor:haar-mes-omega-cube} is invariant under the addition of any $\q'\in\cu^\omega(\ns)$.

By Lemma \ref{lem:mes-pre-1} it is enough to prove that for any $\q'\in \cu^\omega(\ns)$, the map $\varphi_{\q'}:\cu^\omega(\ns)\to\cu^\omega(\ns)$, $\q\mapsto \q+\q'$ is a totally surjective continuous bundle morphism. The continuity follows from continuity of addition in $\ab$. Recall that in the case of group nilspaces, each projection $\pi_k:\ns\to\ns_k$ equals the quotient map $\pi_k:\ab\to \ab/\ab_{\sbr{k+1}}$, and $\ns_k$ is the group nilspace associated with $(\ab/\ab_{\sbr{k+1}},(\ab_{\sbr{i}}/\ab_{\sbr{k+1}})_{i=0}^\infty)$ (this follows from the definition of the characteristic factors and the definition of the Host--Kra cubes). 

Clearly the map $\varphi_{\q'}$ is a bundle morphism, since $\pi_k$ is also a homomorphism of abelian groups. Hence we have to check that the structure homomorphism of $(\varphi_{\q'})_k:\cu^\omega(\ns_k)\to \cu^\omega(\ns_k)$, $\q\!\! \mod \ab_{\sbr{k+1}}\mapsto \q+\q'\!\!\mod \ab_{\sbr{k+1}}$ is surjective for every $k\ge 1$. The $k$-th structure group of $\ns$ can be proved to be $\ab_{\sbr{k}}/\ab_{\sbr{k+1}}$ (\cite[Corollary 3.2.16]{Cand:Notes1}). Given $f\in \cu^\omega(\mc{D}_k(\ab_{\sbr{k}}/\ab_{\sbr{k+1}}))$ we have that $(\varphi_{\q'})_k(\q+f\mod \ab_{\sbr{k+1}}) = \q+\q'+f\mod \ab_{\sbr{k+1}}= (\varphi_{\q'})_k(\q\mod \ab_{\sbr{k+1}})+f\mod \ab_{\sbr{k+1}}$. Thus the corresponding structure homomorphism is just the identity, which is surjective as required.
\end{proof}

\begin{remark} 
In the sequel we will frequently work with group nilspaces $\ns$ associated with filtered compact abelian groups $(\ab,\ab_\bullet)$. In these cases, when there is no risk of confusion, we will say that $\ab$ is a nilspace (when the filtration is understood from the context), and its set of $n$-cubes for $n\in \mb{N}\cup\{\omega\}$ will be denoted by $\cu^n(\ab)$ (or by $\cu^n(\ab_\bullet)$ if the filtration needs to be emphasized). By Lemma \ref{lem:nsabHaarequiv}, we can use $\mu_{\cu^\omega(\ab)}$ as the (unique) Haar measure, from either of the nilspace or group viewpoints, without ambiguity. Note also that using an argument very similar to the previous one, the (group) Haar measure of $\cu^n(\ns)$ equals the (nilspace) Haar measure of $\cu^n(\ns)$ for any $n\in \mb{N}$. 
\end{remark}

\subsection{More background on cubic and affine-exchangeability}\label{subsec:basicex}\hfill\smallskip\\
\noindent In this subsection we gather results providing more detailed background on affine and cubic exchangeability, elaborating on remarks and examples from the introduction. 

We begin by discussing in more detail the inclusions in \eqref{eq:inclu}.

\begin{lemma}\label{lem:clarifinclu}
For every standard Borel space $\Bo$, we have
\begin{equation}\label{eq:inclu2}
\textup{Pr}^{\textup{Aff}(\mb{F}_2^\omega)}(\Bo^{\db{\mb{N}}})  
\;\; \subset \;\; \{\textup{cubic-exchangeable measures on }\Bo^{\db{\mb{N}}}\} 
\;\;  \subset \;\; \textup{Pr}^{\aut(\db{\mb{N}})}(\Bo^{\db{\mb{N}}}).
\end{equation}
\end{lemma}
\begin{proof}
To see the first inclusion, note that given any pair of injective morphisms $\phi_1,\phi_2:\db{k}\to\db{\mb{N}}$, their images define two subspaces of dimension $k$ (identifying $\db{\mb{N}}$ with $\mb{F}_2^{\omega}$). Thus,  there exists an affine map $\alpha\in \textup{Aff}(\mb{F}_2^\omega)\cong  \textrm{GL}(\mb{F}_2^\omega)\ltimes \mb{Z}_2^\omega$ such that $\phi_2 = \alpha\co \phi_1$.  
For $i=1,2$ let $p_{\phi_i}:\Bo^{\db{\mb{N}}}\to \Bo^{\db{k}}$ be the projection $(b_v)_{v\in\db{\mb{N}}}\mapsto (b_{\phi_i(v)})_{v\in\db{k}}$, and note that  $p_{\phi_2} = p_{\phi_1} \co {\alpha'}^{-1}$ where ${\alpha'}^{-1}:\Bo^{\db{\mb{N}}}\to \Bo^{\db{\mb{N}}}$, $(b_v)_{v\in\db{\mb{N}}}\mapsto (b_{\alpha^{-1}(v)})_{v\in \db{\mb{N}}}$. Now, given any $\mu\in \textup{Pr}^{\textup{Aff}(\mb{F}_2^\omega)}(\Bo^{\db{\mb{N}}})$, we have $\mu \co p_{\phi_2}^{-1}  = \mu \co \alpha'\co  p_{\phi_1}^{-1} = \mu \co  p_{\phi_1}^{-1}$, where the last equality here uses that $\mu$ is affine-exchangeable. Hence $\mu$ is cubic-exchangeable, which proves the desired inclusion.

To see that this first inclusion can be strict, it suffices to produce a cubic-exchangeable measure that is not affine-exchangeable. For example, if $\ns$ is the group nilspace associated with a compact filtered abelian group $(\ab,\ab_\bullet)$, and we set $\Bo=\ns$, then the measure $\mu_{\cu^\omega(\ns)}$ (seen as a measure on $\ns^{\db{\mb{N}}}$) is cubic-exchangeable, but by Corollary \ref{cor:2-hom-aff-exch}, in order for this measure to be affine-exchangeable, the nilspace $\ns$ must be 2-homogeneous. A simple example of a compact nilspace that is not 2-homogeneous is the nilspace $\mc{D}_1(\mb{Z}_3)$.

The second inclusion in \eqref{eq:inclu2} is proved in \cite[Remark 6.6]{CScouplings}. This inclusion is strict for instance whenever $\Bo$ is a compact metric space containing at least two elements. To see this, one can use the fact that, on one hand, the convex set of cubic-exchangeable measures on $\Bo^{\db{\mb{N}}}$ is a Bauer simplex relative to the vague topology, i.e.\ its extreme points form a vaguely closed set (this follows from results in \cite{CScouplings} and is recalled in detail in Section \ref{sec:geom} below), and on the other hand, it is proved in \cite[Theorem 4.2]{Austin2} that $\textup{Pr}^{\aut(\db{\mb{N}})}(\Bo^{\db{\mb{N}}})$ is a Poulsen simplex, so its set of extreme points need not be closed.
\end{proof}
\noindent Given a nilspace $\ns$, the next result gives a criterion involving $\cu^\omega(\ns)$ to decide whether $\ns$ is 2-homogeneous. Let $\phi:\db{\mb{N}}\to\mb{F}_2^\omega$ be the bijective map $v\mapsto v\mod 2$ (i.e.\ $\phi$ is just the natural identification of $\db{\mb{N}}$ as a set with $\mb{F}_2^\omega$). With this, we can view any function $\q:\db{\mb{N}}\to\ns$ in a simple way as a function  on $\mb{F}_2^\omega$ (just by considering $\q\co \phi^{-1}$), and vice versa. We shall sometimes switch between these two views of such a function $\q$, in a slight abuse of  notation. We then have the following result, which ensures in particular that the aforementioned abuse of notation will not be problematic when $\ns$ is 2-homogeneous.

\begin{lemma}\label{lem:2-hom-equals-c-omega-poly}
Let $\ns$ be a nilspace. Then $\cu^\omega(\ns)=\hom(\mc{D}_1(\mb{F}_2^\omega),\ns)$ if and only if $\ns$ is 2-homogeneous.
\end{lemma}

\begin{proof}
Suppose that $\ns$ is 2-homogeneous and take any $\q\in\cu^\omega(\ns)$. We have to prove  (via identifying $\db{\mb{N}}$ with $\mb{F}_2^\omega$) that $\q$ is an element of $\hom(\mc{D}_1(\mb{F}_2^\omega),\ns)$. Let $f\in \cu^k(\mc{D}_1(\mb{F}_2^\omega))$. Note that for $n=n(f)\in \mb{N}$ large enough we have $f(\db{k})\subset \mb{F}_2^n\times \{0^{\mb{N}\setminus[n]}\}$. In particular $\q\co f=\q|_{\db{n}\times 0^{\mb{N}\setminus[n]}}\co f$, where (abusing the notation) we can then consider that $f\in\cu^k(\mc{D}_1(\mb{F}_2^n))$, and that $\q|_{\db{n}\times 0^{\mb{N}\setminus[n]}}\in\cu^n(\ns)$ (by definition of $\cu^\omega(\ns)$). Then, since $\ns$ is 2-homogeneous, we have $\q|_{\db{n}\times 0^{\mb{N}\setminus[n]}}\in \hom(\mc{D}_1(\mb{F}_2^n),\ns)$, whence $\q\co f\in \cu^k(\ns)$. This proves that $\q\in \hom(\mc{D}_1(\mb{F}_2^\omega),\ns)$ as required.

For the opposite inclusion, let $\q'\in \hom(\mc{D}_1(\mb{F}_2^\omega),\ns)$, and note that any morphism $g:\db{k}\to \db{\mb{N}}$ can be viewed (via the above identification) as an element of $\cu^k(\mc{D}_1(\mb{F}_2^\omega))$ (for instance, by an explicit description that follows from the definition of $g$; see \cite[Remark 6.2]{CScouplings}). Thus, by the morphism property of $\q'$, we deduce that $\q'\co g\in \cu^k(\ns)$, so $\q'\in \cu^\omega(\ns)$ as required.

The converse is clear, since for any $n$ the cube set $\cu^n(\ns)$ is a projection of $\cu^\omega(\ns)$ and then the assumption $\cu^\omega(\ns)=\hom(\mc{D}_1(\mb{F}_2^\omega),\ns)$ implies (using the above identification) that $\cu^n(\ns)=\hom(\mc{D}_1(\mb{F}_2^n),\ns)$, so $\ns$ is 2-homogeneous. 
\end{proof}

\noindent Finally, let us elaborate on what is perhaps the simplest construction of cubic-exchangeable measures: we take any compact nilspace $\ns$, let $\mu_{\cu^\omega(\ns)}$ be the Haar measure on $\cu^\omega(\ns)$ constructed in Corollary \ref{cor:haar-mes-omega-cube}, and let $\mu$ be the Borel measure on $\ns^{\db{\mb{N}}}$ defined by 
\begin{equation}\label{eq:basicex}
\mu(A)=\mu_{\cu^\omega(\ns)}\big(A\cap \cu^\omega(\ns)\big), \textrm{ for any Borel set }A\subset \ns^{\db{\mb{N}}}.
\end{equation}
Note that this is a particularly simple instance of the construction in \eqref{eq:zeta}, indeed the  measure $\mu$ in \eqref{eq:basicex} is the special case of  \eqref{eq:zeta} where $\Bo=\ns$ and $m$ is the Borel function $\ns\to\mc{P}(\ns)$ sending each $x\in \ns$ to the Dirac measure $\delta_x$. 
The fact that this measure $\mu$ is cubic-exchangeable follows readily from the fact that, from construction of $\mu_{\cu^\omega(\ns)}$, we have that the pushforward of $\mu$ on any $n$-dimensional sub-cube of $\db{\mb{N}}$ (i.e.\ any image of an injective morphisms $\db{n}\to\db{\mb{N}}$) is always $\mu_{\cu^n(\ns)}$, for any $n$. It is then natural to wonder under what conditions this measure $\mu$ is not just cubic-exchangeable, but also \emph{affine-exchangeable}. The following last result of this section gives an answer which further motivates the use of 2-homogeneous nilspaces. 

\begin{lemma}\label{lem:aff-ech-equals-2-hom}
Let $\ns$ be a compact profinite-step nilspace. Then $\mu_{\cu^\omega(\ns)}$ \textup{(}viewed as a measure on $\ns^{\db{\mb{N}}}$\textup{)} is affine-exchangeable if and only if $\ns$ is 2-homogeneous.
\end{lemma}

\noindent  We defer the proof to  Appendix \ref{app:aff-exch-and-2-hom}.

There is a family of cubic-exchangeable measures which is more general than the construction in \eqref{eq:basicex} but still more specific than that in \eqref{eq:zeta}, and which will be very useful in the sequel. This family includes all the measures of the form $\zeta_{\ns,m}$ with the property that the Borel map $m:\ns\to\mc{P}(\Bo)$ takes values only among Dirac measures. We postpone the discussion on this family to Section \ref{sec:repvsnils}. 

\section{2-homogeneous cubic couplings}\label{sec:2-homCC}

\noindent We recall from \cite[Definition 2.18]{CScouplings} that given a probability space $\varOmega=(\Omega,\mc{A},\lambda)$ and a set $S$, a \emph{coupling} (or self-coupling) of $\varOmega$ indexed by $S$ is a measure $\mu$ on the product measurable space $(\Omega^S,\mc{A}^{\otimes S})$ such that for each $v\in S$ we have $\mu\co p_v^{-1}=\lambda$ (where $p_v$ is the projection to the $v$-th coordinate). Given an injection $\tau: R \to S$, we then denote by $\mu_\tau$ the self-coupling of $\varOmega$ indexed by $R$ obtained as follows: we take the pushforward $\mu\co p_{\tau(R)}^{-1}$, and we view it as a measure on $\Omega^R$ by identifying $R$ and $\tau(R)$ (see \cite[Definition 2.26]{CScouplings}).

In this section we shall apply results from \cite{CScouplings} concerning \emph{cubic couplings}. These  are measure-theoretic structures defined as follows (see \cite[Defintion 3.1]{CScouplings}).
\begin{defn}\label{def:cub-cou} 
A \emph{cubic coupling} on a probability space $\varOmega=(\Omega,\mc{A},\lambda)$ is a  sequence $\big(\mu^{\db{n}}\big)_{n\geq 0}$, where the $n$-th term $\mu^{\db{n}}$ is a self-coupling of $\varOmega$ indexed by $\db{n}$, and such that the following axioms hold for all $m,n\geq 0$:
\begin{enumerate}[leftmargin=0.7cm]
\item[1.] (Consistency)\; If $\phi:\db{m}\to\db{n}$ is an injective cube morphism, then $\mu^{\db{n}}_\phi=\mu^{\db{m}}$.
\item[2.] (Ergodicity)\; The measure $\mu^{\db{1}}$ is the independent coupling $\lambda \times \lambda$.
\item[3.] (Conditional independence)\; We have\footnote{For $R_1,R_2\subset S$, the notation $R_1~ \bot ~  R_2$ here denotes orthogonality of the subspaces $L^2(p_{R_i}^{-1}(\mc{A}^{\otimes R_i}))$, $i=1,2$ inside $L^2(\mc{A}^{\otimes S})$; see \cite[Definition 2.29]{CScouplings}.} $(\{0\}\times \db{n-1})~ \bot ~ (\db{n-1}\times \{0\})$ in $\mu^{\db{n}}$.
\end{enumerate}
\end{defn}
\noindent For the purposes of this paper we define the following class of cubic couplings.
\begin{defn}[2-homogeneous cubic coupling]\label{def:2-hom-cou} 
Let $(\mu^{\db{n}})_{n\ge 0}$ be a cubic coupling on a probability space $(\Omega,\mc{A},\lambda)$. We say that $(\mu^{\db{n}})_{n\ge 0}$ is a \emph{2-homogeneous cubic coupling} if for every injective  affine map $T:\mb{F}_2^m \to \mb{F}_2^n$, we have $\mu_T^{\db{n}} = \mu^{\db{m}}$.
\end{defn}
\noindent Recalling the concept of $p$-homogeneous nilspace (Definition \ref{def:p-hom}), we can now state the main result of this section, which tells us, roughly speaking, that every 2-homogeneous cubic coupling is essentially the sequence of cubic Haar measures on some 2-homogeneous compact nilspace.
\begin{theorem}\label{thm:2-hom-cubic} 
Let $(\mu^{\db{n}})_{n\ge 0}$ be a 2-homogeneous cubic coupling on a probability space $(\Omega,\mc{A},\lambda)$. Then there exists a 2-homogeneous compact nilspace $\ns$ and a measure-preserving map $\gamma:\Omega\to \ns$ such that for each $n\in \mb{N}$, the map $\gamma^{\db{n}}$ is measure-preserving $(\Omega^{\db{n}},\mu^{\db{n}})\to (\ns^{\db{n}},\mu_{\cu^n(\ns)})$.  Furthermore, for each $n$ the coupling $\mu^{\db{n}}$ is relatively independent over the factor generated by $\gamma^{\db{n}}$.
\end{theorem}
\begin{proof}  
Recall from \cite[Definition 3.31 and Definition 3.40]{CScouplings} the definitions of $\gamma_k:\Omega\to \ns_k$ and $\gamma:\Omega\to \ns$ respectively. Theorem \ref{thm:2-hom-cubic} follows mainly from \cite[Theorem 4.2]{CScouplings}. Indeed the latter result tells us that there is a compact profinite-step nilspace $\ns$ and map $\gamma$ satisfying the conclusions of Theorem \ref{thm:2-hom-cubic}, and then it only remains to prove that $\ns$ is 2-homogeneous. In order to do this, by Lemma \ref{lem:2-hom-equals-c-omega-poly} we just need to prove that $\cu^\omega(\ns)=\hom(\mc{D}_1(\mb{F}_2^\omega),\ns)$. To prove this it suffices to show that for every integer $n\ge 0$ we have $\cu^n(\ns)=\hom(\mc{D}_1(\mb{F}_2^n),\ns)$. Moreover, since $\ns$ is of profinite-step, the latter is equivalent to proving that for any fixed $k\ge0$ and $n\ge 1$ we have $\cu^n(\ns_k)=\hom(\mc{D}_1(\mb{F}_2^n),\ns_k)$. Thus (by definition of $\hom(\mc{D}_1(\mb{F}_2^n),\ns_k)$) we need to check that given $\q\in \cu^n(\ns_k)$ and $T\in \cu^m(\mc{D}_1(\mb{F}_2^n))$ we have $\q\co T\in \cu^m(\ns_k)$ (identifying $\db{n}$ with $\mb{F}_2^n$ the usual way). But as $\mc{D}_1(\mb{F}_2^n)$ is 2-homogeneous, any $T\in \cu^m(\mc{D}_1(\mb{F}_2^n))$ is just an affine-linear map $T:\mb{F}_2^m\to \mb{F}_2^n$, so it suffices to check the above composition property for such $T$.

Furthermore, we claim that it suffices to verify the above composition property only for \emph{injective} affine  maps. Indeed, suppose this has been verified and let $T:\mb{F}_2^m\to \mb{F}_2^n$ be any affine map. Composing with a translation if needed (which is an injective affine transformation) we can assume that $T$ is a linear map. Let $f$ be an invertible linear map on $\mb{F}_2^m$ such that $\ker(T\co f) = \{0^{\ell}\}\times \mb{F}_2^{m-\ell}$ for some $\ell\ge 0$. Thus $\ker(T)=f(\{0^{\ell}\}\times \mb{F}_2^{m-\ell})$ and we define the subspace $V:=f(\mb{F}_2^{\ell}\times \{0^{m-\ell}\})$. Note that then $\mb{F}_2^m = \ker(T)\oplus V$. Now let us define $p:\mb{F}_2^m\to \mb{F}_2^{\ell}$, $p(v_1,\ldots,v_m):=(v_1,\ldots,v_{\ell})$ and $i:\mb{F}_2^{\ell}\to \mb{F}_2^m$, $i(v_1,\ldots,v_{\ell}):=(v_1,\ldots,v_{\ell},0^{m-\ell})$. It can be checked then that $T=T\co f\co i\co p\co f^{-1}$. Note that $\phi:=T\co f\co i$ is an injective linear map. Thus, if we have the composition property for such maps, then $\q\co T = \q\co \phi\co p\co f^{-1}$ where $\phi$ and $f^{-1}$ are affine injective maps and $p$ is a discrete-cube morphism, so the property will follow in general as required. 

Thus it suffices to prove the result for any such affine injective map $\phi$. Arguing as in the proof of \cite[Lemma 4.7]{CScouplings}, let $V:=\phi(\mb{F}_2^m)$ and $\psi:\Omega^{\mb{F}_2^m}\to \Omega^{V}$ be the bijection that relabels each coordinate $\omega_v$ as $\omega_{\phi(v)}$. Let $\xi:\ns_k^{\mb{F}_2^m}\to \ns_k^V$ be the corresponding bijection between the spaces $\ns_k^{\mb{F}_2^m}$ and $\ns_k^V$. By Definition \ref{def:2-hom-cou}, we have $\mu^{\db{m}} = \mu_{\phi}^{\db{{\ell}}}=\mu_V^{\db{{\ell}}} \co \psi$ and $\psi \co (\gamma_k^{\db{m}})^{-1} = (\gamma_k^V)^{-1} \co \xi$. Thus, the measure support $\Supp(\mu^{\db{m}} \co (\gamma_k^{\db{m}})^{-1})$ is equal to $\xi^{-1}(\Supp(\mu_V^{\db{{\ell}}} \co (\gamma_k^V)^{-1}))$. Now we have to check that for every $\q \in \Supp(\mu^{\db{{\ell}}} \co (\gamma_k^{\db{{\ell}}})^{-1})$ we have $\q\co \phi \in \Supp(\mu^{\db{m}} \co (\gamma_k^{\db{m}})^{-1})$. Given any open neighborhood $U$ of $p_V(\q)$ (where $p_V:\ns_k^{\db{{\ell}}}\to \ns_k^V$ is the usual projection), since $\mu^{\db{{\ell}}}_V \co (\gamma_k^V)^{-1}$ is the image of $\mu^{\db{{\ell}}} \co (\gamma_k^{\db{{\ell}}})^{-1}$ under $p_V$, and $p_V$ is continuous, we have $\mu_V^{\db{{\ell}}} \co (\gamma_k^V)^{-1}(U)>0$. Hence $p_V(\q)$ is in the support $\Supp(\mu^{\db{{\ell}}} \co (\gamma_k^V)^{-1})$. The result now follows by the definition of $\ns_k$ in terms of these measure supports, as in \cite[Lemma 4.7]{CScouplings}.
\end{proof}

\section{\texorpdfstring{$p$}{p}-homogeneous compact nilspaces as fibration images of \texorpdfstring{$\mc{H}_p$}{Hp}}\label{sec:p-hom}

\noindent Throughout this section we fix a prime $p\in \mb{N}$. The main results in this section concern $p$-homogeneous nilspaces, which were introduced in \cite{CGSS-p-hom} and recalled in Definition \ref{def:p-hom} above. Let us refer to \cite{CGSS-p-hom} for more background and motivation on these nilspaces. 

Our goal in this section is to prove that there exists a compact $p$-homogeneous nilspace, that we shall denote by $\mc{H}_p$, which is ``universal" among $p$-homogeneous nilspaces, in the sense that every other compact profinite-step $p$-homogeneous nilspace is an image of $\mc{H}_p$ under a fibration. 

To define $\mc{H}_p$, we shall use the basic $p$-homogeneous nilspaces from \cite[Definition 1.6]{CGSS-p-hom}, which are denoted by $\abph_{k,\ell}$ (or $\abph_{k,\ell}^{(p)}$ when we need to specify the prime $p$). Let us recall their definition here for convenience. For positive integers $k\geq \ell$, let $G$ be the cyclic group of order $p^{\lfloor \frac{k-{\ell}}{p-1}\rfloor+1}$, and let us equip this with  the filtration $G_\bullet=(G_{(i)})_{i=0}^{\infty}$ with $i$-th term $G_{(i)}=G$ for $i\in [0,\ell]$, and $G_{(i)}:=p^{\lfloor\frac{i-{\ell}-1}{p-1}\rfloor+1}G=\{p^{\lfloor\frac{i-{\ell}-1}{p-1}\rfloor+1}x:x\in G\}$ for $i>\ell$. The nilspace $\abph_{k,{\ell}}$ is then the group nilspace associated with the filtered group $(G,G_\bullet)$. We denote by $\mc{Q}_{p,k}$ the set of all finite products of nilspaces of the form $\abph^{(p)}_{k,{\ell}}$ for ${\ell}\in [k]$. 

Note that the $(k+1)$-step nilspace $\abph_{k+1,{\ell}}$ is an extension (in the nilspace sense; see \cite[\S 3.3.3]{Cand:Notes1}) of the $k$-step nilspace $\abph_{k,{\ell}}$. Indeed, the corresponding nilspace factor map $\abph_{k+1,{\ell}}\to \abph_{k,{\ell}}$, which we shall denote by $\pi_{k}^{({\ell})}$,  is just quotienting by the last nilspace structure group of $\abph_{k+1,{\ell}}$; this group is either the trivial group $\{0\}$ or is $\mb{Z}_p$, depending on whether $\lfloor \frac{k+1-{\ell}}{p-1}\rfloor -\lfloor \frac{k-{\ell}}{p-1}\rfloor$ is 0 or 1. This yields a natural way of extending any nilspace in $\mc{Q}_{p,k}$ to a nilspace in $\mc{Q}_{p,k+1}$, as follows.
\begin{defn} 
Let $\ns\in \mc{Q}_{p,k}$, thus for some integers $a_\ell\geq 0$ we have that $\ns$ is the product nilspace $\prod_{{\ell}=1}^k (\abph^{(p)}_{k,{\ell}})^{a_{\ell}}$. Then we define $\ns^+:=\prod_{{\ell}=1}^k (\abph^{(p)}_{k+1,{\ell}})^{a_{\ell}}\in \mc{Q}_{p,k+1}$.
\end{defn}
\noindent The nilspace factor map $\pi_k:\ns^+\to \ns$ is easily described as a product of the above maps $\pi_{k}^{(1)},\ldots,\pi_{k}^{(k)}$, by applying the adequate such map $\pi_{k}^{(\ell)}$ to each component $\abph^{(p)}_{k+1,{\ell}}$ of the product $\ns^+$.

Now, for any fixed ${\ell}\ge 1$, we define the compact profinite-step nilspace $\abph^{(p)}_{\infty,{\ell}}$ as the inverse limit of the nilspaces $\abph^{(p)}_{k,{\ell}}$:
\begin{equation}
\abph^{(p)}_{\infty,{\ell}}:=\varprojlim \abph^{(p)}_{k,{\ell}}.
\end{equation}
More precisely, the inverse system used here is $\{\pi_i^{({\ell})}\co \pi_{i+1}^{({\ell})}\co\cdots\co \pi_{j-1}^{({\ell})} ~|~ i,j\in\mb{N}, i\leq j\}$ (see \cite[\S 2.7]{Cand:Notes2}). We can rephrase this construction, and thus define $\mc{H}_p$, in terms of the $p$-adic integers as follows.
\begin{defn}\label{def:del-p-adic-objects} 
For any prime $p\in \mb{N}$, let $\mf{Z}_p$ denote the group of $p$-adic integers. For $\ell\in\mb{N}$ we denote by $\abph^{(p)}_{\infty,\ell}$ the compact group-nilspace associated with the filtered group  $\big(G, G_\bullet^{(\ell)}\big)$ where $G= \mf{Z}_p$ and the filtration $G_\bullet^{(\ell)}=(G_{(i)}^{(\ell)})_{i=0}^{\infty}$ consists of $G_{(i)}^{(\ell)}= \mf{Z}_p$ for $i\in[0,\ell]$ and $G_{(i)}^{(\ell)}=p^{\lfloor\frac{i-{\ell}-1}{p-1}\rfloor+1}\mf{Z}_p$ for $i>\ell$. We define $\mc{H}_p$ as the product nilspace $\prod_{\ell=1}^{\infty}(\abph^{(p)}_{\infty,\ell})^{\mb{N}}$.
\end{defn}
\noindent Thus $\mc{H}_p$ is the group nilspace associated with the compact abelian group $\prod_{\ell=1}^{\infty}\mf{Z}_p^{\mb{N}}\cong (\mf{Z}_p^{\mb{N}})^{\mb{N}}$, which we equip with the product filtration $\prod_{\ell=1}^\infty (G_\bullet^{(\ell)})^{\mb{N}}$. Note that this filtration is $p$-homogeneous, whence (using \cite[Theorem 1.4]{CGSS-p-hom}) we have that $\mc{H}_p$ is a $p$-homogeneous compact nilspace. It is readily seen that $\mc{H}_p$ is also profinite-step.

 Note that $\mc{H}_p$ is a natural generalization for $p>2$ of the nilspace $\mc{H}$ from Definition \ref{def:H} (indeed the latter nilspace is just $\mc{H}_2$).

We can now state the main result of this section, establishing the above-mentioned universality of $\mc{H}_p$.
\begin{theorem}\label{thm:universal-covering}
Let $\ns$ be a $p$-homogeneous compact profinite-step nilspace. Then there exists a fibration $\varphi:\mc{H}_p\to \ns$. 
\end{theorem}
\noindent We shall obtain this theorem by combining infinitely many applications of a previous  structural result obtained in \cite[Theorem 1.7]{CGSS-p-hom}. This will require ``gluing" countably many fibrations in a unified way. To do this, we need the following preparation.

First we record the following fact about how the nilspace factor maps interact with \emph{translations} on a nilspace. Translations are a basic and important kind of transformations on a nilspace, which can be thought of as generalizations of the translations on a nilpotent group that consist in multiplying by  fixed elements. For the basic background concerning the groups of translations $\tran_i(\ns)$ of a nilspace $\ns$, we refer to \cite[\S 3.2.4]{Cand:Notes1}.
\begin{proposition}\label{prop:lift-trans} 
Let $\ns\in \mc{Q}_{p,k}$ and let $\pi_{k-1}:\ns\to \ns_{k-1}$ be the projection to the $(k-1)$-step nilspace factor. Then for any translation $\alpha\in \tran_i(\ns_{k-1})$ there exists $\beta\in \tran_i(\ns)$ such that $\pi_{k-1} \co \beta = \alpha \co \pi_{k-1}$. 
\end{proposition}
\noindent In other words, every translation $\alpha\in \tran_i(\ns_{k-1})$ can be lifted through $\pi_{k-1}$ to a translation $\beta \in \tran_i(\ns)$.
\begin{proof} 
First note that for $i=k$ the result is trivial, since then $\tran_k(\ns_{k-1})=\{\id\}$. For $i<k$, let us use the nilspace $\mc{T}^*$ constructed in \cite[\S 3.3.4]{Cand:Notes1}, whose structure encodes whether or not the translation $\alpha$ can be lifted as desired. By \cite[Lemma 3.3.38]{Cand:Notes1} we know that $\mc{T}^*$ is a degree-$(k-i)$ extension of $\ns_{k-1}\in \mc{Q}_{p,k-1}$. From the definition of $\mc{T}^*$ and the fact that $\ns$ is $p$-homogeneous, it also follows that $\mc{T}^*$ is $p$-homogeneous. As $i<k$, by \cite[Proposition 4.3 and Lemma 4.4]{CGSS-p-hom} we have that $\mc{T}^*$ is a \emph{split extension} of $\ns_{k-1}$. Then \cite[Proposition 3.3.39]{Cand:Notes1} tells us that there exists $\beta\in\tran_i(\ns)$ satisfying  $\pi_{k-1} \co \beta = \alpha \co \pi_{k-1}$ as required.
\end{proof}
\noindent Next, we extend this last result to be able to lift translations through fibrations that are more general than $\pi_{k-1}$.
\begin{proposition}\label{prop:lift-any-tran}
Let  $\ns\in \mc{Q}_{p,k}$, let $\nss$ be a finite $k$-step $p$-homogeneous nilspace, and let  $\varphi:\ns\to \nss$ be a fibration. Then for any $i\in [k]$ and any translation $\alpha\in \tran_i(\nss)$, there exists $\beta\in \tran_i(\ns)$ such that $\varphi \co \beta = \alpha \co \varphi$.
\end{proposition}
\begin{proof}
We argue by induction on $k$. Note that the case $k=0$ is trivial.

If $i=k$ then, since $\tran_k(\ns)\cong \ab_k(\ns)$, $\tran_k(\nss)\cong \ab_k(\nss)$ (by \cite[Lemma 3.2.37]{Cand:Notes1}) and since $\varphi$ is a fibration,  we can easily lift any such translation using the surjective homomorphism $\ab_k(\ns)\to \ab_k(\nss)$ induced as a \emph{structure morphism} by $\varphi$ (see \cite[\S 3.3.2]{Cand:Notes1}). If $i<k$, then given $\alpha\in \tran_i(\nss)$, let $\alpha_{k-1}\in \tran_i(\nss_{k-1})$ be the translation such that $\pi_{k-1,\nss} \co \alpha =\alpha_{k-1}\co \pi_{k-1,\nss}$ where $\pi_{k-1,\nss}:\nss\to\nss_{k-1}$ is the projection to the $(k-1)$-step factor of $\nss$. By induction on $k$ there is $\beta_{k-1}'\in \tran_i(\ns_{k-1})$ such that $\varphi_{k-1} \co \beta_{k-1}' = \alpha_{k-1}\co \varphi_{k-1}$, where $\varphi_{k-1}:\ns_{k-1}\to \nss_{k-1}$ is the fibration induced by $\varphi$ between the $(k-1)$-step factors (see \cite[\S 3.3.2]{Cand:Notes1}). By Proposition \ref{prop:lift-trans} there is $\beta'\in \tran_i(\ns)$ with $\pi_{k-1,\ns} \co \beta' = \beta_{k-1}'\co \pi_{k-1,\ns}$. 

Note that $\pi_{k-1,\nss} \co \varphi\co \beta' = \pi_{k-1,\nss} \co \alpha \co \varphi$, so $\beta'$ is a  translation of the desired kind but only modulo $\pi_{k-1,\nss}$. However, this implies that the function $f:\ns\to  \ab_k(\nss)$, $x\mapsto \alpha(\varphi(x))-\varphi(\beta'(x))$ is a well-defined $\ab_k(\nss)$-valued map, and is in fact a morphism into $\mc{D}_k(\ab_k(\nss))$. Let $\phi_k:\ab_k(\ns)\to \ab_k(\nss)$ be the above-mentioned structure morphism, i.e.\ the surjective homomorphism such that $\varphi(x+z)=\varphi(x)+\phi_k(z)$ for all $x\in \ns$ and $z\in \ab_k(\ns)$. Since $\ns$ and $\nss$ are both finite $p$-homogeneous nilspaces, this map $\phi_k$ is a linear map between two finite-dimensional vector spaces over $\mb{F}_p$. Thus there exists a homomorphism $s:\ab_k(\nss)\to \ab_k(\ns)$ such that $\phi_k\co s$ is the identity map on $\ab_k(\nss)$. We then define $\beta:\ns\to \ns$ by $\beta(x):=\beta'(x)+s(f(x))$. It is easily checked that $\beta\in \tran_i(\ns)$ and that $\varphi \co \beta = \alpha\co \varphi$, as required.
\end{proof}
\noindent We now use Proposition \ref{prop:lift-any-tran} to show that any fibration between nilspaces in the special class $\mc{Q}_{p,k}$ can be re-expressed essentially as a coordinate projection, which will be useful in order to carry out the above-mentioned gluing in an orderly way.
\begin{proposition}\label{prop:spli-fib-q-pk}
Let $\ns,\nss\in \mc{Q}_{p,k}$ and let $\varphi:\ns\to \nss$ be a fibration. Then $\ns$ is isomorphic to the product nilspace $\nss \times Q$ for some $k$-step, $p$-homogeneous finite nilspace $Q$, and there is a nilspace isomorphism $\Phi:\nss\times \,Q\to \ns$ such that $\varphi(\Phi(y,q))= x$ for any $y\in \nss$, $q\in Q$.
\end{proposition}
\begin{proof}
We have $\nss = \prod_{\ell=1}^k (\abph^{(p)}_{k,\ell})^{a_\ell}$, where $a_\ell\ge 0$ for $\ell\in [k]$. For each $\ell \in [k]$ and $t\in [a_\ell]$, let $\alpha_{\ell,t}\in \tran_\ell(\nss)$ be the translation that acts by adding 1 inside the cyclic group which is the $t$-th component of the group $(\abph^{(p)}_{k,\ell})^{a_\ell}$. By Proposition \ref{prop:lift-any-tran} we know that there exists $\beta_{\ell,t}\in \tran_\ell(\ns)$ such that $\varphi \co \beta_{\ell,t} = \alpha_{\ell,t}\co \varphi$. Let $Q:=\varphi^{-1}(\textbf{0})$ where $\textbf{0}\in \nss$ is the element with all coordinates equal to 0. It is straightforward to check that $Q$ is  a $p$-homogeneous, $k$-step, finite nilspace. Let us define
\[
\begin{array}{cccc}
    \Phi: & \nss\times \,Q & \to & \ns \\
    & (y,q) & \to & \left(\prod_{\ell=1}^k \prod_{t=1}^{a_\ell} \beta_{\ell,t}^{y_{\ell,t}}\right)(q) 
\end{array}
\]
where $y = (y_{\ell,t})_{\ell\in [k],\, t\in [a_\ell]}$. Note that this map $\Phi$ is thus well-defined, since the order of $\beta_{\ell,t}$ as an element of the translation group $\tran(\ns)$ is the same as the order of $y_{\ell,t}$ in its corresponding cyclic group, by \cite[Proposition 3.11]{CGSS-p-hom}. More precisely, since  $y_{\ell,t}\in \mb{Z}_{p^{r}}$ for $r=\lfloor\frac{k-\ell}{p-1}\rfloor+1$, we have $\beta_{\ell,t}^{p^r}=\id$ (using that $\beta_{\ell,t}\in \tran_{\ell}(\ns)$ and $(\tran_{\ell}(\ns))_{\ell\ge 0}$ is a $p$-homogeneous filtration by \cite[Proposition 3.11]{CGSS-p-hom}). Thus, it now suffices to prove that $\Phi$ is a nilspace isomorphism. The fact that it is a morphism follows from the definitions. To see that it is bijective, note that the following function is an inverse:
\[
\Phi^{-1}(x):=\Big(y, \Big(\prod_{\ell=1}^k \prod_{t=1}^{a_\ell} \beta_{\ell,t}^{y_{\ell,t}}\Big)^{-1} (y)\Big), \textrm{ where } y = \varphi(x).
\]
Note also that $\Phi^{-1}$ is a morphism, similarly as above (i.e.\ from the definitions).
\end{proof}
\noindent We can now explain how we glue various fibrations to obtain a fibration from $\mc{H}_p$.
\begin{proposition}\label{prop:aux-universal-covering}
Let $\ns$ be a compact nilspace such that $\ns=\varprojlim \ns_i$ where $\ns_i$ is an $i$-step, finite $p$-homogeneous nilspace for each $i$. Then there is a strict inverse system of nilspaces $(\phi_{i,j}:\nss_j\to \nss_i)_{i\le j}$ and a fibration $\varphi:\varprojlim\nss_i\to \ns$, such that for every $i$ we have $\nss_i\in \mc{Q}_{p,i}$, $\nss_{i+1} = \nss_i^+\times\, Q_{i+1}$ for some $Q_{i+1}\in \mc{Q}_{p,i+1}$ and $\phi_{i,i+1}=\pi_i\co p_1$, where $p_1$ is the projection to the first component $\nss_i^+$ and $\pi_i$ is the quotient map $\nss_i^+\to\nss_i$.
\end{proposition}
\begin{proof}[Proof of Theorem \ref{thm:universal-covering} using Proposition \ref{prop:aux-universal-covering}]
We combine \cite[Theorem 5.71]{GS} and Proposition \ref{prop:aux-universal-covering}. The nilspace obtained is given as an inverse system $\varprojlim \nss_i$ where the factor maps $\phi_{i,i+1}:\nss_{i+1}\to \nss_i$ have the form of a projection to some coordinates and then taking the quotient map $\pi_i$. This means that the inverse system has the form \[\prod_{\ell=1}^1(\abph^{(p)}_{1,\ell})^{a_{1,\ell}}\leftarrow\prod_{\ell=1}^1(\abph^{(p)}_{2,\ell})^{a_{1,\ell}}\times \prod_{\ell=1}^2(\abph^{(p)}_{2,\ell})^{a_{2,\ell}}\leftarrow \prod_{\ell=1}^1(\abph^{(p)}_{3,\ell})^{a_{1,\ell}}\times \prod_{\ell=1}^2(\abph^{(p)}_{3,\ell})^{a_{2,\ell}}\times \prod_{\ell=1}^3(\abph^{(p)}_{3,\ell})^{a_{3,\ell}}\leftarrow\cdots\]
where each map $\phi_{i,i+1}$ deletes the coordinate in the last component $\prod_{\ell=1}^{i+1}(\abph^{(p)}_{i+1,\ell})^{a_{i+1,\ell}}$ and applies the quotient map $\pi_i$ in the remaining coordinates. Thus, the nilspace obtained has the form $\prod_{{\ell}=1}^{\infty} (\abph^{(p)}_{\infty,{\ell}})^{a_{\ell}}$ where $a_{\ell} = \sum_{k=\ell}^{\infty} a_{k,\ell}\in \mb{N}\cup \{\infty\}$. By taking a projection that consists in deleting some coordinates, we see that this nilspace is a fiber-surjective image of $\mc{H}_p = \prod_{{\ell}=1}^{\infty} (\abph^{(p)}_{\infty,{\ell}})^{\mb{N}}$.
\end{proof}

\begin{proof}[Proof of Proposition \ref{prop:aux-universal-covering}]
The first part of the proof is very similar to the proof of \cite[Proposition 4.9]{CGSS-p-hom}. We will construct the sequence $\nss_i$ by induction on $i$, with the case $i=1$ being trivial (as $\ns_1=\mc{D}_1(\mb{Z}_p^n)$).

Suppose that we have already constructed all the factors from $\nss_1$ to $\nss_{n-1}$, thus we have $\nss_j\in \mc{Q}_{p,j}$ for $j\in [n-1]$, fibrations $\psi_j:\nss_j\to \ns_j$ for $j\in [n-1]$ and maps $\phi_{j,{\ell}}:\nss_{\ell}\to \nss_j$ (projections to first components followed by quotient maps) for $1\le j\le {\ell} \le n-1$ such that $\phi_{j,j}=\id$ for all $j\in [n-1]$ and $\phi_{j,{\ell}} \co \phi_{{\ell},d} = \phi_{j,d}$ for $1\le j\le {\ell} \le d\le n-1$. Also, letting $\gamma_{j,{\ell}}:\ns_{\ell} \to \ns_j$ for $1\le j \le {\ell}$ be the fibrations defining the inverse limit of $\ns$, we have $\psi_j \co \phi_{j,{\ell}} = \gamma_{j,{\ell}}\co \psi_{\ell}$ for all $1\le j\le {\ell} \le n-1$. We can represent the situation with the following diagram:
\begin{center}
\begin{tikzpicture}
  \matrix (m) [matrix of math nodes,row sep=3em,column sep=4em,minimum width=2em]
  { \ns_1 & \cdots & \ns_{n-1} & \ns_{n} \\
     \nss_1  & \cdots & \nss_{n-1} & .\\};
  \path[-stealth]
    (m-1-2) edge node [above] {$\gamma_{1,2}$} (m-1-1)
    (m-1-3) edge node [above] {$\gamma_{n-2,n-1}$} (m-1-2)
    (m-1-4) edge node [above] {$\gamma_{n-1,n}$} (m-1-3)
    (m-2-1) edge node [right] {$\psi_1$} (m-1-1)
    (m-2-3) edge node [right] {$\psi_{n-1}$} (m-1-3)
    (m-2-2) edge node [above] {$\phi_{1,2}$} (m-2-1)
    (m-2-3) edge node [above] {$\phi_{n-2,n-1}$} (m-2-2)
    ;
\end{tikzpicture}
\end{center}
Now, the first part of the induction consists in finding a nilspace $\nss_n'\in \mc{Q}_{p,n}$ such that the following diagram commutes:
\begin{equation}\label{diag:induc1}
\begin{tikzpicture}
  \matrix (m) [matrix of math nodes,row sep=3em,column sep=4em,minimum width=2em]
  {\ns_{n-1} & \ns_n  \\
     \nss_{n-1} & \nss_n', \\};
  \path[-stealth]
    (m-1-2) edge node [above] {$\gamma_{n-1,n}$} (m-1-1)
    (m-2-1) edge node [right] {$\psi_{n-1}$} (m-1-1)
    (m-2-2) edge node [right] {$\rho$} (m-1-2)
    (m-2-2) edge node [above] {$\nu$} (m-2-1);
\end{tikzpicture}
\end{equation}
where $\nu$ and $\rho$ are fibrations. To do so, we let $D$ be the subdirect product of $\nss_{n-1}$ and $\ns_n$, i.e., $D:=\{(y,x)\in \nss_{n-1}\times \ns_n: \psi_{n-1}(y)=\gamma_{n-1,n}(x)\}$. Then, we  apply \cite[Theorem 1.7]{CGSS-p-hom} to this nilspace (which is clearly $n$-step) to obtain $\nss_n'$. Note that if $\nss_n'$ and $\nu$ had the desired properties then we could stop here and proceed to the next step of the induction. However, we cannot guarantee this in general, so let us describe how one can define an appropriate nilspace $\nss_n$ with the desired properties. Let us start with the following diagram indicating how the definition of $\nss_n$ will work:
\begin{center}
\begin{tikzpicture}
  \matrix (m) [matrix of math nodes,row sep=3em,column sep=4em,minimum width=2em]
  {
     \ns_{n-1} & \ns_n   & & \\
     \nss_{n-1} & \nss_n' & T & \nss_n \\
     & \pi_{n-1}(\nss_n') &&\\
     & \nss_{n-1}\times Q & \nss_{n-1}^+\times V^+ & \\
     };
  \path[-stealth]
    (m-1-2) edge node [above] {$\gamma_{n-1,n}$} (m-1-1)
    (m-2-1) edge node [right] {$\psi_{n-1}$} (m-1-1)
    (m-2-2) edge node [right] {$\rho$} (m-1-2)
    (m-2-2) edge node [above] {$\nu$} (m-2-1)
    (m-2-2) edge node [right] {$\pi_{n-1}$} (m-3-2)
    (m-3-2) edge node [right] {$\nu_{n-1}$} (m-2-1)
    (m-3-2) edge node [right] {$\phi$} (m-4-2)
    (m-4-3) edge node [above] {$\xi$} (m-4-2)
    (m-2-3) edge node [above] {$p_1$} (m-2-2)
    (m-2-3) edge node [right] {$p_2$} (m-4-3)
    (m-2-4) edge node [above] {$\beta$} (m-2-3)
    ;
\end{tikzpicture}
\end{center}
Now let us detail the construction. First, note that $\nu:\nss_n'\to \nss_{n-1}$ is a fibration from an $n$-step nilspace to an $(n-1)$-step nilspace. Thus $\nu$ factors through the $(n-1)$-step factor $\pi_{n-1}(\nss_n')$, i.e.\ we have $\nu=\nu_{n-1}\co \pi_{n-1}$. Now, as $\pi_{n-1}(\nss_n')\in \mc{Q}_{p,n-1}$,  we apply Proposition  \ref{prop:spli-fib-q-pk} to the fibration $\nu_{n-1}$ and we denote by $\phi:\pi_{n-1}(\nss_n')\to \nss_{n-1}\times Q$ the nilspace isomorphism such that $\nu_{n-1}\co \phi^{-1}(y,q)=y$ for every $y\in \nss_{n-1}$ and $q\in Q$ (so $\nu_{n-1}\co \phi^{-1}$ is simply the projection to the first component $\nss_{n-1}$). Now we apply \cite[Theorem 1.7]{CGSS-p-hom} to $Q$ (which is a $p$-homogeneous  $(n-1)$-step finite  nilspace) to obtain $V\in \mc{Q}_{p,n-1}$ and a fibration $\eta:V\to Q$. We then take the extensions $V^+$ and $\nss_{n-1}^+$, and let $\xi:\nss_{n-1}^+\times V^+\to \nss_{n-1}\times Q$ be the corresponding fibration (thus $\xi$ applies the projection $\wt{\pi}_{n-1}:\nss_{n-1}^+\to\nss_{n-1}$ in the first component, and the fibration $\eta\co \pi_{n-1}: V^+\to Q$ in the second component). Note that $\nss_{n-1}^+\times V^+\in \mc{Q}_{p,n}$.

Next, we define the subdirect product 
\[
T:=\{(y',(y,v))\in\nss_n'\times (\nss_{n-1}^+\times V^+): \phi^{-1}\co \xi (y,v) = \pi_{n-1}(y') \}.
\]
By \cite[Proposition A.17]{CGSS-p-hom} we have that $T$ is a degree-$n$ extension of $\nss_{n-1}^+\times V^+$ by the last structure group $\ab_n(\nss_n')$. By \cite[Proposition 4.3]{CGSS-p-hom} this extension splits. We define $\nss_n=\nss_{n-1}^+\times V^+\times \mc{D}_n(\ab_n(\nss_n'))$ and let $\beta$ be the isomorphism $\nss_n\to T$. Since $T$ is a split extension, we can view $p_2\co \beta$ as projection to the component $(y,v)\in \nss_{n-1}^+\times V^+$ in $\nss_n$. 

We now claim that we can take $\phi_{n-1,n}:\nss_n\to \nss_{n-1}$ to be $\nu \co p_1 \co \beta$ and $\psi_n:=\rho\co p_1\co \beta$. To see this, note first that the commutativity of the diagram in \eqref{diag:induc1} implies a similar commutative diagram with  $\phi_{n-1,n}$ instead of $\nu$ and $\psi_n$ instead of $\rho$. Hence, it only remains to check that $\phi_{n-1,n}(y,v,z)=\wt{\pi}_{n-1}(y)$. To do so, note that the previous diagram commutes, and therefore $
\phi_{n-1,n}(y,v,z) = \nu \co p_1\co \beta (y,v,z)  = \nu_{n-1}\co \pi_{n-1} \co p_1\co \beta (y,v,z) = \nu_{n-1}\co\phi^{-1} \co \xi \co p_2 \co \beta (y,v,z) = \nu_{n-1}\co\phi^{-1} \co \xi (y,v) = \wt{\pi}_{n-1}(y)$.
\end{proof}

\section{Proof of the main result}\label{sec:mainproof}
\noindent In this section we prove Theorem \ref{thm:main}. To this end, a useful observation is that  given a measure $\mu\in \textup{Pr}^{\textup{Aff}(\mb{F}_2^\omega)}(\Bo^{\db{\mb{N}}})$, we can determine whether $\mu$ is of the form $\zeta_{\ns,m}$ by checking whether $\mu$ satisfies an equivalent and often more convenient \emph{independence property}. This observation was already made in the more general context of cubic exchangeability in \cite{CScouplings}, where this independence property was defined. Let us recall the definition here. First we recall the definition of a face of $\db{\mb{N}}$ from \cite[\S 6]{CScouplings}. This involves the notation $\db{S}$ (for a countable set $S$), which denotes the set of sequences in $\{0,1\}^S$ of finite support.
\begin{defn}\label{def:faces}
A set $F\subset \db{\mb{N}}$ is a \emph{face} of $\db{\mb{N}}$ if there is a set $S\subset \mb{N}$ and an element $z\in\db{\mb{N}\setminus S}$ such that for every element $v\in F$ there is a unique $v'\in \db{S}$ such that $v\sbr{i}=v'\sbr{i}$ for $i\in S$ and $v\sbr{i}=z\sbr{i}$ for $i\in \mb{N}\setminus S$.  We call $S$ the set of \emph{free coordinates} of $F$, and we say that $F$ is a \emph{finite} face if $S$ is finite. Note that finite faces of $\db{\mb{N}}$ can be defined equivalently as the images of face maps $\db{n}\to\db{\mb{N}}$, $n\in \mb{N}$. Two faces of $\db{\mb{N}}$ are \emph{independent} if they are disjoint and their sets of free coordinates are disjoint. 
\end{defn}
\noindent Given a $\sigma$-algebra $\mc{B}$ on a set $\Bo$, and countable sets $S\subset T$, we denote by $\mc{B}^{\otimes T}_S$ the sub-$\sigma$-algebra of $\mc{B}^{\otimes T}$ consisting of those sets whose indicator functions are independent of coordinates indexed in $T\setminus S$; equivalently, letting $p_S$ be the coordinate projection $\Bo^T\to \Bo^S$, we have $\mc{B}^{\otimes T}_S=p_S^{-1}(\mc{B}^{\otimes S})$.
\begin{defn}
Let $(\Bo,\mc{B})$ be a standard Borel space. We say that a measure $\mu\in \mc{P}(\Bo^{\db{\mb{N}}})$ has the \emph{independence property} if for all finite independent faces $F_1,F_2\subset \db{\mb{N}}$, the $\sigma$-algebras $\mc{B}^{\otimes \db{\mb{N}}}_{F_1}$, $\mc{B}^{\otimes \db{\mb{N}}}_{F_2}$ are independent according\footnote{As usual, given a probability space $(\Omega,\mc{A},\mu)$, two sub-$\sigma$-algebras $\mc{B}_1,\mc{B}_2$ of $\mc{A}$ are independent according to $\mu$ if for every $A_1\in \mc{B}_1$, $A_2\in \mc{B}_2$ we have $\mu(A_1\cap A_2)=\mu(A_1)\mu(A_2)$.}  to $\mu$. 
\end{defn}
\noindent The next lemma follows from results in \cite{CScouplings}, but was not stated explicitly in that paper and will be used in the sequel, so we record it here. Following the notation from \cite[Ch.\ 11]{Ke}, given a topology $\mc{T}$ on a set $\Bo$, we denote by $\mc{B}(\mc{T})$ the $\sigma$-algebra generated by $\mc{T}$.
\begin{lemma}\label{lem:indepclosed}
Let $(\Bo,\mc{B})$ be a standard Borel space, and let
\begin{equation}\label{eq:indepset}
\mc{I}=\big\{\mu\in \mc{P}(\Bo^{\db{\mb{N}}}): \mu\textrm{ has the independence property}\big\}.
\end{equation}
Then, for any Polish topology $\mc{T}$ on $\Bo$ such that $\mc{B}=\mc{B}(\mc{T})$, the set $\mc{I}$ is a closed subset of $\mc{P}(\Bo^{\db{\mb{N}}})$ in the vague topology induced by $\mc{T}$. In particular, the set $\mc{I}$ is Borel measurable relative to the standard Borel structure on $\mc{P}(\Bo^{\db{\mb{N}}})$.
\end{lemma}
\begin{proof}
Let $(\mu_n)_{n\in\mb{N}}$ be a sequence in $\mc{I}$ converging to $\mu \in \mc{P}(\Bo^{\db{\mb{N}}})$ in the vague topology induced by $\mc{T}$. Let $F_1, F_2$ be any finite independent faces in $\db{\mb{N}}$. We need to show that $\mc{B}^{\otimes \db{\mb{N}}}_{F_1}$, $\mc{B}^{\otimes \db{\mb{N}}}_{F_2}$ are independent according to $\mu$, i.e., that for every function $f_1\in L^\infty(\mc{B}^{\otimes \db{\mb{N}}}_{F_1})$ and $f_2\in L^\infty(\mc{B}^{\otimes \db{\mb{N}}}_{F_2})$ we have $\int_{\Bo^{\db{\mb{N}}}} f_1 f_2 \ud\mu=(\int_{\Bo^{\db{\mb{N}}}} f_1 \ud\mu)(\int_{\Bo^{\db{\mb{N}}}} f_2 \ud\mu)$. By the Doob property for Polish spaces \cite[Lemma 2.17]{CScouplings} , for $i=1,2$ there is $\mc{B}^{\otimes F_i}$-measurable function $f_i'$ such that $f_i=f_i'\co p_{F_i}$. Fix any $\varepsilon>0$. For $i=1,2$, by \cite[Appendix E8]{RudinFA} the continuous functions are dense in the $L^1$-norm on the probability space $(\Bo^{F_i},\mc{B}^{\otimes F_i},\mu\co p_{F_i}^{-1})$, so there is a continuous function $\tilde f_i$ on $\Bo^{F_i}$ (relative to the product topology $\mc{T}^{F_i}$) such that
\[
\int_{\Bo^{\db{\mb{N}}}} |f_i-\tilde f_i\co p_{F_i}| \ud\mu = \int_{\Bo^{F_i}} |f_i'-\tilde f_i| \ud(\mu\co p_{F_i}^{-1}) \leq \frac{\varepsilon}{2\max(\|f_1\|_{L^\infty},\|f_2\|_{L^\infty},1)}.
\]
We thus have in particular $|\int_{\Bo^{\db{\mb{N}}}} f_1 f_2 \ud\mu - \int_{\Bo^{\db{\mb{N}}}}  (\tilde f_1\co p_{F_1})  (\tilde f_2\co p_{F_2})\ud\mu|\leq \varepsilon$. By definition of vague convergence and the supposed independence property of each $\mu_n$, we have
\begin{eqnarray*}
&& \int_{\Bo^{\db{\mb{N}}}} (\tilde f_1\co p_{F_1}) (\tilde f_2\co p_{F_2}) \ud\mu  =  \lim_{n\to\infty} \int_{\Bo^{\db{\mb{N}}}}(\tilde f_1\co p_{F_1}) (\tilde f_2\co p_{F_2})\ud\mu_n\\
 & = & \lim_{n\to\infty} \Big(\int_{\Bo^{\db{\mb{N}}}} \tilde f_1\co p_{F_1} \ud\mu_n\Big)\Big(\int_{\Bo^{\db{\mb{N}}}} \tilde f_2\co p_{F_2} \ud\mu_n\Big)  =  \Big(\int_{\Bo^{\db{\mb{N}}}}  \tilde f_1\co p_{F_1} \ud\mu\Big) \Big(\int_{\Bo^{\db{\mb{N}}}} \tilde f_2\co p_{F_2} \ud\mu\Big).
\end{eqnarray*}
Since this last product differs from $(\int_{\Bo^{\db{\mb{N}}}} f_1 \ud\mu)(\int_{\Bo^{\db{\mb{N}}}} f_2 \ud\mu)$ by at most $\varepsilon$, we conclude that $|\int_{\Bo^{\db{\mb{N}}}} f_1 f_2 \ud\mu - (\int_{\Bo^{\db{\mb{N}}}} f_1 \ud\mu)(\int_{\Bo^{\db{\mb{N}}}} f_2 \ud\mu)|\leq 2\varepsilon$. Letting $\varepsilon\to 0$, the result follows.
\end{proof}

We can now obtain the main result of this section.

\begin{theorem}\label{thm:maincharac}
Let $\Bo$ be a standard Borel space and let $\mu\in \mc{P}(\Bo^{\db{\mb{N}}})$. Then the following statements hold:
\begin{enumerate}[leftmargin=0.65cm]
    \item $\mu$ is affine-exchangeable with the independence property if and only if  $\mu = \zeta_{\ns,m}$ for some 2-homogeneous profinite-step compact nilspace $\ns$ and some Borel map $m:\ns\to \mc{P}(\Bo)$.
    \item If $\mu$ is affine-exchangeable then it is a mixture of affine-exchangeable measures with the independence property.
\end{enumerate}
\end{theorem}
\noindent In other words, statement $(ii)$ here tells us that there is a Borel probability measure $\kappa$ on $\mc{P}(\Bo^{\db{\mb{N}}})$, concentrated on the measurable set $\mc{I}$ from \eqref{eq:indepset}, such that $\mu = \int_{\mc{P}(\Bo^{\db{\mb{N}}})} \nu \,\ud\kappa(\nu)$.
\begin{proof} 
The argument consists in combining results from \cite{CScouplings}, in particular \cite[Theorem 6.7]{CScouplings}, with the additional information given here by Theorem \ref{thm:2-hom-cubic}, afforded by the assumption that $\mu$ is not just cubic-exchangeable but is affine-exchangeable.

The adaptation of the relevant results from \cite{CScouplings} goes as follows. 

First, by a straightforward use of \cite[Proposition 6.10]{CScouplings} we obtain that $\mu\in \mc{P}(\Bo^{\db{\mb{N}}})$ is affine-exchangeable if and only if it is the factor of a weak cubic coupling $\tilde \mu$ (see \cite[Definition 6.8]{CScouplings}) with the additional property of being \emph{2-homogeneous}, i.e.\ such that for every injective affine map $T:\mb{F}_2^m\to \mb{F}_2^n$ we have $\tilde\mu\co p_{T(\mb{F}_2^m)}^{-1}=\tilde\mu\co p_{\mb{F}_2^m}^{-1}$ (recall Definition \ref{def:2-hom-cou}). Indeed, for the forward implication, since $\mu$ is cubic-exchangeable, by \cite[Proposition 6.10]{CScouplings} there is some measurable map $q:V\to\Bo$ from some standard Borel space $V$ such that $\mu=\tilde\mu\co (q^{\db{\mb{N}}})^{-1}$ for some weak cubic coupling $\tilde\mu$ on $V^{\db{\mb{N}}}$. Furthermore, from the proof of \cite[Proposition 6.10]{CScouplings} if $\mb{N}=E\sqcup O$ where $E$ and $O$ are the sets of even and odd numbers respectively, we have that $V:=\Bo^{\db{O}}$. Then we define the coupling $\nu'$ on $V^{\db{E}}$ as the measure $\mu$, using that $V^{\db{E}}=\Bo^{\db{\mb{N}}}$. Given this, we then define $q:V\to \Bo$ as the projection on the coordinate $0^{\db{O}}$. If the measure $\mu$ is affine-exchangeable, then it follows that $\nu'$ is also affine-exchangeable. The measure $\tilde{\mu}$ is then defined as $\nu'$ composed with doubling the coordinate-indices in $\db{\mb{N}}$ (recall that $\nu'$ is a measure defined on $\Bo^{\db{E}}$ and we need a measure on $\Bo^{\db{\mb{N}}}$). The result then follows similarly as in the proof of \cite[Proposition 6.10]{CScouplings}.

The backward implication is clear because $\mu$ directly inherits from $\tilde \mu$ the additional 2-homogeneity property, and this property is easily seen to imply affine-exchangeability. (Indeed, given a 2-homogeneous cubic coupling $\mu\in \mc{P}(\Bo^{\db{\mb{N}}})$ and an element $T\in \Aff(\mb{F}_2^\omega)$, the map $T$ can be seen as an injective map $T:\mb{F}_2^n\to \mb{F}_2^n$ for some $n\ge 0$ large enough. Thus, by Definition \ref{def:2-hom-cou} we have $\mu^{\db{n}}=\mu^{\db{n}}_T$ where $\mu^{\db{n}}:=\mu\co p_n^{-1}$ and $p_n:\Bo^{\db{\mb{N}}}\to \Bo^{\db{n}}$ is the projection to the first $\db{n}$ coordinates; this is precisely the definition of $\mu$ being affine-exchangeable.)

Next, we use a straightforward adaptation of \cite[Proposition 6.13]{CScouplings}. This adaptation states that a cubic-exchangeable measure $\eta$ on $\Bo^{\db{\mb{N}}}$ that is in addition 2-homogeneous (such as the measure $\tilde \mu$ obtained above) is a mixture of cubic couplings (viewed as measures on $\Bo^{\db{\mb{N}}}$), almost everyone of which is also 2-homogeneous. This adaptation is obtained by going through the proof of \cite[Proposition 6.13]{CScouplings},  replacing the injective cube morphisms $\phi_1,\phi_2:\db{k}\to\db{\mb{N}}$ by injective affine maps $\mb{F}_2^k\to \mb{F}_2^\omega$. 

Thus far we have obtained (similarly as for \cite[Theorem 6.14]{CScouplings}) that every affine-exchangeable probability measure is a mixture of factors of 2-homogeneous cubic couplings (in the sense that a measure $\mu\in\mc{P}(\Bo^{\db{\mb{N}}})$ can be regarded as a cubic coupling when restricted to sets of the form $\db{n}\times \{0^{\mb{N}\setminus[n]}\}\subset \db{\mb{N}}$). As 2-homogeneous cubic couplings are affine exchangeable (as explained above), and cubic couplings have the independence property (by \cite[Lemma 6.15]{CScouplings}), this proves statement $(ii)$.

To prove statement $(i)$ and thus complete the proof, we now obtain an adaptation of  \cite[Theorem 6.17]{CScouplings} where the assumption is strengthened by replacing cubic-exchangeability with affine-exchangeability, and the conclusion is strengthened by adding 2-homogeneity to the resulting cubic coupling. Finally, we apply an adaptation of \cite[Lemma 6.18]{CScouplings} where the assumption is strengthened by adding the 2-homogeneity property, and as a result we may use Theorem \ref{thm:2-hom-cubic} from this paper instead of \cite[Theorem 4.1]{CScouplings}, thus obtaining that the resulting nilspace $\ns$ is 2-homogeneous, as required.
\end{proof}
\noindent We complete the proof of Theorem \ref{thm:main} by explaining how it follows from Theorem \ref{thm:maincharac}. 

\begin{proof}[Proof of Theorem \ref{thm:main}]
Theorem \ref{thm:maincharac} implies that for every $\mu\in  \textup{Pr}^{\textup{Aff}(\mb{F}_2^\omega)}(\Bo^{\db{\mb{N}}})$ there is a Borel probability measure $\kappa$ on $\mc{P}(B^{\db{\mb{N}}})$, concentrated on the set of measures of the form $\nu=\zeta_{\ns,m}$ for $\ns$ a 2-homogeneous profinite-step nilspace $\ns$ and Borel map $m:\ns\to\mc{P}(\Bo)$, such that $\mu$ is the mixture
\begin{equation}\label{eq:mainmixture}
\mu = \int_{\mc{P}(B^{\db{\mb{N}}})} \nu \,\ud\kappa(\nu).
\end{equation}
By Theorem \ref{thm:universal-covering}, for each nilspace $\ns=\ns_\nu$ appearing in this mixture, there exists a fibration $\varphi_\nu:\mc{H}\to \ns_\nu$, and it is then readily seen that the corresponding measure $\zeta_{\ns_\nu,m_\nu}$ is equal to $\zeta_{\mc{H}, m_\nu\co\varphi_\nu}$. Indeed, by standard results it suffices to check this equality for cylinder sets $S=\prod_{v\in\db{\mb{N}}} A_v\subset \Bo^{\db{\mb{N}}}$, and for any such set, using the fact that $\varphi_\nu^{\db{\mb{N}}}$ preserves the Haar measures, we have $\zeta_{\ns_\nu,m_\nu}(S) = \int_{\cu^\omega(\mc{H})} \prod_{v\in \db{\mb{N}}} m_\nu\co\varphi_\nu(\q(v))(A_v)\; \ud\mu_{\cu^{\omega}(\mc{H}_2)}(\q) = \zeta_{\mc{H},m_\nu\co \varphi_\nu}(S)$ as required. As each map $m_\nu\co \varphi_\nu$ is Borel $\mc{H}\to\mc{P}(\Bo)$, we thus deduce that the measure $\kappa$ in \eqref{eq:mainmixture} is concentrated on measures of the form $\nu=\zeta_{\mc{H},m}$ for Borel maps $m: \mc{H}\to\mc{P}(\Bo)$, and this completes the proof.
\end{proof}

\section{On the geometry of the class of affine-exchangeable measures}\label{sec:geom}

\noindent In this section we prove Theorem \ref{thm:BauerProp}, yielding the Bauer property for affine exchangeabi- lity, i.e., that for every compact metric space $\Bo$ the set $\textup{Pr}^{\textup{Aff}(\mb{F}_2^\omega)}(\Bo^{\db{\mb{N}}})$ is a Bauer simplex.

We begin by recording the following fact concerning cubic-exchangeable measures more generally.
\begin{proposition}\label{prop:extremeequivs}
Let $\Bo$ be a standard Borel space, and let $\mu\in \mc{P}(\Bo^{\db{\mb{N}}})$. The following statements are equivalent.
\begin{enumerate}[leftmargin=0.8cm]
\item $\mu$ is an extreme point in the convex set of cubic-exchangeable measures in $\mc{P}(\Bo^{\db{\mb{N}}})$.
\item $\mu$ is cubic-exchangeable and satisfies the independence property.
\item $\mu=\zeta_{\ns,m}$ for some compact profinite-step nilspace $\ns$ and Borel map $m:\ns\to \mc{P}(\Bo)$.
\end{enumerate}
\end{proposition}
\noindent The ingredients for a proof of this result are all essentially contained in \cite{CScouplings}, but the lemma is not explicitly established in that paper. We take the opportunity to do so here.
\begin{proof}

$(ii)\Rightarrow (iii)$: this is given by \cite[Theorem 6.7 (i)]{CScouplings}, more specifically by combining \cite[Theorem 6.17 and Lemma 6.18]{CScouplings}. Note that in the forward implication, the fact that the obtained nilspace $\ns$ can be assumed to be profinite-step follows from the main structure theorem \cite[Theorem 4.1]{CScouplings}, where the compact nilspace $\ns$ is obtained as an inverse limit of its characteristic factors (see \cite[Remark 4.3 and Definition 3.41]{CScouplings}). 

$(ii)\Rightarrow (i)$: suppose that $\mu$ has the independence property, and suppose for a contradiction that $\mu$ is not an extreme point, i.e., that there exist distinct cubic-exchangeable measures $\nu_1,\nu_2\in \mc{P}(\Bo^{\db{\mb{N}}})$ and $t\in (0,1)$ such that $\mu= t\nu_1 + (1-t)\nu_2$. By \cite[Theorem 6.7 (ii)]{CScouplings}, each $\nu_i$ is a mixture of cubic-exchangeable measures with the independence property, i.e.\ there exist Borel probability measures $\kappa_1,\kappa_2$ concentrated on $\mc{I}$ such that $\nu_i=\int_{\mc{P}(\Bo^{\db{\mb{N}}})} \lambda \ud\kappa_i(\lambda)$ for $i=1,2$. Note that $\kappa_1\neq \kappa_2$ (otherwise $\nu_1=\nu_2$). Then $\kappa:= t \kappa_1+(1-t)\kappa_2$ is a Borel probability measure on $\mc{P}(\Bo^{\db{\mb{N}}})$ concentrated on $\mc{I}$ such that $\mu = \int_{\mc{P}(\Bo^{\db{\mb{N}}})} \lambda \ud\kappa(\lambda)$. Note that $\kappa$ is not an extreme point of the space of Borel probability measures on $\mc{P}(\Bo^{\db{\mb{N}}})$, so in particular it is not a Dirac measure. On the other hand, since $\mu$ and $\kappa$-almost-every $\lambda$ have the independence property, arguing as in the proof of \cite[Theorem 6.17]{CScouplings} we deduce that $\kappa$-almost surely we have $\lambda=\mu$, so $\kappa$ is concentrated on a single point in $\mc{P}(\Bo^{\db{\mb{N}}})$, which yields a  contradiction since $\kappa$ is not a Dirac measure.

$(i)\Rightarrow (ii)$: suppose that $\mu$ is an extreme point among cubic-exchangeable measures, and note that by \cite[Theorem 6.7 (ii)]{CScouplings} we have  $\mu = \int_{\mc{P}(\Bo^{\db{\mb{N}}})} \nu \,\ud\kappa(\nu)$ for some Borel probability measure $\kappa$ on $\mc{P}(\Bo^{\db{\mb{N}}})$ concentrated on cubic-exchangeable measures with the independence property. Then $\kappa$ is a Dirac measure (otherwise there is a Borel set $A$ such that $\kappa(A)=:t \in (0,1)$, and then $\mu=t \nu_1 + (1-t)\nu_2$, where the probability measures $\nu_1= t^{-1}\int_A \nu \,\ud\kappa(\nu)$ and $\nu_2= (1-t)^{-1}\int_{\mc{P}(\Bo^{\db{\mb{N}}})\setminus A} \nu \,\ud\kappa(\nu)$ are cubic-exchangeable, contradicting that $\mu$ is an extreme point).  Hence $\mu$ has the independence property.
\end{proof}
\noindent Proposition \ref{prop:extremeequivs} entails the following counterpart for affine-exchangeable measures, involving the compact group nilspace $\mc{H}$ from Definition \ref{def:H}.
\begin{proposition}\label{prop:zeta-iff-extreme}
Let $\Bo$ be a standard Borel space, and let $\mu\in \mc{P}(\Bo^{\db{\mb{N}}})$. The following statements are equivalent.
\begin{enumerate}[leftmargin=0.8cm]
\item $\mu$ is an extreme point in the convex set $\textup{Pr}^{\textup{Aff}(\mb{F}_2^\omega)}(\Bo^{\db{\mb{N}}})$.
\item $\mu$ is in $\textup{Pr}^{\textup{Aff}(\mb{F}_2^\omega)}(\Bo^{\db{\mb{N}}})$ and satisfies the independence property.
\item  $\mu=\zeta_{\mc{H},m}$ for some Borel map $m:\mc{H}\to \mc{P}(\Bo)$.
\end{enumerate}
\end{proposition}
\begin{proof}
$(i)\Leftrightarrow (iii)$. If $\mu=\zeta_{\mc{H},m}$ then by the implication $(iii)\Rightarrow (i)$ in Proposition \ref{prop:extremeequivs} the measure $\mu$ is an extreme point among cubic-exchangeable measures, and is therefore an extreme point in $\textup{Pr}^{\textup{Aff}(\mb{F}_2^\omega)}(\Bo^{\db{\mb{N}}})$ (since every measure in the latter set is cubic-exchangeable). Conversely, suppose that $\mu$ in an extreme point in $\textup{Pr}^{\textup{Aff}(\mb{F}_2^\omega)}(\Bo^{\db{\mb{N}}})$. By Theorem \ref{thm:main} $\mu$ is a mixture of measures of the form $\zeta_{\mc{H},m}$. Since $\mu$ is an extreme point, arguing as in the proof of Proposition \ref{prop:extremeequivs} we see that the measure $\kappa$ underlying this mixture is a Dirac measure, so $\mu$ itself is equal to $\zeta_{\mc{H},m}$ for some Borel map $m:\mc{H}\to\mc{P}(\Bo)$.

$(i)\Leftrightarrow(ii)$. The backward implication here follows from the implication $(ii)\Rightarrow (i)$ in Proposition \ref{prop:extremeequivs}. Conversely, if $(i)$ holds then by the previous paragraph $(iii)$ holds, and then by the implication $(iii)\Rightarrow (ii)$ from Proposition \ref{prop:extremeequivs} we have that $\mu$ satisfies the independence property.
\end{proof}
\begin{remark}\label{rem:ergmeas}
Via the well-known characterization of ergodic measures as the extreme points of the space of $\Gamma$-invariant measures, Proposition \ref{prop:zeta-iff-extreme} also gives characterizations of the ergodic $\Gamma$-invariant measures in $\mc{P}(\Bo^{\db{\mb{N}}})$ for $\Gamma = \textup{Aff}(\mb{F}_2^\omega)$.
\end{remark}
We can now prove the main result of this section.
\begin{proof}[Proof of Theorem \ref{thm:BauerProp}]
Let $(\mu_n)_{n\in \mb{N}}$ be a sequence of extreme points in $\textup{Pr}^{\textup{Aff}(\mb{F}_2^\omega)}(\Bo^{\db{\mb{N}}})$ converging vaguely to $\mu\in \textup{Pr}^{\textup{Aff}(\mb{F}_2^\omega)}(\Bo^{\db{\mb{N}}})$. By $(i)\Leftrightarrow (ii)$ in Proposition \ref{prop:zeta-iff-extreme} and Lemma \ref{lem:indepclosed}, it follows that $\mu$ is also an extreme point in $\textup{Pr}^{\textup{Aff}(\mb{F}_2^\omega)}(\Bo^{\db{\mb{N}}})$.
\end{proof}

\begin{remark}
An alternative route to establish Theorem \ref{thm:BauerProp} could be to verify that our main representation result (Theorem \ref{thm:main}) implies that affine exchangeability satisfies a property called \emph{representability}, introduced by Austin in \cite{Austin2}, and then apply \cite[Proposition 3.2]{Austin2}, which states that representability implies the Bauer property. The  route followed in this section, using our previous results, was chosen in particular to take the opportunity to record Propositions \ref{prop:extremeequivs} and \ref{prop:zeta-iff-extreme}.
\end{remark}

\section{Correspondence between representations of affine-exchangeable measures and limits for convergent sequences of functions}\label{sec:exch-limits}

\noindent In this section we show that finding an appropriate integral representation for affine-exchangeable measures is equivalent to finding limit objects for convergent sequences of functions, for a concept of convergence that we recall in Definition \ref{def:convergence-notion} below. This concept is an analogue in arithmetic combinatorics of a well-known convergence notion for sequences of graphs studied in many works (see in particular \cite{BCLSV,LS06}).

Throughout this section let $\Gamma$ denote the group $\textup{Aff}(\mb{F}_2^\omega)$ of invertible affine maps $\mb{F}_2^\omega\to \mb{F}_2^\omega$, and recall that $\Gamma\cong \text{GL}(\mb{F}_2^{\omega})\ltimes \mb{Z}_2^{\omega}$. Given any function $f$ on $\mb{F}_2^{\omega}$ and any $T\in \Gamma$, we use the usual ergodic-theory notation $Tf$ to denote the composition $f\co T$.

The main result of this section, Theorem \ref{thm:corresp} below, establishes an equivalence between two properties, the first of which is the following.
\begin{defn}\label{def:repX}
Let $\Bo$ be a standard Borel space and let $\ns$ be a compact profinite-step $2$-homogeneous nilspace. We say that $\ns$ \emph{represents} (or \emph{is a representing nilspace} for) $\textup{Pr}^{\textup{Aff}(\mb{F}_2^\omega)}(\Bo^{\db{\mb{N}}})$ if for every 
 $\mu\in \textup{Pr}^{\textup{Aff}(\mb{F}_2^\omega)}(\Bo^{\db{\mb{N}}})$ with the independence property, there exists a Borel map $m:\ns\to \mc{P}(\Bo)$ such that $\mu=\zeta_{\ns,m}$.
\end{defn}
\noindent The second property involves the above-mentioned notion of convergence of functions. This notion has been studied in previous works, notably in \cite{HHH} in the special case of Boolean functions (taking values in $\{0,1\}$). We shall define this notion for functions taking values in any fixed compact metric space. To this end we shall use the following terminology, much of which is taken from \cite{HHH} (rephrasing some of it using nilspace theory).

Given a finite set $\mc{L}\subset \mb{F}_2^{\omega}$, clearly there exists $k\in \mb{N}$ such that $\mc{L}\subset \mb{F}_2^k\times \{0^{\mb{N}\setminus [k]}\}$. We shall sometimes abuse the notation by identifying $\mb{F}_2^k$ with the set $\mb{F}_2^k\times \{0^{\mb{N}\setminus [k]}\}\subset \mb{F}_2^{\omega}$. Note that the set of affine linear maps $A:\mb{F}_2^k\to\mb{F}_2^n$ with pointwise addition is isomorphic to the abelian group of $k$-cubes $\cu^k(\mb{F}_2^n)$, and we can thereby define a \emph{random} affine-linear map $A$ using the Haar measure $\mu_{\cu^k(\mb{F}_2^n)}$ (which here is just the normalized counting measure); we can thereby also apply an element $A\in \cu^k(\mb{F}_2^n)$ as an affine-linear map to elements $L=(\lambda_1,\ldots,\lambda_k)\in \mb{F}_2^k$. Indeed, for some coefficients $a_0,a_1,\ldots,a_k\in \mb{F}_2^n$ uniquely associated with $A$ we have $A(L)=a_0+\lambda_1a_1+ \cdots + \lambda_k a_k$ for every such element $L$. We can then define the following measures, which underpin the upcoming notion of convergence.
\begin{defn}
Let $\mc{L}$ be a finite subset of $\mb{F}_2^\omega$ and let $k\in\mb{N}$ be such that $\mc{L}\subset \mb{F}_2^k\times \{0^{\mb{N}\setminus [k]}\}$. Let $\Bo$ be a compact metric space, and let  $f:\mb{F}_2^n\to\Bo$. We define the probability measure $\mu_{\mc{L},f}\in\mc{P}(\Bo^{\mc{L}})$ as the pushforward of $\mu_{\cu^k(\mb{F}_2^n)}$ under the map $A\mapsto \big(f\co A(L)\big)_{L\in\mc{L}}$, that is
\begin{equation}\label{eq:samplemeas}
\forall \,\textrm{Borel set }S\subset \Bo^{\mc{L}},\quad \mu_{\mc{L},f}(S):=\frac{\left| \{A\in \cu^k(\mb{F}_2^n): \, (f\co A(L))_{L\in\mc{L}}\in S\} \right|}{|\cu^k(\mb{F}_2^n)|}.
\end{equation}
\end{defn}
\noindent It follows from the consistency axiom for the cubic coupling $(\mu_{\cu^k(\mb{F}_2^n)})_{k\ge 0}$ that, given any such finite set $\mc{L}\subset \mb{F}_2^\omega$, the definition of $\mu_{\mc{L},f}$ is independent of $k$ provided that $k$ is large enough so that $\mc{L}\subset  \mb{F}_2^k\times \{0^{\mb{N}\setminus [k]}\}$. When $S$ is a singleton $\{g\}$ we shall simplify the notation by writing $\mu_{\mc{L},f} (g)$ instead of $\mu_{\mc{L},f}(\{g\})$.
\begin{defn}\label{def:convergence-notion} 
Let  $\Bo$ be a compact metric space. We say that a sequence of functions $(f_n:\mb{F}_2^n\to \Bo)_{n\in \mb{N}}$ is \emph{convergent} if for every finite set $\mc{L}\subset \mb{F}_2^{\omega}$ the measures $\mu_{\mc{L},f_n}$ converge in the vague topology as $n\to\infty$.
\end{defn}
\noindent Note that the special case $\Bo=\{0,1\}$ agrees with the case $d=\infty$ of  \cite[Definition 3.1]{HHH}. 

\begin{remark}
Each point $(\lambda_1,\ldots,\lambda_k)\in \mb{F}_2^k$ can be viewed as a linear form $\mb{F}_2^k\to\mb{F}_2$, $x\mapsto \lambda_1x_1+\cdots+\lambda_kx_k$. Following \cite{HHH} we  say that this is an \emph{affine} linear form if $\lambda_1=1$. For every linear form $L=(\lambda_1,\ldots,\lambda_k)$ we can define the affine linear form $\widetilde L=(1,\lambda_1,\ldots,\lambda_k)\in \mb{F}_2^{k+1}$. 
From this viewpoint, in the case of Boolean functions the above notion of convergence can be seen to be equivalent to the following one, familiar in arithmetic combinatorics, and phrased in terms of averages over systems of affine linear forms: a sequence $(f_n:\mb{F}_2^n\to \{0,1\})_{n\in \mb{N}}$ is convergent if for every $k\in \mb{N}$, for every $\mc{L}=\{L_1,\ldots,L_r\}\subset \mb{F}_2^k$, the following averages converge as $n\to\infty$:
\begin{eqnarray*}
\Lambda_\mc{L}(f_n) & := &\mb{E}_{a_0,a_1,\ldots,a_k\in \mb{F}_2^n}\, f_n(a_0+L_1(a_1,\ldots,a_k))\cdots  f_n(a_0+L_r(a_1,\ldots,a_k))\\
& = &\mb{E}_{a_0,a_1,\ldots,a_k\in \mb{F}_2^n}\, f_n(\widetilde{L}_1(a_0,\ldots,a_k))\cdots  f_n(\widetilde{L}_r(a_0,\ldots,a_k)).
\end{eqnarray*}
One implication in the equivalence follows from the fact that $\Lambda_\mc{L}(f_n)=\sum_{g\in \{0,1\}^\mc{L}} \mu_{\mc{L},f_n}(g)$, and the opposite implication follows from the fact that each $\mu_{\mc{L},f_n}(g)$ is itself a linear combination of averages $\Lambda_{\mc{L}'}(f_n)$, as can be shown using Fourier analysis (see e.g.\ \cite[Observation 4.1]{HHH}).
\end{remark}

\noindent Given any finite set $\mc{L}\subset \mb{F}_2^\omega$, we denote by $p_{\mc{L}}$ the coordinate projection $\Bo^{\mb{F}_2^\omega}\to \Bo^{\mc{L}}$.

The second property involved in the main result of this section is the following.
\begin{defn}\label{def:limdom} 
Let $\ns$ be a compact profinite-step $2$-homogeneous nilspace, and let $\Bo$ be a compact metric space. We say that $\ns$ is a \emph{limit domain} for convergent sequences $(f_n:\mb{F}_2^n \to \Bo)_{n\in\mb{N}}$ if for every such sequence there exists a Borel map $m:\ns\to \mc{P}(\Bo)$ such that for every finite set $\mc{L}\subset \mb{F}_2^{\omega}$ the measures $\mu_{\mc{L},f_n}$ converge vaguely to the pushforward of $\zeta_{\ns,m}$ under $p_{\mc{L}}$, i.e.\ to the measure defined by
\begin{equation}\label{eq:limdom}
\zeta_{\ns,m}\co p_{\mc{L}}^{-1} (S)= \int_{\cu^\omega(\ns)} \Big(\prod_{L\in \mc{L}}m(\q(L))\Big)(S)\,\ud\mu_{\cu^\omega(\ns)}(\q),\;\textrm{ for any Borel set }S\subset \Bo^{\mc{L}}.
\end{equation}
\end{defn}

\begin{remark}
It can be seen that $\ns$ being a limit domain for convergent sequences of \emph{Boolean} functions $(f_n:\mb{F}_2^n \to \{0,1\})_{n\in\mb{N}}$ is equivalent to any such sequence converging to an $\infty$-\emph{limit object} $\Gamma:\ns\to [0,1]$ in the sense of \cite[Definition 3.2]{HHH} for $\ns=\mb{G}_\infty$ (indeed this limit object can be seen to correspond to the above Borel map $m:\ns\to\mc{P}(\{0,1\})$). For more details on this correspondence see Appendix \ref{app:correspondence}.
\end{remark}

We are now ready to state and prove the main result of this section.

\begin{theorem}\label{thm:corresp}
Let $\ns$ be a compact profinite-step $2$-homogeneous nilspace, and let $\Bo$ be a compact metric space. The following statements are equivalent:
\begin{enumerate}[leftmargin=0.8cm]
\item $\ns$ is a representing nilspace for  $\textup{Pr}^{\textup{Aff}(\mb{F}_2^\omega)}(\Bo^{\db{\mb{N}}})$.
\item $\ns$ is a limit domain for convergent sequences of functions $(f_n:\mb{F}_2^n \to \Bo)_{n\in\mb{N}}$.
\end{enumerate}
\end{theorem}
\begin{proof}
Throughout the proof we use the fact that the group $\Gamma:=\textup{Aff}(\mb{F}_2^\omega)\cong\text{GL}(\mb{F}_2^{\omega})\ltimes \mb{Z}_2^{\omega}$ can be regarded as $\bigcup_{n=1}^{\infty} \Gamma_n$ for $\Gamma_n:=\textup{Aff}(\mb{F}_2^n)\cong \text{GL}(\mb{F}_2^n)\ltimes \mb{Z}_2^{n}$, where $\Gamma_n$ acts on the first $n$ coordinates of any element of $\mb{F}_2^{\omega}$. We split the proof into two parts.

\smallskip

\textbf{Proof of $(ii)\Rightarrow(i)$.} 
Let $\mu\in \textup{Pr}^{\textup{Aff}(\mb{F}_2^\omega)}(\Bo^{\db{\mb{N}}})$ satisfy the independence property. We shall define a convergent sequence $(f_n:\mb{F}_2^n\to\Bo)_{n\in \mb{N}}$ such that for every finite set $\mc{L}\subset \mb{F}_2^\omega$ the measures $\mu_{\mc{L},f_n}$ converge vaguely to $\mu\co p_{\mc{L}}^{-1}$. On the other hand, our assumption of $(ii)$ will give us a Borel map $m:\ns\to\mc{P}(\Bo)$ (dependent on this sequence $(f_n)$), such that  $\mu_{\mc{L},f_n}$ converges vaguely to $\zeta_{\ns,m}\co p_{\mc{L}}^{-1}$. This will imply that $\mu\co p_{\mc{L}}^{-1}=\zeta_{\ns,m}\co p_{\mc{L}}^{-1}$ for every finite such set $\mc{L}$, whence we will deduce that $\mu=\zeta_{\ns,m}$, confirming $(i)$ as required.

To define the desired sequence $(f_n)$ we use the pointwise ergodic theorem for amenable groups \cite[Theorem 1.2]{Lind}. For any function $t:\Bo^{\mb{F}_2^{\omega}}\to \mb{C}$ let us define the following ergodic average for each $n\in \mb{N}$:
\begin{equation}\label{eq:ergavs}
\mc{E}_n(t)\; :\; f\; \mapsto \; \mb{E}_{A\in \Gamma_n} t(A f) =  \frac{1}{|\Gamma_n|}\sum_{A\in \Gamma_n} t(f\co A).
\end{equation}
The space of continuous functions $\Bo^{\mb{F}_2^{\omega}}\to \mb{C}$ equipped with the supremum norm is separable, so we can pick a dense sequence  $(t_i)_{i\in \mb{N}}$ in this space. By the ergodic theorem \cite[Theorem 1.2]{Lind}, and the ergodicity of $\mu$ (which follows from the independence property, by Proposition \ref{prop:zeta-iff-extreme} and Remark \ref{rem:ergmeas}), for every fixed $i\in \mb{N}$ we have $\mc{E}_n(t_i)(f) \to \int t_i\,\ud\mu$ for $\mu$-almost every $f\in \Bo^{\mb{F}_2^{\omega}}$. 
Using this for each $i$, we deduce that there is a set $D\subset \Bo^{\mb{F}_2^{\omega}}$ with $\mu(D)=1$ such that for every $f\in D$ we have $ \lim_{n\to\infty}\mc{E}_n(t_i)(f) = \int t_i\,\ud\mu$ for all $i\in \mb{N}$. From the assumed density of $(t_i)_{i\in \mb{N}}$, it follows that
\begin{equation}\label{eq:ergconv}
\forall\,f\in D, \;\forall\,\textrm{continuous function }t:\Bo^{\mb{F}_2^\omega} \to \mb{C},\quad \lim_{n\to\infty}\mc{E}_n(t)(f)\,=\, \int t\,\ud\mu.
\end{equation} 
Fix any $f'\in D$. For each $n\in \mb{N}$ we define $f_n:\mb{F}_2^n\to \Bo$ by $f_n(v)= f'(v,0^{\mb{N}\setminus[n]})$. We claim that for any finite $\mc{L}\subset \mb{F}_2^{\omega}$ we have $\mu_{\mc{L},f_n}\to \mu\co p_{\mc{L}}^{-1}$ in the vague topology (in particular $(f_n)_{n\in\mb{N}}$ is convergent). To see this, fix any such $\mc{L}$, and any $k$ such that $\mc{L}\subset \mb{F}_2^k$. 

Fix any continuous function $t:\Bo^{\mc{L}}\to \mb{C}$, and observe that for $n\ge k$ we have $\mc{E}_n(t\co p_{\mc{L}})(f')=\frac{1}{|\Gamma_n|}\sum_{A\in \Gamma_n} t\co p_{\mc{L}}(f'\co A)= \frac{1}{|\Gamma_n|}\sum_{A\in \Gamma_n} t(p_{\mc{L}}(f_n\co A))$ where in the last equation we are abusing the notation, treating $p_{\mc{L}}$ as a map $\Bo^{\db{n}}\to\Bo^{\mc{L}}$ and $A$ as an element of $\Aff(\mb{F}_2^n)$. Let $p_k: \Gamma_n\to \{A':\mb{F}_2^k\to\mb{F}_2^n~|~ A'\textrm{ is  affine linear and injective}\}$ be the projection sending $A\in \Gamma_n$ to the map $A'(v):=A(v,0^{n-k})$. Note that $p_k$ is surjective and in fact for each injective affine map $A'\in \cu^k(\mb{F}_2^n)$ the preimage $p_k^{-1}(A')$ has the same cardinality $\prod_{j=k}^{n-1}(2^n-2^j)$. (Indeed if  $A'(v_1,\ldots,v_k)=a_0+a_1v_1+\cdots+a_kv_k$ is injective, the linear span of $\{a_1,\ldots,a_k\}$ has dimension $k$, whence, to extend $A'$ to a map $A\in \Gamma_n$ of the form $A(v_1,\ldots,v_n)=a_0+a_1v_1+\cdots+a_kv_k+a_{k+1}v_{k+1}+\cdots+a_nv_n$, we have $2^n-2^k$ choices for $a_{k+1}$, then $2^n-2^{k+1}$ choices for $a_{k+2}$, etc.). Also, if $p_k(A)=p_k(B)$ then $p_{\mc{L}}(f_n\co A)=p_{\mc{L}}(f_n\co B)$. Hence
\[
\mc{E}_n(t\co p_{\mc{L}})(f')=\frac{1}{|p_k(\Gamma_n)|}\sum_{A'\in p_k(\Gamma_n)} t(p_{\mc{L}}(f_n\co A')).
\]
Now, while $p_k(\Gamma_n)$ is not exactly $\cu^k(\mb{F}_2^n)$ (not every affine map in the latter set is injective), we do have 
\begin{equation}\label{eq:proof-eq-2}
\frac{|p_k(\Gamma_n)|}{|\cu^k(\mb{F}_2^n)|} = \frac{2^n\prod_{i=1}^k(2^n-2^{i-1})}{2^{n(k+1)}}= \prod_{i=1}^k\left(1-\frac{1}{2^{n-i+1}}\right)= 1+o_k(1)_{n\to \infty}.
\end{equation}
Indeed, to choose $A'\in p_k(\Gamma_n)$ we have  $2^n$ choices for $a_0$, $2^n-1$ choices for $a_1$ (all but $0^n$), $2^n-2$ choices for $a_2$ (all but the subspace generated by $a_1$), and so on. Hence
\begin{multline}
\mc{E}_n(t\co p_{\mc{L}})(f') =  \frac{1}{|p_k(\Gamma_n)|}\sum_{A'\in \cu^k(\mb{F}_p^n)} t(p_{\mc{L}}(f_n\co A'))-\frac{1}{|p_k(\Gamma_n)|}\sum_{\substack{A'\in \cu^k(\mb{F}_p^n)\\A' \text{ not injective}}} t(p_{\mc{L}}(f_n\co A'))\\
= (1+o_k(1)_{n\to \infty})\int t \;\ud\mu_{\mc{L},f_n}+\|t\|_{\infty}\,o_k(1)_{n\to \infty}.
\end{multline}
By \eqref{eq:ergconv} the left hand side here converges to $\int t\co p_{\mc{L}}\; \ud\mu$ as $n\to\infty$, whence so does the right hand side. This proves that $\mu_{\mc{L},f_n}$ converges vaguely $\mu\co p_{\mc{L}}^{-1}$, as we claimed. By our assumption that $\ns$ is a limit domain, there exists a Borel map $m:\ns\to \mc{P}(\Bo)$ such that $\mu_{\mc{L},f_n}$ converges vaguely to $\zeta_{\ns,m}\co p_{\mc{L}}^{-1}$, so $\mu\co p_{\mc{L}}^{-1} = \zeta_{\ns,m}\co p_{\mc{L}}^{-1}$. Since the projections $p_{\mc{L}}$ (for finite sets $\mc{L}$) generate the Borel $\sigma$-algebra on $\Bo^{\mb{F}_2^\omega}$, we deduce that $\mu=\zeta_{\ns,m}$. 

\textbf{Proof of $(i)\Rightarrow (ii)$.} Suppose that $\ns$ represents $\textup{Pr}^{\textup{Aff}(\mb{F}_2^\omega)}(\Bo^{\db{\mb{N}}})$, and let $(f_n:\mb{F}_2^n\to\Bo)_{n\in \mb{N}}$ be a convergent sequence. We will show that the limits of the measures $\mu_{\mc{L},f_n}$ for finite sets $\mc{L}\subset \mb{F}_2^\omega$ enable the definition of a measure $\mu\in\textup{Pr}^{\textup{Aff}(\mb{F}_2^\omega)}(\Bo^{\db{\mb{N}}})$ with the independence property and such that $\mu\co p_{\mc{L}}^{-1}$ is the vague limit of  $\mu_{\mc{L},f_n}$ for every such $\mc{L}$, whence $(ii)$ will follow from the representation $\mu=\zeta_{\ns,m}$ given by $(i)$.

By assumption, for every finite set $\mc{L}\subset \mb{F}_2^\omega$, the sequence $(\mu_{\mc{L},f_n})_{n\in\mb{N}}$ converges vaguely to a measure in $\mc{P}(\Bo^{\mc{L}})$, which we shall denote by $\mu_{\mc{L}}$. We claim that there exists a measure $\mu\in \mc{P}(\Bo^{\mb{F}_2^\omega})$ such that for any finite $\mc{L}$ we have $\mu\co p_{\mc{L}}^{-1} = \mu_{\mc{L}}$. To prove this we shall define a continuous linear functional $\wh{\mu}$ on $C(\Bo^{\mb{F}_2^\omega})$,  and then apply the Riesz representation theorem \cite[p.\ 464]{R&F} to identify $\wh{\mu}$ with a measure $\mu\in \mc{P}(\Bo^{\mb{F}_2^\omega})$ as required ($\wh{\mu}$ will be the integral with respect to $\mu$). For any finite $\mc{L}\subset \mb{F}_2^\omega$ and continuous $t:\Bo^{\mc{L}}\to \mb{C}$ we define $\wh{\mu}(t\co p_{\mc{L}}):=\int t \,\ud\mu_{\mc{L}}$. Note that if $t\co p_{\mc{L}}=t'\co p_{\mc{L}'}$ then $\int t\,\ud\mu_{\mc{L}}=\int t'\,\ud\mu_{\mc{L}'}$, so $\wh{\mu}$ is well-defined. For any continuous $t:\Bo^{\mb{F}_2^\omega}\to\mb{C}$, by the Stone-Weierstrass theorem there is a sequence of pairs $(\mc{L}_n,f_n)$ where $\mc{L}_n$ are finite subsets of $\mb{F}_2^\omega$ and $t_n:\Bo^{\mc{L}_n}\to\mb{C}$, such that $\|t-t_n\co p_{\mc{L}_n}\|_{\infty}<1/n$ for every $n$. It follows that the sequence $(\int t_n \,\ud\mu_{\mc{L}_n})_{n\in\mb{N}}$ is Cauchy and thus we can define $\wh{\mu}(t):=\lim_{n\to\infty} \int t_n \,\ud\mu_{\mc{L}_n}$ (this limit is easily seen to be the same for every such sequence $(t_n\co p_{\mc{L}_n})_{n\in \mb{N}}$ converging to $t$). The operator $\wh{\mu}$ is easily seen to be linear and bounded  (in fact it has norm 1), so we conclude that $\wh{\mu}$ is an element of the dual of $C(\Bo^{\mb{F}_2^\omega})$, and we obtain a measure $\mu$ as announced, with the property that $\wh{\mu}(f)=\int f \ud\mu$ for every $f\in C(\Bo^{\mb{F}_2^\omega})$ (in particular $\mu$ is a probability measure).

To prove that $\mu$ is affine-exchangeable, since the functions of the form $t\co p_{\mc{L}}$ (for finite $\mc{L}\subset \mb{F}_2^{\omega}$ and continuous $t:\Bo^{\mc{L}}\to \mb{C}$) are dense in the space of continuous functions $\Bo^{\mb{F}_2^\omega}\to\mb{C}$ by the Stone-Weierstrass theorem, it suffices to prove that for every such function and every $A\in \Aff(\mb{F}_2^\omega)$ we have $\int t\co p_{\mc{L}}\ud\mu = \int t\co p_{\mc{L}}\co \theta_A \,\ud \mu$,  where $\theta_A$ is the coordinate permutation on $\Bo^{\mb{F}_2^\omega}$ induced by $A$. To prove this, we shall use the following notation: given $f:\mb{F}_2^\omega\to \Bo$, and any fixed $k$ such that $\mc{L}\subset \mb{F}_2^k\times 0^{\mb{N}\setminus [k]}$, we define $T_{f,\mc{L},n}:\cu^k(\mb{F}_2^n)\to \Bo^{\mc{L}}$ to be the map $A\mapsto (f\co A)|_{\mc{L}}$, and for $A\in \Aff(\mb{F}_2^k)$ we define the map $\phi_{\mc{L},A}:\Bo^{\mc{L}}\to \Bo^{A\mc{L}}$ by $\phi_{\mc{L},A}(r)(A(L)):=r(L)$ for any $L\in\mc{L}$ (where $A\mc{L}:=\{A(L):L\in \mc{L}\}$).

Now fix any continuous function $t:\Bo^{\mc{L}}\to\mb{C}$ and any $A\in \Aff(\mb{F}_2^\omega)$. First note that $\mu_{\mc{L},f} = \mu_{\cu^k(\mb{F}_2^n)}\co T_{f,\mc{L},n}^{-1}$. If $k$ is large enough so that $\mc{L}\subset \mb{F}_2^k\times 0^{\mb{N}\setminus[k]}$ and $A$ can be viewed as an element of $\Aff(\mb{F}_2^k)$, then $\int t\co p_{\mc{L}}\,\ud\mu = \lim_{n\to\infty} \int t\,\ud\mu_{\mc{L},f_n} = \lim_{n\to\infty} \int t\co T_{f_n,\mc{L},n}\,\ud\mu_{\cu^k(\mb{F}_2^n)}$. The map $\cu^k(\mb{F}_2^n)\to \cu^k(\mb{F}_2^n)$, $\q\mapsto \q\co A$ is a surjective homomorphism, so it preserves the Haar measure. Hence $\int t\co T_{f_n,\mc{L},n}(\q)\,\ud\mu_{\cu^k(\mb{F}_2^n)}(\q)= \int t\co T_{f_n,\mc{L},n}(\q\co A)\,\ud\mu_{\cu^k(\mb{F}_2^n)}(\q)$. By definition $T_{f_n,\mc{L},n}(\q\co A) = (f_n\co \q\co A)|_{\mc{L}} = \phi_{\mc{L},A}^{-1}((f_n\co \q)|_{A\mc{L}})$. Hence we have
$\int t \co p_{\mc{L}} \,\ud\mu = \lim_{n\to\infty} \int t\co \phi_{\mc{L},A}^{-1}((f_n\co \q)|_{A\mc{L}})\,\ud\mu_{\cu^k(\mb{F}_2^n)}(\q)$, and this equals $\lim_{n\to\infty} \int t\co \phi_{\mc{L},A}^{-1}\co p_{A\mc{L}} \; \ud\mu_{A\mc{L},f_n}= \int (t\co \phi_{\mc{L},A}^{-1})\co  p_{A\mc{L}}\,\ud\mu  = \int  t\co p_{\mc{L}}\co \theta_A\ud\mu$ as required.

Finally, we prove that $\mu$ satisfies the independence property. Similarly as in the proof of Lemma \ref{lem:indepclosed}, it suffices to prove that for any pair of finite independent faces $\mc{L}_1,\mc{L}_2\subset \mb{F}_2^\omega$ and continuous functions $t_i:\Bo^{\mc{L}_i}\to \mb{C}$ for $i=1,2$ we have
\[
\int (t_1\co p_{\mc{L}_1})(t_2\co p_{\mc{L}_2})\,\ud\mu = \int (t_1\co p_{\mc{L}_1})\,\ud\mu \int (t_2\co p_{\mc{L}_2})\,\ud\mu.
\]
Fix any $k$ such that $\mc{L}_1\sqcup\mc{L}_2\subset \mb{F}_2^k$. Then
\[
\int (t_1\co p_{\mc{L}_1})(t_2\co p_{\mc{L}_2})\;\ud\mu = \lim_{n\to\infty}\;  \mb{E}_{A\in \cu^k(\mb{F}_2^n)} (t_1\co T_{f_n,\mc{L}_1,n})(t_2\co T_{f_n,\mc{L}_2,n}).
\]
Fix any $n\in \mb{N}$. Then this last average equals
\begin{equation}\label{eq:equi-2}
\mb{E}_{A\in \cu^k(\mb{F}_2^n)}\; t_1(f_n\co A|_{\mc{L}_1}) \;t_2(f_n\co A|_{\mc{L}_2}).
\end{equation}
For $i=1,2$ we have that $(f_n\co A)|_{\mc{L}_i}$ is a vector indexed by $L\in\mc{L}_i$ with values $f_n(A(L))$.

Note that there exists two finite faces $F_1$ and $F_2$ such that $\mc{L}_i\subset F_i$ for $i=1,2$, $F_1\cap F_2=\{K\}$ for some $K\in \mb{F}_2^{\omega}$ and such that $\langle F_1\cup F_2\rangle = \mb{F}_2^k$, $K\in \mc{L}_1$ and $K\not\in \mc{L}_2$. Furthermore, by applying if necessary an element of $\Gamma_k$ (which preserves finiteness and independence of faces), we can assume that $K=0^{\omega}$, $F_1=\{0,1\}^{d_1}\times \{0^{\mb{N}\setminus[d_1]}\}$, $F_2=\{0^{d_1-1}\}\times \{r\} \times \{0,1\}^{d_2}\times \{0^{\mb{N}\setminus [d_1+d_2]}\}$ (for some $r\in \mb{F}_2$). Clearly $k=d_1+d_2$. Using the bijection between $\cu^k(\mb{F}_2^n)$ and $\{(x,a_1,\ldots,a_{d_1+d_2})\in \mb{F}_2^n\}$ yielded by the expression $A(v)=x+a_1v_1+\cdots a_{d_1+d_2}v_{d_1+d_2}$, the average in \eqref{eq:equi-2} can then be written as follows:
\[
\mb{E}_{\substack{x,a_1,\ldots,a_{d_1}\in\mb{F}_2^n\\a_{d_1+1},\ldots,a_{d_1+d_2}\in \mb{F}_2^n}} \;  t_1(f_n\co A|_{\mc{L}_1}) \;t_2(f_n\co A|_{\mc{L}_2}).
\]
Since the elements of $\mc{L}_1$ have coordinates $v_{d_1+1},\ldots,v_{d_2}$ all 0, the first term is independent of the variables $a_{d_1+1},\ldots,a_{d_1+d_2}$, and thus the average equals
\begin{multline}\label{eq:sep-variables}
\mb{E}_{ x,a_1,\ldots,a_{d_1}\in\mb{F}_2^n}\; t_1\big( (f_n (x+a_1v_1^i+\cdots+a_{d_1}v_{d_1}^i))_{i\in[I]}\big) \\
\;\cdot \; \mb{E}_{\substack{a_{d_1+1},\ldots,a_{d_1+d_2}\in \mb{F}_2^n}} \; t_2\big(f_n (x+a_{d_1+1}v_{d_1+1}^j+\cdots+a_{d_1+d_2}v_{d_1+d_2}^j))_{j\in[J]}\big),
\end{multline}
where we assume that $\mc{L}_1 = \{(v_1^i,\ldots,v_{d_1}^i,0^{d_2}):i\in [I]\}$ for some $I\ge 0$ and similarly $\mc{L}_2=\{(0^{d_1},v_{d_1+1}^j,\ldots,v_{d_1+d_2}^j):j\in [J]\}$ for some $J\ge 0$. 

Now note that since $\mc{L}_2\not\ni 0^k$, there must be some $s\in \{d_1+1,\ldots,d_1+d_2\}$ such that for all $j\in [J]$ we have that $v_s^j$ is a non-zero constant $c$ \emph{independent of $j\in [J]$}. Hence, in \eqref{eq:sep-variables} we can eliminate the variable $x$ in the second function $f_n$, by the simple change of variables $a_s\mapsto a_s- c^{-1}x$, thus showing that this average equals
\[
\mb{E}_{ a_{d_1+1},\ldots,a_{d_1+d_2}\in \mb{F}_2^n}(t_2(f_n (a_{d_1+1}v_{d_1+1}^j+\cdots+a_{d_1+d_2}v_{d_1+d_2}^j))_{j\in[J]}).
\]
We can now change again the variable $a_s$ to $a_s+y$ for $y\in \mb{F}_2^n$, and add an averaging over $y\in \mb{F}_2^n$, thus proving that this last average equals $\int t_2 \,\ud \mu_{\mc{L}_2,f_n}$. Taking this constant out of the average in \eqref{eq:sep-variables} we conclude that \eqref{eq:sep-variables} equals $\int t_1 \,\ud \mu_{\mc{L}_1,f_n}\,\int t_2 \,\ud \mu_{\mc{L}_2,f_n}$. Since this holds for each $n$, the result follows by letting $n\to\infty$.
\end{proof}

\begin{remark}\label{rem:finite-complexity-limit}
Theorem \ref{thm:corresp} concerns a notion of convergence involving systems of linear forms which can be of  \emph{arbitrary finite complexity} (in the sense of true complexity from \cite{GW}).  In many situations in higher-order Fourier analysis, we are only interested in working with systems of complexity at most some prescribed finite bound. In this case, it turns out that the correct notion of complexity is the true complexity of $\widetilde{\mc{L}}:=\{\widetilde{L}=(1,L) : L\in \mc{L}\}$. Using the compactness of the set $\mc{P}(\Bo^{\mc{L}})$, a diagonalization argument, Theorem \ref{thm:corresp}, and Theorem \ref{thm:main} the following result can be proved. Fix any $k\in \mb{N}$. Let $(f_n:\mb{F}_2^n\to \Bo)_{n\in \mb{N}}$ be a sequence of functions for some compact metric space $\Bo$ such that $\mu_{\mc{L},f_n}$ converges vaguely for all finite $\mc{L}\subset \mb{F}_2^\omega$ such that the true complexity of $\widetilde{\mc{L}}$ is at most $k$. Then there exists a measurable function $m:\mc{H}\to \mc{P}(\Bo)$ such that the following holds. For any finite $\mc{L}\subset \mb{F}_2^\omega$ such that the true complexity of $\widetilde{\mc{L}}$ at most $k$, the measures $\mu_{\mc{L},f_n}$ converge vaguely to $\zeta_{\mc{H},m}\co p_{\mc{L}}^{-1}$.

This result can be further refined in the following sense. Note that $\mc{H}$ has infinite step, whereas we only required convergence for linear forms \emph{up to some finite complexity} $k$. Using an equidistribution theorem (that we omit in this paper), any average over $\zeta_{\mc{H},m}\co p_{\mc{L}}^{-1}$ can be written as the limit when $n\to\infty$ of some average over $\cu^s(\mb{F}_2^n)$ (where $\mc{L}\subset \mb{F}_2^s\times \{0\}^{\mb{N}\setminus [s]}$ for some $s\in\mb{N}$). Then, using the definition of true complexity it can be proved that if the true complexity of $\widetilde{\mc{L}}$ is at most $k$ (for some $\mc{L}\subset \mb{F}_2^\omega$) then $\zeta_{\mc{H},m}\co p_{\mc{L}}^{-1}$ equals $\zeta_{\mc{H}/\mc{H}_{(k+1)},m_k}\co p_{\mc{L}}^{-1}$ where $m_k:\mc{H}/\mc{H}_{(k+1)}\to \mc{P}(B)$ is given by the formula $m_k(\pi_k(x)):=\int_{\pi_k^{-1}(\pi_k(x))} m(y)\;\ud\mu_{\pi_k^{-1}(\pi_k(x))}$ and $\pi_k:\mc{H}\to \mc{H}/\mc{H}_{(k+1)}$ is the quotient map\footnote{Recall that $\mc{H}$ is a filtered group with filtration $\mc{H}_{(j)}$ for $j\in\mb{N}$ by Definition \ref{def:H}.}. 
\end{remark}

\begin{remark}\label{rem:extensionto-p}
Affine exchangeability (or, more precisely, $\mb{F}_2^\omega$-affine exchangeability) has the following natural analogue for any prime $p>2$: a measure $\mu\in \mc{P}(\Bo^{\mb{F}_p^\omega})$ is $\mb{F}_p^\omega$\emph{-affine-exchangeable} if $\mu$ is invariant under the coordinate permutations induced by the group $\textrm{Aff}(\mb{F}_p^\omega)\cong \textrm{GL}(\mb{F}_p^\omega)\ltimes \mb{Z}_p^\omega$. As mentioned at the end of the introduction, the main results of this paper have counterparts for $\mb{F}_p^\omega$-affine exchangeability, but proving these requires extending previous work. In particular, the main results concerning cubic couplings from \cite{CScouplings} would need to be extended to product spaces indexed by $\mb{F}_p^\omega$, and these would have to be combined with corresponding extensions of the results from Section \ref{sec:2-homCC} in this paper, to obtain structural results involving $p$-homogeneous nilspaces. This is beyond the scope of the present paper.
\end{remark}

\section{On representing affine-exchangeability with nilspaces}\label{sec:repvsnils}
\noindent  Recall from Definition \ref{def:repX} the notion of a nilspace $\ns$ \emph{representing}  $\textup{Pr}^{\textup{Aff}(\mb{F}_2^\omega)}(\Bo^{\db{\mb{N}}})$ for \emph{some} standard Borel space $\Bo$. In this section we shall say that $\ns$ \emph{represents affine-exchangeability} if $\ns$ \emph{represents} $\textup{Pr}^{\textup{Aff}(\mb{F}_2^\omega)}(\Bo^{\db{\mb{N}}})$ for \emph{every} standard Borel space $\Bo$. The main result of this subsection is Theorem \ref{thm:repre-with-2-hom}, which we recall here for convenience, establishing counterparts of this notion of representation that concern nilspace theory and limit domains.

\vspace{0.4cm}
\noindent \textbf{Theorem \ref{thm:repre-with-2-hom}}
\emph{Let $\ns$ be a compact profinite-step 2-homogeneous nilspace. The following statements are equivalent.
\begin{enumerate}[leftmargin=0.8cm]
\item $\ns$ represents affine-exchangeability.
\item For every compact 2-homogeneous profinite-step nilspace $\nss$ there is a \textup{(}continuous\textup{)} fibration $\varphi:\ns\to \nss$.
\item $\ns$ is a limit domain for convergent sequences $(f_n:\mb{F}_2^n\to\Bo)_{n\in \mb{N}}$, for every compact metric space $\Bo$.
\end{enumerate}}

\noindent From previous main results in this paper (specifically Theorems \ref{thm:maincharac} and \ref{thm:universal-covering}) we know that the nilspace $\mc{H}$ from Definition \ref{def:H} is an example of a nilspace $\ns$ satisfying properties $(i)$ and $(ii)$ above.

To prove Theorem \ref{thm:repre-with-2-hom}, we start by studying a class of examples announced at the end of Section \ref{sec:prelims}, which we call \emph{deterministic} cubic-exchangeable measures. In particular,  we show in Proposition \ref{prop:tau} that the mere assumption that $\ns$ represents \emph{this} class of examples already implies a non-trivial property of $\ns$, which can be viewed as a measure-theoretic version of property $(ii)$ above.

To define deterministic cubic-exchangeable measures, we shall use the following basic result on  disintegrations of measures \cite[(17.35)]{Ke}.

\begin{lemma}\label{lem:disint}
Let $R$, $S$ be standard Borel spaces and let $f: R \to S$ be a Borel map. Let $\mu\in \mc{P}(R)$ and $\nu = f_*\mu:=\mu\co f^{-1}\in \mc{P}(S)$. Then there is a Borel map $S\to \mc{P}(R)$, $s\mapsto\mu_s$ such that for $\nu$-almost every $s\in S$ we have $\mu_s(f^{-1}\{s\}) = 1$ and $\mu = \int_S \mu_s\ud\nu(s)$, i.e., for any Borel set $A\subset R$ we have $\mu(A) = \int_S \mu_s(A)\ud\nu(s)$. Moreover, if $s\mapsto \kappa_s$ is another map with these properties, then for $\nu$-almost every $s$ we have $\mu_s = \kappa_s$. 
\end{lemma}
We call $(\mu_s)_{s\in S}$ the \emph{disintegration} of $\mu$ relative to $f$.
\begin{defn}\label{def:determ}
Let $\mu$ be a cubic-exchangeable measure in $\mc{P}(\Bo^{\db{\mb{N}}})$. We say that $\mu$ is \emph{deterministic} if for every $w\in \db{\mb{N}}$, letting $R=\Bo^{\db{\mb{N}}}$, $S=\Bo^{\db{\mb{N}}\setminus\{w\}}$, letting $f:R\to S$ be the coordinate projection and letting $\nu$ be the pushforward $\mu\co f^{-1}$, we have that $\nu$-almost every measure $\mu_s$ in the disintegration of $\mu$ relative to $f$ is a Dirac measure.
\end{defn}
\noindent Note that in this definition we can replace ``for every $w\in \db{\mb{N}}$" equivalently by ``for \emph{some} $w\in \db{\mb{N}}$", using that the group $\aut(\db{\mb{N}})$ (whose induced coordinate-action leaves $\mu$ invariant, by cubic exchangeability) acts transitively on $\db{\mb{N}}$. Note also that, by Lemma \ref{lem:disint}, each measure $\mu_s$ in Definition \ref{def:determ} is concentrated on $f^{-1}(\{s\})$ (even though it is presented as a measure on $\Bo^{\db{\mb{N}}}$), and $f^{-1}(\{s\})$ is measure-theoretically equivalent to $\Bo$ (indeed $f^{-1}(\{s\})=\Bo\times \{s\}$). Thus, in particular, the assumption that $\mu_s$ is a Dirac measure here implies that $\mu_s$ can be viewed as $\delta_z\in\mc{P}(\Bo)$ for some $z\in \Bo$.

We can now take a first step towards Theorem \ref{thm:repre-with-2-hom}, with the following result.
\begin{lemma}\label{lem:determDirac}
Let $\ns$ be a compact profinite-step nilspace, let $\Bo$ be a standard Borel space, and let $m:\ns\to\mc{P}(\Bo)$ be a Borel map. Then $\zeta_{\ns,m}$ is deterministic if and only if for $\mu_{\ns}$-almost-every $x\in \ns$, the measure $m(x)$ is a Dirac measure on $\Bo$.
\end{lemma}
\begin{proof}
Fix any $w\in\db{\mb{N}}$, let $f:\Bo^{\db{\mb{N}}}\to \Bo^{\db{\mb{N}}\setminus\{w\}}$ be the coordinate projection (i.e.\ the map that just deletes the $w$-th coordinate). 

Forward implication: consider the disintegration $\zeta_{\ns,m}= \int_{\Bo^{\db{\mb{N}}\setminus\{w\}}} \mu_s\ud\nu(s)$ relative to $f$ given by Lemma \ref{lem:disint} (in particular $\nu=\zeta_{\ns,m}\co f^{-1}$). We are assuming that $\zeta_{\ns,m}$ is deterministic, so $\mu_s$ is a Dirac measure $\delta_{g(s)}$ for $\nu$-almost every $s$.  From the construction \eqref{eq:zeta} of $\zeta_{\ns,m}$, we have
\begin{eqnarray}\label{eq:zetadisint}
\zeta_{\ns,m} & = & \int_{\cu^{\omega}(\ns)} \Big[\prod_{v\in \db{\mb{N}}\setminus\{w\}} m(\q(v))\Big]\times m(\q(w))\, \ud \mu_{\cu^{\omega}(\ns)}(\q).
\end{eqnarray}
Note that this last expression can be viewed as a disintegration relative to $f$. Indeed, first note that the underlying assumption that $\ns$ is profinite-step implies uniqueness of $\omega$-corner completion (by Lemma \ref{lem:uniqueomegacomp}). By this uniqueness of completion, this last integral can be written as an integral over the space $\cor^\omega(\ns)$ of $\omega$-corners ``rooted" at $w$ (i.e.\ the map $f$ restricts to a bijection $\cu^\omega(\ns)=\hom(\db{\mb{N}},\ns)\to \cor_w^\omega(\ns)=\hom(\db{\mb{N}}\setminus\{w\},\ns)$ that preserves the Haar measures on these two morphism sets; see Remark \ref{rem:co-omega-equals-cor-omega}). But then, denoting these $\omega$-corners by $s$, we can see that $\nu$ agrees with the measure $\int_{\cor^\omega(\ns)}\prod_{v\in \db{\mb{N}}\setminus\{w\}} m(s(v))\ud\mu_{\cor^\omega(\ns)}(s)$, viewed as a measure on $\Bo^{\db{\mb{N}}\setminus\{w\}}$ as usual. Moreover, comparing with the last integral in \eqref{eq:zetadisint}, we see that $\zeta_{\ns,m} = \int_{\Bo^{\db{\mb{N}}\setminus\{w\}}} \mu_s \ud\nu(s)$, where $\mu_s=m(\q_s(w))$, for the unique $\q_s\in \cu^\omega(\ns)$ completing $s$. By the almost-sure uniqueness of the fiber-measures in the  disintegration, and the assumption that each fiber measure $\mu_s$ is a Dirac measure, it follows that for $\nu$-almost every corner $s$ the measure $m(\q_s(w))$ is a Dirac measure. Finally, note that, as $s$ varies $\nu$-uniformly in $\cor^\omega(\ns)$, the point $\q_s(w)$ varies $\mu_{\ns}$-uniformly in $\ns$ (just recall that $\cor^\omega(\ns)=\cu^\omega(\ns)$ and that the projection $p_w:\cu^\omega(\ns)\to \ns$, $\q\mapsto\q(w)$ is a continuous totally-surjective bundle morphism and by Lemma \ref{lem:mes-pre-1} it preserves the Haar measure). Thus we are covering almost every point $x\in \ns$ as a value $\q_s(w)$, and we can therefore deduce that $m(x)$ is a Dirac measure for $\mu_{\ns}$-almost every $x$ as required.

Backward implication: if almost every $m(x)$ is a Dirac measure, then using the view of \eqref{eq:zetadisint} as a disintegration of $\zeta_{\ns,m}$ relative to $f$, we see that this is a disintegration with almost every fiber-measure being a Dirac measure, so $\zeta_{\ns,m}$ is deterministic.
\end{proof}
\noindent We deduce the following result, which yields a strong measure-theoretic relation between nilspaces $\ns,\nss$ assuming that $\zeta_{\ns,m}$, $\zeta_{\nss,m'}$ are the same measure on $\ns^{\db{\mb{N}}}$. 

\begin{proposition}\label{prop:tau}
Let $\ns,\nss$ be compact profinite-step nilspaces, let $m$ be the Borel map $\nss\to\mc{P}(\nss)$ sending $y$ to $\delta_y$, and suppose that $\zeta_{\ns,m'}=\zeta_{\nss,m}$ for some Borel map $m':\ns\to \mc{P}(\nss)$. Then there is a Borel map  $\tau:\ns\to\nss$ such that $m'(x)=\delta_{\tau(x)}$ for $\mu_{\ns}$-almost-every $x\in\ns$, and for every integer $n\geq 0$ we have $\mu_{\cu^n(\ns)}\co(\tau^{\db{n}})^{-1}=\mu_{\cu^n(\nss)}$.
\end{proposition}

\begin{proof}
Since by assumption $m$ is Dirac-measure valued, by the backward implication in Lemma \ref{lem:determDirac} we have that $\zeta_{\nss,m}$ is deterministic. Hence, the assumption $\zeta_{\ns,m'}=\zeta_{\nss,m}$ implies that $\zeta_{\ns,m'}$ is deterministic, so by the forward implication in Lemma \ref{lem:determDirac}, we deduce that $m'(x)$ is a Dirac measure for $\mu_{\ns}$-almost-every $x\in \ns$. We can therefore define a measurable map $\tau:\ns\to\nss$ such that $m'(x)=\delta_{\tau(x)}$ for $\mu_{\ns}$-almost-every $x\in \ns$.

To show that $\tau$ preserves the Haar measure on every cube set, we use the cubic exchangeability of the measures $\zeta_{\ns,m'}$, $\zeta_{\nss,m}$. Take any injective morphism $\phi:\db{n}\to \db{\mb{N}}$; for simplicity we can take the morphism embedding $\db{n}$ as $\db{n}\times \{0^{\mb{N}\setminus[n]}\}$. Then from the formula \eqref{eq:zeta} defining $\zeta_{\nss,m}$ it is seen directly that the push-forward of $\zeta_{\nss,m}$ induced by $\phi$ on $\nss^{\db{n}}$ is $\mu_{\cu^n(\nss)}$. More precisely, identifying $\phi(\db{n})\subset \db{\mb{N}}$ with $\db{n}$, we have an induced identification of $\nss^{\db{n}}$ with $\nss^{\phi(\db{n})}$, whereby the pushforward of $\zeta_{\nss,m}$ under the coordinate projection $\nss^{\db{\mb{N}}}\to \nss^{\phi(\db{n})}$ can be identified with $\mu_{\cu^n(\nss)}$ (viewed as a measure on $\nss^{\db{n}}$ concentrated on $\cu^n(\nss)$ the usual way). On the other hand, the same formula \eqref{eq:zeta} applied to $\zeta_{\ns,m'}$ shows that this same push-forward is $\mu_{\cu^n(\ns)}\co(\tau^{\db{n}})^{-1}$. Since by assumption $\zeta_{\ns,m'}=\zeta_{\nss,m}$, the result follows.
\end{proof}
\noindent In the proof of Theorem \ref{thm:repre-with-2-hom} we shall use a strengthened version of the above result, in which the map $\tau$ is modified on a null set to obtain a continuous cube-surjective map, which is then shown to be a fibration. This technical upgrade using nilspace-theoretic tools is deferred to Appendix \ref{app:proof-tau-cont}.

We are now ready to prove the theorem of this subsection.
\begin{proof}[Proof of Theorem \ref{thm:repre-with-2-hom}]
$(i)\Rightarrow (ii)$. It suffices to prove $(ii)$ for $\nss=\mc{H}$. Indeed, if we have this, then for any other 2-homogeneous compact profinite-step nilspace $\nss$, by Theorem \ref{thm:universal-covering} there is a fibration $\varphi':\mc{H}\to\nss$, so with the fibration $\varphi:\ns\to\mc{H}$ that we already have we obtain a fibration  $\varphi'\co\varphi:\ns\to\nss$ which proves $(ii)$. Let $m:\mc{H}\to \mc{P}(\mc{H})$ be the map $x\mapsto \delta_x$. Then $\zeta_{\mc{H},m}=\mu_{\cu^\omega(\mc{H})}$ is an affine-exchangeable measure on $\mc{H}^{\db{\mb{N}}}$ with the independence property, so by $(i)$ there is a Borel map $m':\ns\to\mc{P}(\mc{H})$ such that $\zeta_{\mc{H}, m}=\zeta_{\ns,m'}$. By Proposition \ref{prop:tau}, there exists a Borel map $\tau:\ns\to \mc{H}$ that preserves the $n$-cubic Haar measures for all $n\ge 0$. By Theorem \ref{thm:cube-surjective-app} the map $\tau$ agrees $\mu_{\ns}$-almost-everywhere with a continuous cube-surjective morphism $\varphi:\ns\to\mc{H}$, and by Theorem \ref{thm:cube-sur-implies-fib} the map $\varphi$
is a continuous fibration.

$(ii)\Rightarrow (i)$. If $\ns$ satisfies $(ii)$ then in particular there is a fibration $\varphi:\ns\to\mc{H}$. By Theorem \ref{thm:main} the nilspace $\mc{H}$ represents $\textup{Pr}^{\textup{Aff}(\mb{F}_2^\omega)}(\Bo^{\db{\mb{N}}})$ for every standard Borel space $\Bo$.  By Proposition \ref{prop:mes-pre-inf-cube} the map $\varphi^{\db{\mb{N}}}$ preserves the Haar measures on $\cu^{\omega}(\mc{\nss})$, $\cu^{\omega}(\mc{H})$, whence for any Borel map $m:\mc{H}\to \mc{P}(\Bo)$ we have $\zeta_{\mc{H},m}=\zeta_{\ns,m\co \varphi}$, so $\ns$ also represents affine-exchangeability.

$(iii)\Rightarrow (ii)$. The particular case $\Bo=\mc{H}$ of the assumption tells us that $\ns$ is a limit domain for convergent sequences of $\mc{H}$-valued functions. By Theorem \ref{thm:corresp}, the nilspace  $\ns$ represents $\textup{Pr}^{\textup{Aff}(\mb{F}_2^\omega)}(\mc{H}^{\db{\mb{N}}})$. Then, arguing similarly as in the proof of $(i)\Rightarrow (ii)$ above, using the particular measure $\mu_{\cu^\omega(\mc{H})}$, we deduce that there is a fibration $\ns\to\mc{H}$, whence $(ii)$ follows by Theorem \ref{thm:universal-covering}.

$(i)\Rightarrow (iii)$. This implication follows immediately from Theorem \ref{thm:corresp}.
\end{proof}

\subsection{Some restrictions on the structure of limit domains}\label{subsec:restrictions-on-lim-dom}\hfill\\
\noindent Let us record the following consequence of Theorem \ref{thm:repre-with-2-hom} $(ii)$.
\begin{corollary}\label{prop:limdomconst}
Let $(\ab,\ab_\bullet)$ be a filtered compact abelian group such that the associated group nilspace $\ns$ is a limit domain for convergent sequences $(f_n:\mb{F}_2^n\to\Bo)_{n\in \mb{N}}$ for every compact metric space $\Bo$. Then there is a fibration $\varphi:\ns\to\mc{H}$.
\end{corollary} 
\noindent This result yields non-trivial restrictions on the objects that could be used as such general limit domains. In this subsection we illustrate these restrictions with a specific example that has appeared in previous work. 

Throughout this section let $G$ be the compact abelian group $\prod_{k=1}^{\infty}\prod_{\ell=1}^k (\mb{Z}/2^{k-\ell+1}\mb{Z})^{\mb{N}}$.
This group appeared in \cite{HHH} (denoted therein by $\mb{G}_{\infty}$) in connection with convergent sequences of Boolean functions $(f_n:\mb{F}_2^n\to\{0,1\})_{n\in \mb{N}}$. The results in  \cite{HHH}  (specifically \cite[Theorem 3.5]{HHH} with $d=\infty$) would suggest that, if we equip $G$ with the filtration  $G_\bullet=\prod_{k=1}^{\infty}\prod_{\ell=1}^k \abph_{k,\ell}^{\mb{N}}$, then the resulting group nilspace can be used as a representing nilspace for affine exchangeability, or equivalently as a limit domain for convergent sequences of $\Bo$-valued functions, for every compact metric space $\Bo$ (see Appendix \ref{app:correspondence} for more details). However, we will show below that this is not true. More precisely, if $(G,G_\bullet)$ were usable as such a general limit domain then by Corollary \ref{prop:limdomconst} there would exist a continuous fibration from the associated nilspace onto $\mc{H}$, and we shall prove below that such a fibration cannot exist. To prove this, we first note that if there existed such a fibration, then composing it with a projection to a single component $\abph_{\infty,1}$, we would obtain a fibration from $(G,G_\bullet)$ to $\abph_{\infty,1}$. Before proving that such a fibration does not exist, let us prove an analogous fact in the category of abelian groups, namely, that there is no homomorphism from $G$ to the group $\mf{Z}$ of 2-adic integers.

To begin with, we recall the following basic fact.

\begin{lemma}
The group of $2$-adic integers is torsion-free.
\end{lemma}
\noindent From this we can immediately deduce the following.

\begin{lemma}
There is no \emph{non-constant} continuous affine homomorphism $G\to \mf{Z}$.
\end{lemma}

\begin{proof}
Using translations it clearly suffices to prove that there is no \emph{non-zero continuous} homomorphism $G\to \mf{Z}$. Suppose that we have a homomorphism $\varphi:G\to \mf{Z}$. For any $x\in G$ with only finitely many non-zero coordinates, we must have $\varphi(x)=0$ (otherwise, since $x$ has finite order, then so does $\varphi(x)$, contradicting the previous lemma). As every element of $G$ is the limit of elements of finite order in $G$, we then deduce by continuity that $\varphi$ is the zero homomorphism, a contradiction.
\end{proof}

\noindent We can generalize this argument to the case of fibrations between nilspaces. Let us start by proving the following lemma.

\begin{lemma}\label{lem:finite-degree-1-factor} 
Let $H_{\bullet}$ be a 2-homogeneous filtration of degree $k$ on an abelian group $H$, and let  $\ns$ be the group nilspace associated with $(H,H_\bullet)$. Then for every morphism $\varphi:\ns\to \abph_{\infty,1}$, there exists a morphism $\varphi_1: \pi_1(\ns)\to \abph_{\infty,1}$ such that $\varphi = \varphi_1\co \pi_1$. 
\end{lemma}

\begin{proof} 
We argue by induction on $k$. The case $k=1$ is trivial. Now assume that the result holds for step less than $k$ and let $H_{\bullet}$ be a degree-$k$ filtration. We want to prove that for any $g\in H$ and any $g_k\in H_{(k)}$ we have $\varphi(g+g_k)=\varphi(g)$. 

Fix any $g\in H$ and $g_k\in H_{(k)}$. Let $a_0:=\varphi(g)$ and $a_1:=\varphi(g+g_k)$. Recall that for any pair of abelian groups $Z,Z'$ and any function $f:Z\to Z'$, the forward difference of $f$ with step $h\in Z$ is the function $\partial_{h}f(x):=f(x+h)-f(x)$, $x\in H$. We now claim that $\partial_{g_k}^n\varphi(g) = 2^{n-1}(-1)^n(a_0-a_1)$ for every $n\in \mb{N}$. Indeed this is clear for $n=1$, and for $n>1$ we have by induction
\begin{eqnarray*}
\partial_{g_k}^n\varphi(g)  & = & \partial_{g_k}(2^{n-2}(-1)^{n-1}(\varphi(g)-\varphi(g+g_k))) 
 =  2^{n-2}(-1)^{n-1}(\partial_{g_k}\varphi(g)-\partial_{g_k}\varphi(g+g_k)) \\
& = & 2^{n-2}(-1)^{n-1}(a_1-a_0-(a_0-a_1)) = 2^{n-1}(-1)^n(a_0-a_1)
\end{eqnarray*}
where we have used that $g_k$ has order 2 (since $H_\bullet$ is 2-homogeneous).

Now, by standard properties of polynomial maps \cite[Theorem 2.2.14]{Cand:Notes1}, since $g_k\in H_{(k)}$ we have $\partial_{g_k}^n\varphi(g)\in (\abph_{\infty,1})_{kn}$ where $(\abph_{\infty,1})_{j}$ is the $j$-th term defining the filtration on $\mf{Z}$ associated with $\abph_{\infty,1}$ by Definition \ref{def:del-p-adic-objects}. From this definition it follows that this group is precisely $2^{nk-1}\mf{Z}$. Now since $a_0\not=a_1$, we have for some $r\in\mb{N}$ that $a_0-a_1=\sum_{i=r}^{\infty }\alpha_i 2^i$ where $\alpha_i\in\{0,1\}$ for all $i\ge r$ and $\alpha_r=1$. In particular the 2-adic order of $a_0-a_1$ is at most $r \in \mb{N}$, and then by the previous paragraph the 2-adic order of $\partial_{g_k}^n\varphi(g)$ is at most $r+n-1$. But in the paragraph we have also proved that this order is at least $nk-1$. This yields a contradiction for $n$ large enough.

Thus we have proved that $\varphi:\ns\to \abph_{\infty,1}$ factors through $\pi_{k-1}$. By induction on $k$, we deduce that $\varphi$ factors though $\pi_1$.
\end{proof}
\noindent Let us now adapt the above arguments to the category of nilspaces, to apply it to the filtered group $(G,G_\bullet)$ defined above.
\begin{lemma}\label{lem:not-many-morphisms} 
Let $\ns$ be the group nilspace associated with $(G,G_\bullet)$, and let $\varphi:\ns\to \abph_{\infty,1}$ be a morphism. Then there is a morphism $\varphi_1:\pi_1(\ns)\to \abph_{\infty,1}$ such that $\varphi = \varphi_1 \co \pi_1$.
\end{lemma}

\begin{proof} 
Take $g,g'\in \ns$ such that $\pi_1(g)=\pi_1(g')$. For any $n\in\mb{N}$ let $p_n:G \to G$ be the homomorphism that projects onto the subgroup $\big(\prod_{k=1}^n\prod_{\ell=1}^k \mb{Z}_{2^{k-\ell+1}}^n\times \{0\}^{\mb{N}\setminus [n]}\big)\times (\prod_{k=n+1}^{\infty}\prod_{\ell=1}^k \{0\}^{\mb{N}})$ (by switching the relevant coordinates to 0). Clearly we have $p_n(g)\to g$ and $p_n(g')\to g'$ as $n\to \infty$. Furthermore $\pi_1(p_n(g))=\pi_1(p_n(g'))$ for all $n$ (since $p_n$ commutes with the projections to the 1-factors). 

By Lemma \ref{lem:finite-degree-1-factor} we have $\varphi(g_n)=\varphi(g'_n)$ for all $n$. Hence, taking the limit as $n\to \infty$ and using the continuity of $\varphi$ we conclude that $\varphi(g)=\varphi(g')$.
\end{proof}

\begin{theorem}\label{thm:morphinex}
Let $\ns$ be the group nilspace associated with $(G,G_\bullet)$. There exists no cube-surjective morphism $\ns\to \mc{H}$.
\end{theorem}

\begin{proof} 
If there existed such a morphism $\varphi:\ns\to \mc{H}$, then there would exist a cube-surjective morphism $\varphi':\ns\to \abph_{\infty,1}$. By Theorem \ref{thm:cube-sur-implies-fib} we know that $\varphi'$ would then be a fibration. But by Lemma \ref{lem:not-many-morphisms} we also know that $\varphi'$ would factor through $\pi_1$, i.e., there would exist a morphism $\varphi_1:\pi_1(\ns)\to \mc{H}$ such that $\varphi'=\varphi_1\co \pi_1$. Then, since $\varphi'$ and $\pi_1$ would be fibrations, so would $\varphi_1$. Thus we would have a fibration $\varphi_1:\pi_1(\ns)\to \abph_{\infty,1}$, and this yields a contradiction because $\pi_1(\ns)$ is a 1-step nilspace whereas $\abph_{\infty,1}$ has infinite step (but the step of a fibration's image cannot be larger than the step of its domain).
\end{proof}

\appendix

\section{\texorpdfstring{$\infty$}{infinite}-fold compact abelian bundles}\label{app:cubic-coupling-infinite-step}

\noindent The concept of Haar measure on a compact nilspace of finite step is crucial in nilspace theory. A key fact enabling the definition of this measure is that any such nilspace can be expressed as an abelian bundle construction iterated finitely many times (see \cite[\S 2.2.2]{Cand:Notes2}). In this paper we need an analogous Haar measure for profinite-step compact nilspaces, and to this end, similarly, we shall use the following notion of abelian bundle.

\begin{defn}[$\infty$-fold abelian bundle]\label{def:infty-fold-ab-bundle}
A set $\Bo$ is an $\infty$-fold abelian bundle if there exists a sequence of sets $(\Bo_i)_{i\ge 0}$ and a sequence of maps $\pi_{j,i}:\Bo_i\to\Bo_j$ for $0\le j\le i$ such that the following properties hold. For every $i\ge 1$ there exists an abelian group $\ab_i$ such that $\Bo_i$ is a $\ab_i$-bundle over $\Bo_{i-1}$ with projection $\pi_{i-1,i}$ (in the sense of \cite[Definition 3.2.17]{Cand:Notes1}). Then for all $0\le j\le i$ we have $\pi_{j,i}= \pi_{j,j+1}\co\cdots\co \pi_{i-1,i}$. Finally, we have that $\Bo_0$ is a  singleton and $\Bo = \varprojlim \Bo_i =\{(b_k)_{k=0}^\infty\in \prod_{k=0}^\infty \Bo_k:\forall\; 0\le j\le i, \pi_{j,i}(b_i)=b_j\}$. For each $i\geq 0$ we denote by $\pi_i$ the limit map $\Bo\to\Bo_i$, $(b_k)_{k=0}^\infty\mapsto b_i$. 
\end{defn}
\noindent Thus $\infty$-fold abelian bundles are inverse limits of $k$-fold abelian bundles. When we carry out this construction in the category of \emph{compact} abelian bundles, we obtain the following.
\begin{defn}[$\infty$-fold compact abelian bundles]\label{def:infty-fold-compact-ab-bundle}
An \emph{$\infty$-fold compact abelian bundle} is an $\infty$-fold abelian bundle $\Bo=\varprojlim \Bo_i$ such that $\Bo_i$ is a compact abelian $\ab_i$-bundle over $\Bo_{i-1}$ (in the sense of \cite[Definition 2.2.6]{Cand:Notes2}) and $\Bo$ is equipped with the compact topology generated by the limit maps $\pi_i$.
\end{defn}

\begin{remark} 
Note that in particular this implies that $\Bo$ is a compact space and the maps  $\pi_i:\Bo\to\Bo_i$ are continuous and open for all $i\ge 0$.
\end{remark}

We can now define the desired notion of Haar measure.

\begin{lemma}[Haar measure on $\infty$-fold compact abelian bundles]\label{lem:haar-mes-infty-fold} 
Let $\Bo=\varprojlim \Bo_i$ be an $\infty$-fold compact abelian bundle. Then there exists a unique measure $\mu_{\Bo}$ on $\Bo$ such that for every $i\ge 0$ we have $\mu_{\Bo}\co \pi_i^{-1} = \mu_{\Bo_i}$ where $\mu_{\Bo_i}$ is the Haar measure on the $i$-fold compact abelian bundle $\Bo_i$ \textup{(}constructed in \cite[Proposition 2.2.5]{Cand:Notes2}\textup{)}.
\end{lemma}

\begin{proof} 
Let $\mc{S}:=\{\pi_i^{-1}(A):i\ge 0, A\subset \mc{A}(\Bo_i)\}$ where $\mc{A}(\Bo_i)$ is the Borel $\sigma$-algebra on $\Bo_i$. Note that $\mc{S}$ is a semi-ring of sets. We define $\mu_{\Bo}((\pi_i)^{-1}(A)):=\mu_{\Bo_i}(A)$. As the functions $\pi_{j,i}:\Bo_i\to\Bo_j$ for $i\ge j$ form an inverse system, note that if $\pi_i^{-1}(A)=\pi_j^{-1}(C)$ for some $i\ge j$, then $\pi_i^{-1}(A)=\pi_i^{-1}(\pi_{j,i}^{-1}(C))$. Hence $A=\pi_{j,i}^{-1}(C)$ and by \cite[Lemma 2.2.6]{Cand:Notes2} we know that $\pi_{j,i}$ preserves the Haar measure. Thus $\mu_{\Bo_i}(A)=\mu_{\Bo_j}(C)$.

We now check that with this definition the measure $\mu_{\Bo}$ is $\sigma$-additive on $\mc{S}$. Suppose that $A\in \mc{S}$ is of the form $A = \sqcup_{j=1}^{\infty} A_j$ for some $A_j\in \mc{S}$. In particular $\cup_{j=1}^k A_j \subset A$ for any $k\ge 1$. We also have $A=\pi_r^{-1}(C)$ and $A_j = \pi_{r_j}^{-1}(C_j)$ for all $j\ge 0$ (since  $A,A_j\in \mc{S}$). Letting $r'\ge \max(r,r_1,\ldots,r_k)$, we have that these $A$ and these $A_j$ can all be seen as subsets of $\Bo_{r'}$. Hence, by additivity of the measure $\mu_{\Bo_{r'}}$ we have $\mu_{\Bo}(A) \ge \sum_{j=1}^k \mu_{\Bo}(A_j)$. Letting $k\to \infty$ we obtain the inequality $\mu_{\Bo}(A) \ge \sum_{j=1}^{\infty} \mu_{\Bo}(A_j)$.

To prove the reversed inequality, let us fix any $\epsilon>0$. By \cite[Proposition 8.1.12]{Cohn} we know that for all $n\ge 0$ and $k\ge 1$, the measures $\mu_{\Bo_k}$ are regular. Thus, there exists a compact set $D\subset C$ such that $\mu_{\Bo_r}(C\setminus D)<\epsilon$ (recall that $A=\pi_r^{-1}(C)$). Similarly, for every $j\ge 1$ there exists an open set $E_j\supset C_j$ such that $\mu_{\Bo_{r_j}}(E_j\setminus C_j) < \epsilon/2^j$. As $\pi_r^{-1}(D)\subset \cup_{j=1}^{\infty}\pi_{r_j}^{-1}(E_j)$, by compactness there is a finite set $I\subset \mb{N}$ such that $\pi_r^{-1}(D)\subset \cup_{j\in I}\pi_{r_j}^{-1}(E_j)$. Now, taking $r''\ge \max(r,(r_j)_{j\in I})$, arguing as in the previous paragraph we obtain $\mu_{\Bo}(A)-\epsilon \le \sum_{j\in I} (\mu_{\Bo}(A_j)+\epsilon/2^j)\le \sum_{j=1}^{\infty} (\mu_{\Bo}(A_j)+\epsilon/2^j)$. Hence $\mu_{\Bo}(A)\le 2\epsilon+ \sum_{j=1}^{\infty} \mu_{\Bo}(A_j)$. Letting $\epsilon\to 0$, the desired inequality follows.

We complete the proof by applying Carath\'eodory's extension theorem.
\end{proof}

Next, we need to generalize the notion of bundle morphism from \cite[Definition 3.3.1]{Cand:Notes1}.

\begin{defn}\label{def:infty-fold-bundle-morphism}
Let $\Bo,\Bo'$ be $\infty$-fold abelian bundles. Let $\Bo_i,\Bo_i'$ be the corresponding terms of the inverse limit and $\ab_i,\ab_i'$ the structure groups. A bundle morphism from $\Bo$ to $\Bo'$ is a map $\varphi:\Bo\to\Bo'$ satisfying the following properties.
\begin{enumerate}[leftmargin=0.8cm]
    \item For every $i\ge 0$, if $\pi_i(x)=\pi_i(y)$ then $\pi_i(\varphi(x))=\pi_i(\varphi(y))$. Thus there is an induced well-defined map $\varphi_i:\Bo_i\to\Bo_i'$.
    \item For every $i\ge 1$ there is a homomorphism $\alpha_i:\ab_i\to\ab_i'$ such that for every $x\in \Bo_i$ and $z\in \ab_i$ we have $\varphi_i(x+z)=\varphi_i(x)+\alpha_i(z)$.
\end{enumerate}
The bundle morphism is \emph{totally surjective} if the maps $\alpha_i$ are surjective for all $i\ge 1$.
\end{defn}
\noindent Totally surjective bundle morphisms for $\infty$-fold compact abelian bundles generalize the concept of surjective homomorphisms for compact abelian groups. In particular, these maps preserve the Haar measure.

\begin{lemma}\label{lem:mes-pre-1} 
Let $\Bo,\Bo'$ be two $\infty$-fold compact abelian bundles and let $\varphi:\Bo\to\Bo'$ be a continuous totally surjective bundle morphism. Then $\mu_{\Bo'} = \mu_{\Bo}\co \varphi^{-1}$.
\end{lemma}

\begin{proof} 
By the Riesz representation theorem it suffices to check that for any continuous function $f:\Bo'\to \mb{C}$ we have $\int f\;\ud\mu_{\Bo'} = \int f \co \varphi\;\ud\mu_{\Bo}$. By the Stone-Weierstrass theorem, the continuous functions of the form $f'\co\pi_i$, for $i\ge 0$ and $f':\Bo_i'\to\mb{C}$ continuous, are dense in the space of continuous functions $\Bo\to\mb{C}$. Therefore it suffices to check that $\int f'\co\pi_i\;\ud\mu_{\Bo'} = \int f'\co\pi_i \co \varphi\;\ud\mu_{\Bo}$ for any such function. This is indeed the case since $\int f'\co\pi_i \co \varphi\;\ud\mu_{\Bo}=\int f' \co\varphi_i\co\pi_i\;\ud\mu_{\Bo'} = \int f'\co\varphi_i\;\ud\mu_{\Bo_i'}$, and then by \cite[Lemma 2.2.6]{Cand:Notes2} this equals $\int f'\;\ud\mu_{\Bo_i}=\int f'\co\pi_i \;\ud\mu_{\Bo}$.
\end{proof}
\noindent Given two $\infty$-fold compact abelian bundles $\Bo,\Bo'$ and $\varphi:\Bo\to\Bo'$ a continuous totally surjective bundle morphism it can be checked that for every $t\in \Bo'$ the set $\varphi^{-1}(t)$ is an $\infty$-fold compact abelian sub-bundle of $\Bo$ (the proof is a generalization of \cite[Lemma 3.3.6]{Cand:Notes1}). This means in particular that the $i$-th limit map on $\varphi^{-1}(t)$  is the restriction of $\Bo$'s $i$-th limit map $\pi_i$ to $\varphi^{-1}(t)$. It is also straightforward to check that the $i$-th structure group of $\varphi^{-1}(t)$ is $\ker(\alpha_i)$, for every $i\ge 1$. We can now generalize the quotient integral formula \cite[Lemma 2.2.10]{Cand:Notes2}.

\begin{lemma}\label{lem:mes-pre-2} 
Let $\Bo,\Bo'$ be $\infty$-fold compact abelian bundles and let $\varphi:\Bo\to\Bo'$ be a continuous totally surjective bundle morphism. Then for any Borel set $E\subset \Bo$ we have
\[
\int_{\Bo} 1_E\; \ud\mu_{\Bo} = \int_{\Bo'} \int_{\varphi^{-1}(t)} 1_E(x)\;\ud\mu_{\varphi^{-1}(t)}(x) \ud\mu_{\Bo'}(t)
\]
where $1_E$ is the indicator function of $E$.
\end{lemma}

\begin{proof} 
Similarly as before, by the Riesz representation and Stone-Weierstrass theorems it suffices to check that $\int_{\Bo} f\co\pi_i\; \ud\mu_{\Bo} = \int_{\Bo'} \int_{\varphi^{-1}(t)} f\co\pi_i(x)\;\ud\mu_{\varphi^{-1}(t)}(x)\ud\mu_{\Bo'}(t)$ for any continuous function $f:\Bo_i\to\mb{C}$ and $i\ge 1$. Note that for any $t\in \Bo'$, $\int_{\varphi^{-1}(t)} f\co\pi_i(x)\;\ud\mu_{\varphi^{-1}(t)}(x) = \int_{\varphi_i^{-1}(\pi_i(t))} f(x')\;\ud\mu_{\varphi_i^{-1}(\pi_i(t))}(x') =\Big[ b\in \Bo_i  \mapsto  \int_{\varphi_i^{-1}(b)} f(x')\;\ud\mu_{\varphi_i^{-1}(b)}(x')\Big] \co \pi_i$. Hence
\[
\int_{\Bo'} \int_{\varphi^{-1}(t)} f\co\pi_i(x)\;\ud\mu_{\varphi^{-1}(t)}(x)\ud\mu_{\Bo'}(t) = \int_{\Bo_i'} \int_{\varphi_i^{-1}(t)} f\co\pi_i(x')\;\ud\mu_{\varphi_i^{-1}(t')}(x')\ud\mu_{\Bo_i'}(t'),
\]
and the latter equals $\int_{\Bo_i} f\;\ud\mu_{\Bo_i} = \int_{\Bo} f\co\pi_i\;\ud\mu_{\Bo}$ by \cite[Lemma 2.2.10]{Cand:Notes2}.
\end{proof}
\begin{remark} 
Note that by standard arguments using approximations by simple functions, we can deduce from  Lemma \ref{lem:mes-pre-2} that for every $f\in L^1(\Bo)$ we have $\int_{\Bo} f\; \ud\mu_{\Bo} = \int_{\Bo'} \int_{\varphi^{-1}(t)} f(x)\;\ud\mu_{\varphi^{-1}(t)}(x)\ud\mu_{\Bo'}(t)$.
\end{remark}

\subsection{The cubic coupling property for profinite-step compact nilspaces}\hfill\\ 
The goal of this section is to prove that the $n$-cubic Haar measures on a compact profinite-step nilspace form a cubic coupling (see Definition \ref{def:cub-cou}). Let us start proving the ergodicity and consistency axioms.

\begin{proposition}
Let $\ns$ be a compact profinite-step nilspace. Then its $n$-cubic Haar measures satisfy the ergodicity and consistency axioms.
\end{proposition}

\begin{proof} 
To prove that the ergodicity axiom holds we have to prove that $\mu_{\cu^1(\ns)}=\mu_{\ns}\times\mu_{\ns}$. Suppose that $\ns=\varprojlim \ns_i$ where $\ns_i$ is the $i$-th characteristic factor of $\ns$. Let $\mc{S}(\ns^{\db{1}}):=\{(\pi_i^{\db{1}})^{-1}(A):i\ge 0, A\in \mc{A}(\ns_i)^{\db{1}}\}$ be the semi-ring that generates the $\sigma$-algebra on $\ns^{\db{1}}$, where $\mc{A}(\ns_i)$ is the Borel $\sigma$-algebra on $\ns_i$. It suffices to check that the equality $\mu_{\cu^1(\ns)}=\mu_{\ns}\times\mu_{\ns}$  holds on $\mc{S}(\ns^{\db{1}})$. But note that $\mu_{\cu^1(\ns)}((\pi_i^{\db{1}})^{-1}(A)) = \mu_{\cu^1(\ns_i)}(A)=(\mu_{\ns_i}\times \mu_{\ns_i})(A) = (\mu_{\ns}\times \mu_{\ns})((\pi_i^{\db{1}})^{-1}(A))$ where in the second equality we have used the ergodicity axiom for $\ns_i$ given by \cite[Proposition 3.6]{CScouplings}.

To establish the consistency axiom, let $\phi:\db{n}\to \db{m}$ be an injective discrete-cube morphism ($n\le m$) and $p_{\phi}:\ns^{\db{m}}\to \ns^{\db{n}}$ the projection $(\q(v))_{v\in\db{m}}\mapsto (\q(\phi(v)))_{v\in\db{n}}$. We have to prove that $\mu_{\cu^m(\ns)} \co p_{\phi}^{-1} = \mu_{\cu^n(\ns)}$. Similarly as before it is enough to check this for the semi-ring $\mc{S}(\ns^{\db{n}})$. We clearly have $\pi_i^{\db{n}} \co p_{\phi} = p_{\phi} \co \pi_i^{\db{m}}$. Thus $\mu_{\cu^m(\ns)} \co p_{\phi}^{-1} \co (\pi_i^{\db{n}})^{-1} = \mu_{\cu^n(\ns)} \co (\pi_i^{\db{m}})^{-1}\co p_{\phi}^{-1} = \mu_{\cu^m(\ns_i)} \co p_{\phi}^{-1}$. 
By the consistency axiom for $\ns_i$ (\cite[Proposition 3.6]{CScouplings}), we have $\mu_{\cu^m(\ns_i)} \co p_{\phi}^{-1}=\mu_{\cu^{n}(\ns_i)}$, and by construction this is $ \mu_{\cu^n(\ns)} \co (\pi_i^{\db{n}})^{-1}$.
\end{proof}
\noindent To verify the conditional independence axiom, we shall need some preparation. The following result generalizes \cite[Lemma 2.1.10]{Cand:Notes2}. Recall from \cite[Definition 3.1.3]{Cand:Notes1} the notion of a subset of $\db{n}$ having the \emph{extension property}.
\begin{lemma}\label{lem:hom-spaces-infty-fold-ab-bundles}
Let $\ns$ be a compact profinite-step nilspace, and let $n\ge 0$. Let $P\subset\db{n}$ be a set with the extension property in $\db{n}$ and $S\subset P$ be a set with the extension property in $P$. Let $f:S\to \ns$ be a morphism. Then the set $\hom_f(P,\ns):=\{\q\in \hom(P,\ns):\q|_S=f\}$ is an $\infty$-fold compact abelian sub-bundle of $\ns^P$ with factors $\hom_{\pi_i\co f}(P,\ns_i)$ and structure groups $\hom_{S\to0}(P,\mc{D}_i(\ab_i))$ for all $i\ge 1$, where $\ab_i$ is the $i$-th structure group of $\ns$.
\end{lemma}

\begin{proof} 
This follows from \cite[Lemma 2.1.10]{Cand:Notes2}.
\end{proof}

\noindent Recall from \cite[Definition 2.2.13]{Cand:Notes2} that two subsets $P_1,P_2$ of a discrete cube $\db{n}$ are said to form a \emph{good pair} if the following conditions are satisfied: both $P_1$ and $P_1\cap P_2$ have the extension property in $\db{n}$ (see \cite[Definition 3.1.3]{Cand:Notes1}),  and for any $k\ge 1$, any abelian group $\ab$, and every morphism $f':P_2\to \mc{D}_k(\ab)$ (see \cite[Definition 2.2.30]{Cand:Notes1}) satisfying $f'|_{P_1\cap P_2}=0$, there exists a morphism $f:\db{n}\to \mc{D}_k(\ab)$ extending $f'$ such that $f|_{P_1}=0$. The next result extends \cite[Lemma 2.2.14]{Cand:Notes2}.

\begin{proposition}\label{prop:rest-tot-sur} 
Let $\ns$ be a compact profinite-step nilspace, let $P_1,P_2\subset \db{n}$ be a good pair in $\db{n}$, and let $f:P_1\to\ns$ be a morphism. Then the restriction map
\[
\Psi:\hom_f(\db{n},\ns)\to \hom_{f|_{P_1\cap P_2}}(P_2,\ns)
\]
is a totally surjective continuous bundle morphism.
\end{proposition}

\begin{proof}
The proof is very similar to that of \cite[Lemma 2.2.14]{Cand:Notes2}.
\end{proof}

\noindent Now let us state the main result that will enable us to prove the conditional independence axiom for compact profinite-step nilspaces.

\begin{lemma}\label{lem:goo-pai-ind}
Let $P_0,P_1\subset \db{n}$ be a good pair. Then for every compact profinite-step nilspace $\ns$, the sets $P_0,P_1$ are conditionally independent with respect to $\mu_{\cu^n(\ns)}$.
\end{lemma}

\begin{proof}
The proof is very similar to the proof of \cite[Lemma 5.108]{GS}. We replace Lemma A.28, Proposition 5.57 and Proposition 5.60 from \cite{GS} respectively by Proposition \ref{prop:rest-tot-sur}, Lemma \ref{lem:mes-pre-1} and Lemma \ref{lem:mes-pre-2} from this paper.
\end{proof}

\begin{corollary} 
Let $\ns$ be a compact profinite-step nilspace. Then its cubic Haar measures satisfy the conditional independence axiom.
\end{corollary}

\begin{proof}
It is not hard to check that if $P_0,P_1\subset \db{n}$ are $(n-1)$-dimensional faces with $P_0\cap P_1\neq\emptyset$, then they form a good pair. Then the result follows from Lemma \ref{lem:goo-pai-ind}.
\end{proof}

\section{Affine-exchangeable measures on nilspaces and 2-homogeneity}\label{app:aff-exch-and-2-hom}

In this appendix we prove Lemma \ref{lem:aff-ech-equals-2-hom}, which we restate here for convenience.
\begin{lemma}\label{lem:aff-ech-equals-2-hom-app}
Let $\ns$ be a compact profinite-step nilspace. Then $\mu_{\cu^\omega(\ns)}$ \textup{(}viewed as a measure on $\ns^{\db{\mb{N}}}$\textup{)} is affine-exchangeable if and only if $\ns$ is 2-homogeneous.
\end{lemma}

\begin{proof}
Suppose that $\ns$ is 2-homogeneous. By Lemma \ref{lem:2-hom-equals-c-omega-poly} this is equivalent to having $\cu^\omega(\ns)=\hom(\mc{D}_1(\mb{F}_2^{\omega}),\ns)$. Given any $T\in \textrm{GL}(\mb{F}_2^\omega)\ltimes \mb{Z}_2^\omega$, let us define the map $T':\cu^{\omega}(\ns)\to\cu^{\omega}(\ns)$ by $T'(\q) = \q\co T$. To see that $T'(\q)$ is indeed  in $\cu^{\omega}(\ns)$ for every $\q\in\cu^{\omega}(\ns)=\hom(\mb{F}_2^{\omega},\ns)$, we have to check that for every $q\in\cu^n(\mb{F}_2^\omega)$ we have $\q\co T\co q\in \cu^n(\ns)$. But $T\co q$ is clearly in $\cu^n(\mb{F}_2^\omega)$, so $\q\co T\in \cu^{\omega}(\ns)$ as required.

Moreover, it is seen by a straightforward computation that $T'$ defines a totally surjective continuous bundle morphism from the $\infty$-fold compact abelian bundle $\cu^\omega(\ns)$ to itself. Thus, by Lemma \ref{lem:mes-pre-1}, $T'$ preserves the Haar measure, i.e.\ $\mu_{\cu^\omega(\ns)} \co {T'}^{-1} = \mu_{\cu^\omega(\ns)}$. Hence $\mu_{\cu^\omega(\ns)}$ is affine-exchangeable.

For the converse, suppose that $\mu_{\cu^{\omega}(\ns)}$ (seen as a measure on $\ns^{\db{\mb{N}}}$) is affine-exchangeable, and for any $n\in\mb{N}$ let $T\in \textrm{GL}(\mb{F}_2^n)\ltimes \mb{Z}_2^n$.  Considering $\textrm{GL}(\mb{F}_2^n)\ltimes \mb{Z}_2^n\subset \textrm{GL}(\mb{F}_2^\omega)\ltimes \mb{Z}_2^\omega$ in the usual way, we can view $T$ as a transformation in $\textrm{GL}(\mb{F}_2^\omega)\ltimes \mb{Z}_2^\omega$ (abusing the notation). The map $T':\cu^{\omega}(\ns)\to\cu^{\omega}(\ns)$, $\q\mapsto \q\co T$ preserves the measure $\mu_{\cu^\omega(\ns)}$ by assumption. Let $\iota:\db{n}\to {\db{\mb{N}}}$ be the inclusion map defined by $\iota(v):=(v, 0^{\mb{N}\setminus[n]})$, and let $p_n:\ns^{\db{\mb{N}}}\to \ns^{\db{n}}$ be the projection map $(b_v)_{v\in \db{\mb{N}}}\mapsto (b_{\iota(v)})_{v\in\db{n}}$. We then have that $\mu_{\cu^n(\ns)}=\mu_{\cu^{\omega}(\ns)}\co p_n^{-1}$ is invariant under the action of $T'$, viewing the latter as a map $\ns^{\db{n}}\to \ns^{\db{n}}$.

For each $k\ge 0$ we know that $\pi_k^{\db{n}}:\ns^{\db{n}}\to \ns_k^{\db{n}}$ preserves the $n$-cubic Haar measures. We abuse the notation to see $T'$ as a function on $\ns^{\db{n}}$ and $\ns_k^{\db{n}}$, and thus $T'\co \pi_k^{\db{n}} = \pi_k^{\db{n}}\co T'$. We obtain that for every $k\ge 0$ the measure $\mu_{\cu^n(\ns_k)}$ is invariant under the action of $T\in \textrm{GL}(\mb{F}_2^n)\ltimes \mb{Z}_2^n$. 

Note that $\mu_{\cu^n(\ns_k)}(T'(\cu^n(\ns_k))=1$. Since $T'$ is continuous, $T'(\cu^n(\ns_k)$ is a compact subset of $\ns_k^{\db{n}}$ of measure 1. In particular $T'(\cu^n(\ns_k))\cap \cu^n(\ns_k)$ is a closed set of measure 1 (because $\mu_{\cu^n(\ns_k)}$ is supported on $\cu^n(\ns_k)$). Since $\mu_{\cu^n(\ns_k)}$ is a strictly positive Borel measure \cite[Proposition 2.2.11]{Cand:Notes2}, a simple argument deduces that $\cu^n(\ns_k)\subset T'(\cu^n(\ns_k))$ (indeed, otherwise there would be a non-empty open subset of $\cu^n(\ns_k)$ of zero $\mu_{\cu^n(\ns_k)}$-measure). Repeating this argument with ${T}^{-1}$ we obtain the opposite inclusion, and thus conclude that  $T'(\cu^n(\ns_k))=\cu^n(\ns_k)$. We have thus obtained that for any $n,k\ge 0$, $\q\in\cu^n(\ns_k)$, and any $T\in \textrm{GL}(\mb{F}_2^n)\ltimes \mb{Z}_2^n$ we have $\q\co T\in \cu^n(\ns_k)$.

Now to complete the proof we need to establish the claim that given any $f\in\cu^{\ell}(\mb{F}_2^m)$ and any $\q\in \cu^m(\ns_k)$ we have $\q\co f\in \cu^{\ell}(\ns_k)$ for any ${\ell},m,k\ge 0$. This will show that every element of $\cu^m(\ns_k)$ can be regarded as an element of $\hom(\mc{D}_1(\mb{F}_2^m),\ns_k)$ (identifying $\db{m}$ with $\mb{F}_2^m$). Note that earlier we proved the above claim only for $f=T$ an invertible affine map. We now need the result for any $f\in\cu^{\ell}(\mb{F}_2^m)$, which can be viewed as an affine map (not necessarily injective or surjective) $\mb{F}_2^{\ell}\to \mb{F}_2^m$ (recall that $\mc{D}_1(\mb{F}_2^m)$ is 2-homogeneous). To prove the claim, first we reduce the proof to the case where $f$ is surjective. If $f:\mb{F}_2^{\ell}\to \mb{F}_2^m$ is not surjective, there exists an invertible transformation $T_1\in \Aff(\mb{F}_2^m)$ such that $T_1\co f(\mb{F}_2^{\ell})=\mb{F}_2^{\ell'}\times 0^{m-\ell'}$. Thus, if $p:\mb{F}_2^m\to \mb{F}_2^{\ell'}$, $p(v_1,\ldots,v_m):=(v_1,\ldots,v_{\ell'})$ and $i:\mb{F}_2^{\ell'}\to \mb{F}_2^m$, $i(v):=(v,0^{m-\ell'})$ we have that $f = T_1^{-1}\co i\co p\co T_1 \co f$. Now suppose that $\q\in \cu^m(\ns_k)$ and we have already proved the result for the case of surjective affine maps. Then $\q \co f = \q \co T_1^{-1}\co i\co p\co T_1 \co f$. But note that $\q\co T_1^{-1}\in \cu^m(\ns_k)$ by hypothesis. Then, since $i$ is a discrete-cube morphism, we have $(\q\co T_1^{-1})\co i\in \cu^{\ell'}(\ns_k)$. Finally, note that $p\co T_1 \co f$ is an affine surjective morphism and thus by hypothesis $(\q\co T_1^{-1}\co i) \co (p\co T_1 \co f)\in \cu^{\ell}(\ns_k)$. An analogous argument shows that we can further reduce the proof to the case where $f$ is also injective and then we are done (since we have then reduced to the case of a bijective affine map, which was addressed above). This proves the claim.

We can now conclude that $\cu^n(\ns_k)=\hom(\mb{F}_2^n,\ns_k)$ and thus $\ns_k$ is 2-homogeneous. As $\ns$ is the inverse limit of $\ns_k$, we have that $\ns$ is also 2-homogeneous.
\end{proof}

\begin{corollary}\label{cor:2-hom-aff-exch}
Let $(\ab,\ab_\bullet)$ be a compact filtered abelian group and let $\ns$ be the associated compact group nilspace. Then $\mu_{\cu^{\omega}(\ns)}$ \textup{(}seen as a measure on $\ns^{\db{\mb{N}}}$\textup{)} is affine-exchangeable if and only if the filtration $\ab_\bullet$ is 2-homogeneous.
\end{corollary}

\begin{proof}
By Lemma \ref{lem:aff-ech-equals-2-hom} the measure $\mu_{\cu^\omega(\ns)}$ is affine-exchangeable if and only if $\ns$ is 2-homogeneous. But for group nilspaces we know that being 2-homogeneous is equivalent to the associated filtration being 2-homogeneous, by \cite[Theorem 3.8]{CGSS-p-hom}.
\end{proof}

\section{On cube-surjective continuous morphisms}\label{app:cube-sur-implies-fibration}

\noindent Recall that a map $\varphi$ between two nilspaces $\ns,\nss$ is \emph{cube-surjective} if for every $n\geq 0$ the map $\varphi^{\db{n}}$ is surjective from $\cu^n(\ns)$ onto $\cu^n(\nss)$. 
The goal of this section is to prove the following result.

\begin{theorem}\label{thm:cube-sur-implies-fib}
Let $\ns$, $\nss$ be compact profinite-step $p$-homogeneous nilspaces and suppose that $\varphi:\ns\to\nss$ is a continuous cube-surjective morphism. Then $\varphi$ is a fibration.
\end{theorem}

\noindent To prove this we shall need the following result about finite nilspaces.

\begin{theorem}\label{thm:cube-sur-implies-fib-2} 
Let $\ns$ be a $k$-step finite nilspace, let $\nss$ be a $d$-step nilspace, and let $\varphi:\ns\to \nss$ be a cube-surjective morphism. Then $\varphi$ is a fibration. In particular $d\le k$.
\end{theorem}

Indeed, we first need to prove the following particular case:

\begin{proposition}\label{prop:diff-steps-impossible}
Let $\ns$ be a $k$-step finite nilspace, let $\nss$ be a $d$-step nilspace with $d>k$, and suppose that $\nss$ is not $(d-1)$-step. Then no morphism $\varphi:\ns\to \nss$ can be cube-surjective.
\end{proposition}

\begin{proof} 
First note that $\nss$ must be finite as otherwise the result is trivial (since cube-surjective maps are in particular surjective). For $i\in [k]$ and $j\in[d]$ let $\ab_i(\ns)$, $\ab_j(\nss)$ be the structure groups of $\ns$ and $\nss$ respectively. We want to compute the size of $\cu^n(\ns)$. Assume that $n>k$.

A possible way of doing this consists in first assigning values to the vertices $v\in\db{n}$ in an order given by $|v|:=v\sbr{1}+\cdots+v\sbr{n}$. That is, note that the possible values for $\q(0^n)$ are all possible elements of $\ns$. Thus we have $|\ns|=|\ab_1(\ns)|\cdots|\ab_k(\ns)|$ possibilities for $\q(0^n)$. Next, for any $v\in\db{n}$ with $|v|=1$ note that again we have $\ns$ possibilities for each such $\q(v)$. By the ergodicity axiom all these maps are morphisms from $\{v\in\db{n}:|v|\le1\}$ to $\ns$ (see \cite[Definition 3.1.4]{Cand:Notes1}). Hence, in order to assign a value to all of these vertices we have $|\ns|^n=(|\ab_1(\ns)|\cdots|\ab_k(\ns)|)^n$ possibilities.

Take now any element $w\in \db{n}$ with $|w|=2$. Let us say that $v\le w$ for $v\in\db{n}$ if for all $i\in[n]$, $v\sbr{i}\le w\sbr{i}$. The map $f:\{v\in\db{n}:v\le w\}\setminus\{w\}\to \ns$ defined by $v\mapsto \q(v)$ is then a corner in $\cor^2(\ns)$ (identifying $\{v\in\db{n}:v\le w\}$ with $\db{|w|}$). We are interested in how many $x\in\ns$ there are such that the map $\{v\in\db{n}:v\le w\}\to\ns$ defined as $v\mapsto \q(v)$ if $v\not=w$ and $w\mapsto x$ is a cube. Clearly if we have two such  values $x_1,x_2\in \ns$ then by unique completion in $\ns_1$ we have $\pi_1(x_1)=\pi_1(x_2)$. But the converse is also true by \cite[Lemma 3.2.7]{Cand:Notes1} and thus the number of possible $x\in\ns$ that complete $f$ equals the size of the fiber $\pi_1^{-1}(\pi_1(x_1))$ for some possible $x_1\in\ns$ that completes $f$. Hence, for each $w\in\db{n}$ with $|w|=2$ we have $|\pi_1^{-1}(\pi_1(x_1))|=|\ab_2(\ns)|\cdots|\ab_k(\ns)|$ possibilities (this follows from \cite[Lemma 3.3.6]{Cand:Notes1} and \cite[Theorem 3.2.19]{Cand:Notes1}). As we have $\binom{n}{2}$ such possible $w\in\db{n}$ we have a total of $(|\ab_2(\ns)|\cdots|\ab_k(\ns)|)^{\binom{n}{2}}$ possibilities once we have fixed the values for $|v|\le 1$.

Repeating this process, as $n>k$ we have
\[
|\cu^n(\ns)|=\underbrace{|\ns|}_{\#\{\text{pos. for } 0^n\}}\underbrace{|\ns|^{n}}_{\#\{\text{pos. for } |v|=1\}}\underbrace{(|\ab_2(\ns)|\cdots|\ab_k(\ns)|)^{\binom{n}{2}}}_{\#\{\text{pos. for } |v|=2\}}\cdots \underbrace{|\ab_k(\ns)|^{\binom{n}{k}}}_{\#\{\text{pos. for } |v|=k\}}.
\]
Similarly, for $n>d$ we have $|\cu^n(\nss)|=|\nss|^{n+1}(|\ab_2(\nss)|\cdots|\ab_d(\nss)|)^{\binom{n}{2}}\cdots |\ab_d(\nss)|^{\binom{n}{d}}$. But it is clear that for $n$ large enough, as $\ab_d(\nss)\not=\{\id\}$ we have $|\cu^n(\ns)|<|\cu^n(\nss)|$, whence $\varphi^{\db{n}}$ cannot be surjective. 
\end{proof}

\noindent Let us recall the following construction of a nilspace modulo a subgroup of the last structure group.
\begin{proposition}\label{prop:H-nil}\cite[Proposition A.20]{CGSS-p-hom} 
Let $\nss$ be a $k$-step nilspace and let $H<\ab_k(\nss)$ be any subgroup. Let us define the following relation on $\nss$: for $y_1,y_2\in\nss$, we write $y_1\sim y_2$ if and only if $y_1=y_2+h$ for some $h\in H$. Then the following statements hold:
\begin{enumerate}[leftmargin=0.8cm]
    \item The relation $\sim$ is an equivalence relation.
    \item The set $\tilde{\nss}:=\nss/\sim$ together with the sets $\cu^n(\tilde{\nss}):=\{\pi_{\sim}\co \q:\q\in\cu^n(\nss)\}$ is a nilspace.
    \item $\tilde{\nss}$ is $k$-step, with last structure group $\ab_k(\tilde{\nss})=\ab_k(\nss)/H$, and $\tilde{\nss}_{k-1}\simeq \nss_{k-1}$.
\end{enumerate}
\end{proposition}

\begin{proof}[Proof of Theorem \ref{thm:cube-sur-implies-fib-2}] 
We argue by induction on $k$. The case $k=0$ is trivial. To prove the inductive step, by Proposition \ref{prop:diff-steps-impossible} we know that $d\le k$. If $d\le k-1$ then $\varphi$ factors through $\pi_{k-1}$, i.e.\ $\varphi = \varphi_{k-1}\co \pi_{k-1}$. As $\pi_{k-1}$ is a fibration it follows that $\varphi_{k-1}:\ns_{k-1}\to \nss$ is cube-surjective and thus by induction we have that $\varphi_{k-1}$ is a fibration. Hence $\varphi$ is a fibration (as a composition of two fibrations).

If $d=k$, then the morphism $\varphi_{k-1}$ satisfying $\pi_{k-1}\co \varphi = \varphi_{k-1}\co \pi_{k-1}$ is cube-surjective (similarly as in the previous paragraph), so by induction it is a fibration. Therefore it suffices to check that the last structure homomorphism $\phi_k:\ab_k(\ns)\to \ab_k(\nss)$ is surjective. Let $H:=\phi_k(\ab_k(\ns))<\ab_k(\nss)$ and define the nilspace $\tnss$ as in Proposition \ref{prop:H-nil}. Now let $\psi:\ns_{k-1}\to \tnss$ be defined as $\pi_{k-1}(x)\mapsto\pi_{\sim}(\varphi(x))$. We claim that this is a cube-surjective morphism. To see this, we first check that $\psi$ is well-defined: for any $z\in \ab_k(\ns)$ and $x\in \ns$ we have $\psi(\pi_{k-1}(x+z))=\pi_{\sim}(\varphi(x+z))=\pi_{\sim}(\varphi(x)+\phi_k(z))=\pi_{\sim}(\varphi(x))=\psi(\pi_{k-1}(x))$. To see that $\psi$ is a morphism note that for every $\q\in \cu^n(\ns)$ we have $\psi\co\pi_{k-1}\co \q = \pi_{\sim}\co \varphi\co \q\in \cu^n(\tnss)$. Now, to see that $\psi$ is cube-surjective, let $\pi_{\sim}\co \q^*$ be any element of $\cu^n(\tnss)$ (where $\q^*\in \cu^n(\nss)$). As $\varphi$ is cube-surjective, there is $\q'\in \cu^n(\ns)$ such that $\varphi\co \q'=\q^*$. Thus $\psi\co\pi_{k-1}\co \q' = \pi_{\sim}\co \varphi \co \q' = \pi_{\sim} \co \q^*$.

We thus have a cube-surjective morphism $\psi$ from a $(k-1)$-step nilspace $\ns_{k-1}$ to a $k$-step nilspace $\tnss$. By Proposition \ref{prop:diff-steps-impossible}, the step of $\tnss$ is at most $k-1$, so $\ab_k(\tnss)=\{0\}$. But by part $(iii)$ of Proposition \ref{prop:H-nil}, we know that $\{0\}=\ab_k(\tnss) = \ab_k(\nss)/H = \ab_k(\nss)/\phi_k(\ab_k(\ns))$. This implies that $\phi_k(\ab_k(\ns))=\ab_k(\nss)$, whence $\phi_k$ is surjective as required. 
\end{proof}

\begin{proof}[Proof of Theorem \ref{thm:cube-sur-implies-fib}] 
Note that by \cite[Proposition 1.5]{CGSS-p-hom} we know that all $p$-homogeneous, \textsc{cfr} (i.e.\ compact and finite-rank) nilspaces are finite. Hence we can apply Theorem \ref{thm:cube-sur-implies-fib-2} to any cube-surjecive morphism between \textsc{cfr} $p$-homogeneous nilspaces of finite step.

For the general case of the theorem, suppose $\varphi:\ns\to\nss$ is continuous and cube-surjective. Let $\nss = \varprojlim  \nss_i$ be an inverse limit decomposition of $\nss$ with each $\nss_i$ of finite rank and $i$-step, and limit maps $\psi_{i,\nss}$ \cite[Theorem 5.71]{GS}. Fix any $k\in \mb{N}$ and note that $\psi_{k,\nss}\co\varphi$ is a cube-surjective continuous morphism $\ns\to \nss_k$. In particular, this map factors through the $k$-th characteristic factor of $\ns$, i.e.\ $\psi_{k,\nss}\co\varphi =(\psi_{k,\nss}\co\varphi)_k\co \pi_k$, and $(\psi_{k,\nss}\co\varphi)_k$ is cube-surjective as well. Now let $\pi_k(\ns)=\varprojlim (\pi_k(\ns))_i$\footnote{In this proof we denote the $k$-characteristic factor of $\ns$ by $\pi_k(\ns)$ to avoid confusion with the sub-indices used for the inverse limit.} be an inverse limit decomposition of $\pi_k(\ns)$ with each $(\pi_k(\ns))_i$ of finite rank and $k$-step and factor maps $\psi_{k,\ns_k,i}:\pi_k(\ns)\to (\pi_k(\ns))_i$. By \cite[Theorem 1.7]{CGSS}, there exists $j\in \mb{N}$ such that $(\psi_{k,\nss}\co\varphi)_k = \varphi'\co\psi_{k,\ns,j}$ for some morphism $\varphi':(\pi_k(\ns))_j\to\nss_k$. It follows from the cube surjectivity of $(\psi_{k,\nss}\co\varphi)_k$ that $\varphi'$ is also cube-surjective. But now, since $(\pi_k(\ns))_j$ and $\nss_k$ are both finite-rank, $k$-step, and $p$-homogeneous it follows from \cite[Proposition 1.5]{CGSS} that these nilspaces are finite. Hence $\varphi'$ is a fibration by Theorem \ref{thm:cube-sur-implies-fib-2}, and so (since $\psi_{j,\ns}$ is also a fibration), we conclude that $(\psi_{k,\nss}\co\varphi)_k = \varphi'\co\psi_{j,\ns}$ is a fibration. Finally, note that in particular $\psi_{k,\nss}\co\varphi=(\psi_{k,\nss}\co\varphi)_k\co\pi_k$ is a fibration as well.

Now from the previous paragraph we have to deduce that $\varphi$ itself is a fibration. We know that $\psi_{k,\nss}\co\varphi$ is a fibration for every $k$. Let $\q'$ be an $n$-corner on $\ns$ and suppose $\tilde \q\in \cu^n(\nss)$ completes the corner $\varphi\co\q'\in\cor^n(\nss)$. For every $k$, let $Q_k$ denote the set of $n$-cubes on $\ns$ completing $\q'$ and whose image under $\psi_{k,\nss}\co\varphi$ equals $\psi_{k,\nss}\co\tilde \q$. Note that $Q_k$ is non-empty for every $k$ because $\psi_{k,\nss}\co\varphi$ is a fibration. Note also that $Q_k$ is closed (hence compact) by standard facts (it is the intersection of the preimage of the cube $\psi_{k,\nss}\co\tilde \q$ under $\psi_{k,\nss}\co\varphi^{\db{n}}$ with the set of cubes completing $\q'$, so it is the intersection of two closed sets). Finally, note that $Q_k\supset Q_{k+1}$ for all $k\ge 0$. It then follows from the finite intersection property that $\cap_{k=1}^\infty Q_k$ is non-empty. Then any cube $\q$ in this intersection completes $\q'$ (by the closeness of $\cu^n(\ns)$), and $\varphi\co \q=\tilde \q$ because by construction we have this equality modulo $\psi_{k,\nss}$ for every $k$. This proves that $\varphi$ is a fibration.
\end{proof}

\section{Maps preserving the cubic Haar measures are essentially cube-surjective morphisms}\label{app:proof-tau-cont}

\noindent In this section we prove the following result used in the proof of Theorem \ref{thm:repre-with-2-hom}.
\begin{theorem}\label{thm:cube-surjective-app}
Let $\ns,\nss$ be compact profinite-step nilspaces, and suppose that $\ns$ is the group nilspace associated with a compact abelian filtered group $(\ab,\ab_\bullet)$. Let $\tau:\nss\to\ns$ be a Borel map such that $\tau^{\db{n}}$ preserves the $n$-cubic Haar measures for all $n\ge 0$. Then there is a continuous cube-surjective morphism $\varphi:\nss\to\ns$ such that $\tau=\varphi$  $\mu_{\nss}$-a.e.
\end{theorem}

\noindent The proof uses the following stability result, whose main ideas are adapted from \cite[\S 4]{CSinverse}.

\begin{lemma}\label{lem:l1-small-implies-constant}
Let $\nss$ be a compact profinite-step nilspace and let $\ns$ be a $k$-step compact group nilspace associated with some compact abelian filtered group $(\ab,\ab_\bullet)$. Let $d$ be any compatible\footnote{Meaning that the metric generates the compact second-countable topology on $\ab$.} metric on $\ab$. Then there exists $\delta=\delta(\ns,d)>0$ such that if $\psi:\nss\to\ns$ is a continuous morphism with $d_1(\psi,0):=\int d(\psi(y),0)\,\ud\mu_{\nss}(y)<\delta$ then $\psi$ is constant.
\end{lemma}

\begin{proof}
First note that since $\ns$ is $k$-step we have $d_1(\psi,0)=\int d(\psi_k(y),0)\,\ud\mu_{\nss_k}(y)$ where $\psi_k:\nss_k\to\ns$ is such that $\psi=\psi_k\co\pi_k$. Hence it suffices to prove the result assuming that $\nss$ is $k$-step as well.

Let $\delta_0>0$ be a constant to be fixed later. Note that by Markov's inequality, the set of points $y\in\nss$ such that $d(\psi(y),0)>\delta_0^{1/2}$ has Haar measure at most $\delta_0^{1/2}$. Fix some ${y_0}\in\nss$ and consider the set $\cu_{y_0}^{k+1}(\nss):=\{ \q\in \cu^{k+1}(\nss) : \q(0^{k+1})= y_0\}$. For every $v\in \db{k+1}\setminus\{0^{k+1}\}$, the coordinate projection $p_v:\cu_{y_0}^{k+1}(\nss)\to\nss$, $\q\mapsto \q(v)$ preserves the Haar measure (see \cite[Lemma 2.2.17]{Cand:Notes2} for the definition of the Haar measure on $\cu^{k+1}_{y_0}(\nss)$, and for the measure-preserving property of the projection apply \cite[Lemma 2.2.14]{Cand:Notes2} with the good pair $P_1=\{0^{k+1}\}$, $P_2=\{v\}$ inside $P=\db{k+1}$). Thus $\delta_0^{1/2}>\mu_{\nss}(\{y\in\ns:d(\psi(y),0)>\delta_0^{1/2}\}) = \mu_{\cu_{y_0}^{k+1}(\nss)}\co p_v^{-1}(\{y\in\nss:d(\psi(y),0)>\delta_0^{1/2}\}) = \mu_{\cu_{y_0}^{k+1}(\nss)}(\{\q\in\cu_{y_0}^{k+1}(\nss):d(\psi(\q(v)),0)>\delta_0^{1/2}\})$. Hence 
\begin{eqnarray*}
&& \mu_{\cu_{y_0}^{k+1}(\nss)}(\{\q\in \cu_{y_0}^{k+1}(\nss):\forall \;v\in\db{k+1}\setminus\{0^{k+1}\},d(\psi(\q(v)),0)\le \delta_0^{1/2}\}\\
& = & 1-\mu_{\cu_{y_0}^{k+1}(\nss)}(\{\q\in \cu_{y_0}^{k+1}(\nss):\exists \;v\in\db{k+1}\setminus\{0^{k+1}\},d(\psi(\q(v)),0)> \delta_0^{1/2}\}\\
&\ge & 1-\sum_{v\in\db{k+1}\setminus\{0^{k+1}\}} \mu_{\cu_{y_0}^{k+1}(\nss)}(\{\q\in \cu_{y_0}^{k+1}(\nss):d(\psi(\q(v)),0)> \delta_0^{1/2}\} \;\; \ge \;\; 1-2^{k+1}\delta_0^{1/2}.
\end{eqnarray*}
\noindent In particular if $\delta_0$ is small enough (depending only on $k$) there exists $\q_{y_0}\in \cu_{y_0}^{k+1}(\nss)$ such that $d(\psi(\q_{y_0}(v)),0)\le \delta_0^{1/2}$ for all $v\in \db{k+1}\setminus\{0^{k+1}\}$.

Now we recall that by \cite[Lemma 2.1.12]{Cand:Notes2} the map $K:\cor^{k+1}(\ns)\to \ns$ that sends a corner to its unique completion is continuous. As both $\cor^{k+1}(\ns)$ and $\ns$ are compact spaces, this function is uniformly continuous and thus for every $\epsilon>0$ there exists $\tau>0$ such that for every $\q'\in\cor^{k+1}(\ns)$ if $d(\q'(v),0)<\tau$ for all $v\in \db{k+1}\setminus\{0^{k+1}\}$ we have $d(K(\q'),0)<\epsilon$. Then, for some $\epsilon_0>0$ to be fixed later, let  $\delta_0>0$ be small enough (depending only on $\ns$ and $k$) so that $\delta_0^{1/2}<\tau_0(\epsilon_0)$. Then, by the above facts applied to $\psi\co \q_{y_0}|_{\db{k+1}\setminus\{0^{k+1}\}}\in \cor^{k+1}(\ns)$, we have $d(K(\psi\co|_{\db{k+1}\setminus\{0^{k+1}\}}),0)<\epsilon_0$. But $\ns$ is $k$-step and clearly $\psi\co\q_{y_0}$ is a cube completing $\psi\co \q_{y_0}|_{\db{k+1}\setminus\{0^{k+1}\}}$. Hence $K(\psi\co\q_{y_0}|_{\db{k+1}\setminus\{0^{k+1}\}})=\psi(y_0)$ and therefore $d(\psi(y_0),0)<\epsilon_0$. Let us emphasize that for any $\epsilon_0$ we can choose an appropriate $\delta_0=\delta_0(\ns,k)$ such that the previous bound holds \emph{for every} $y_0\in\nss$.

We now choose $\epsilon_0$ small enough (depending only on $\ns$ and $k$) in such a way that  \cite[Lemma 3.4]{CGSS} is satisfied. Hence $\psi$ is constant and the result follows with $\delta=\delta_0$.
\end{proof}

\begin{proof}[Proof of Theorem \ref{thm:cube-surjective-app}]
We start with the assumption that there is a Borel map $\tau:\nss\to \ns$ such that for every integer $n\ge 0$ we have $\mu_{\cu^n(\nss)}\co (\tau^{\db{n}})^{-1}=\mu_{\cu^n(\ns)}$, and our aim is to prove that there is a continuous cube-surjective morphism $\nss\to \ns$ that is equal to $\tau$ almost surely relative to $\mu_{\nss}$.

Recall that the nilspace $\ns$ is the abelian group nilspace generated by some filtered abelian group $(\ab,\ab_\bullet=(\ab_{\sbr{k}})_{k=0}^{\infty})$. By standard results $\ab$ is the inverse limit (in the category of abelian groups) of compact abelian Lie groups, $\ab = \varprojlim \ab^i$. Let $p_i:\ab\to\ab^i$ be the projection maps. It can be shown that then the nilspace $\ns$ is the inverse limit of compact nilspaces that we shall denote by $\ns_i$ ($i\ge 1$), where $\ns_i$ is the abelian group nilspace associated with the filtered group $(\ab^i/p_i(\ab_{\sbr{i}}),(p_i(\ab_{\sbr{j}})/p_i(\ab_{\sbr{i}}))_{j=0}^\infty)$. Note that each $\ns_i$ is an $i$-step \textsc{cfr} nilspace, and let us denote by $\psi_i$ the limit map $\ns\to\ns_i$. Hence $\ns=\varprojlim\ns_i$ (as nilspaces).

Note that $\psi_j\co\tau:\nss\to\ns_j$ is Borel measurable and $\mu_{\cu^n(\ns_j)} = \mu_{\cu^n(\ns)}\co (\psi_j^{\db{n}})^{-1} = \mu_{\cu^n(\nss)}\co (\tau^{\db{n}})^{-1}\co (\psi_j^{\db{n}})^{-1} = \mu_{\cu^n(\nss)} \co ((\psi_j\co\tau)^{\db{n}})^{-1}$ for all $n\ge 0$. We shall now prove that $\psi_j\co\tau$ agrees $\mu_{\nss}$-almost-surely with some cube-surjective continuous morphism.

Following \cite[\S 4]{CSinverse}, for compact profinite-step nilspaces $\ns,\nss$ with a fixed metric $d_{\ns}$ compatible with the topology on $\ns$, let us say that a map $\phi':\nss\to\ns$ is a \emph{$(\delta,1)$-quasimorphism} of \emph{degree $k-1$} if it is a Borel map with the following property:
\begin{equation}\label{eq:quasimorph}
\mu_{\cu^k(\nss)}\big(\big\{\q\in\cu^k(\nss): \;\exists\, \q'\in \cu^k(\ns), \,\forall\, v\in \db{k},\, d_{\ns}\big(\phi'\co\q(v),\q'(v)\big)\leq \delta\;\big\}\big)\geq 1-\delta.
\end{equation}
For Borel maps $\phi,\psi:\nss\to \ns$ we define $d_1(\phi,\psi):=\int_{\nss} d_{\ns}\big(\phi(x),\psi(x)\big)\ud\mu_{\nss}(x)$.

Fix any $j\ge 0$. For any $\delta>0$ note that $\psi_j\co \tau:\nss\to\ns_j$ is a $(\delta,1)$-quasimorphism of degree $j$. As $\ns_j$ are abelian group nilspaces, we fix a metric $d_{\ns_j}$ on $\ns_j$ that is invariant under addition. By \cite[Theorem 4.2]{CSinverse} there is a continuous morphism $\phi_{\delta,j}:\nss\to \ns_j$ such that $d_1(\psi_j\co\tau,\phi_{\delta,j})\leq \varepsilon$, where $\varepsilon(\delta)\to 0$ as $\delta\to 0$. Applying this for each $\delta_n:=1/n$, $n\in \mb{N}$, we obtain a sequence of continuous morphisms $\phi_{n,j}:\nss\to \ns_j$ such that $d_1(\psi_j\co\tau,\phi_{n,j})\to 0$ as $n\to\infty$. By the triangle inequality this implies that $d_1(\phi_{m,j},\phi_{n,j})=d_1(\phi_{m,j}-\phi_{n,j},0)\to 0$ as $m,n\to\infty$. 

By Lemma \ref{lem:l1-small-implies-constant} we have that $\phi_{m,j}-\phi_{n,j}$ is constant for $n,m\ge N_0$ for some large $N_0=N_0(\ns_j)$. In particular, letting $z_{j,n,m}:=\phi_{m,j}-\phi_{n,j}$ for $n,m\ge N_0$, by compactness of $\ns_j$ we can assume (passing to a subsequence if necessary and relabeling it as $(z_{j,n,m})$), that $z_{j,N_0,m}\to z^*_{j,N_0}$ as $m\to\infty$. Then $d_1(\psi_j\co\tau,\phi_{m,j})=d_1(\psi_j\co\tau,\phi_{N_0,j}+(\phi_{m,j}-\phi_{N_0,j}))$, where the left hand side converges to 0 as $m\to\infty$ by the previous paragraph, and the right hand side converges to $d_1(\psi_j\co\tau,\phi_{N_0,j}+z^*_{j,N_0})$ by construction (and the dominated convergence theorem). Thus $\psi_j\co\tau = \phi_{N_0,j}+z^*_{j,N_0}$ $\mu_{\nss}$-a.e.. Let $\gamma_j:=\phi_{N_0,j}+z^*_{j,N_0}$ and note that this is a continuous morphism $\nss\to\ns_j$ by construction.

We now \emph{glue} all these maps $\gamma_j$ into a single continuous cube-surjective morphism $\gamma:\nss\to\ns$. The reduces to proving that the maps $\gamma_j$ are consistent with the given inverse limit expression $\ns=\varprojlim\ns_j$, that is, given the corresponding factor maps $\psi_{i,j}:\ns_j\to\ns_i$ in this inverse limit, we want to prove that for every $j\ge i\ge 0$ we have $\gamma_i =\psi_{i,j}\co\gamma_j$. Since $\psi_j:\ns\to\ns_j$ are the limit maps, we have $\psi_i=\psi_{i,j}\co\psi_j$. Thus $\psi_i\co\tau = \psi_{i,j}\co\psi_j\co \tau$ but $\psi_i\co\tau=\gamma_i$ and $\psi_j\co\tau=\gamma_j$ $\mu_{\nss}$-a.e.. Hence $\gamma_i = \psi_{i,j}\co\gamma_j$ $\mu_Y$-a.e.. As $\ns_j$ is $j$-step, note that $\{y\in\nss:\gamma_i(y) = \psi_{i,j}\co\gamma_j(y)\}=\pi_j^{-1}(\{\tilde{y}\in\pi_j(\nss):(\gamma_i)_j(\tilde{y}) = \psi_{i,j}\co(\gamma_j)_j(\tilde{y})\})$ where $\pi_j(\nss)$ is the $j$-th characteristic factor\footnote{Here we denote by $\pi_j(\nss)$ the $j$-th characteristic factor of $\nss$ to avoid the confusion with $\ns_j$, which in this argument is \emph{not necessarily} the $j$-th characteristic nilspace factor of $\ns$.} of $\nss$, $(\gamma_i)_j:\pi_j(\nss)\to\ns_i$ and $(\gamma_j)_j:\pi_j(\nss)\to\ns_j$.\footnote{As $\gamma_i:\nss\to \ns_i$ is a morphism and $\ns_i$ is $i$-step, this map factors through the $j$-th characteristic factor of $\nss$ for $j\ge i$. We denote this map by $(\gamma_i)_j:\pi_{j}(\nss)\to\ns_i$. Similarly for $(\gamma_j)_j$.} As $\pi_j$ preserves the Haar measure we have that $(\gamma_i)_j = \psi_{i,j}\co(\gamma_j)_j$ $\mu_{\pi_j(\nss)}$-a.e.. Therefore $(\gamma_i)_j$ and $\psi_{i,j}\co(\gamma_j)_j$ are continuous functions that agree $\mu_{\pi_j(\nss)}$-almost-surely. If there was a point $\tilde{y}_0\in\pi_j(\nss)$ such that $(\gamma_i)_j(\tilde{y}_0) \not= \psi_{i,j}\co(\gamma_j)_j(\tilde{y}_0)$ by continuity we would have an open set $U\ni \tilde{y}_0$ such that $(\gamma_i)_j(\tilde{y}) \not= \psi_{i,j}\co(\gamma_j)_j(\tilde{y})$ for all $\tilde{y}\in U$. This would contradict the almost-sure equality just established, because $\mu_{\pi_j(\nss)}(U)>0$ by \cite[Proposition 2.2.11]{Cand:Notes2}. Hence $\gamma_i= \psi_{i,j}\co\gamma_j$, and we can now define $\gamma:\nss\to\ns=\varprojlim\ns_i$ as $y\mapsto (\gamma_j(y))_{j=1}^\infty$. Clearly this is a continuous morphism.

It remains to check that $\gamma$ is cube-surjective. First we claim that $\gamma=\tau$ $\mu_{\nss}$-a.e.. In order to prove this note that $\{y\in\nss:\tau(y)\not= \gamma(y)\}=\cup_{j=1}^\infty\{y\in\nss:\psi_j\co\tau(y)\not=\gamma_j(y)\}$. As the latter set has measure 0 for every $j$, the claim follows. Next we claim that for all $n\ge 0$, $\mu_{\cu^n(\nss)}\co (\gamma^{\db{n}})^{-1} = \mu_{\cu^n(\ns)}$. Indeed note that by the consistency axiom for cubic couplings it follows that $\gamma^{\db{n}}=\tau^{\db{n}}$ $\mu_{\cu^n(\nss)}$-a.e., and thus we have the desired measure preserving property. Finally we want to prove that for every $n\ge 0$ we have $\gamma^{\db{n}}(\cu^n(\nss))=\cu^n(\ns)$. Suppose for a contradiction that there exists $\q\in \cu^n(\ns)$ which is not in the image of $\gamma^{\db{n}}$. This readily implies that for some $j\ge 1$ we have $\pi_j\co \q\notin (\gamma)_j^{\db{n}}(\cu^n(\pi_j(\nss)))$ where $(\gamma)_j:\pi_j(\nss)\to \pi_j(\ns)$ is the map such that $\pi_j\co\gamma = (\gamma)_j\co\pi_j$ (this map $(\gamma)_j$ is \emph{not to be confused} with the map $\gamma_j:\nss\to\ns_j$ from the previous paragraph). As $(\gamma)_j^{\db{n}}$ is continuous and $\cu^n(\pi_j(\nss))$ is compact, we have that $(\gamma)_j^{\db{n}}(\cu^n(\pi_j(\nss)))$ is a compact and hence closed set. If $\pi_j\co\q\notin (\gamma)_j^{\db{n}}(\cu^n(\pi_j(\nss)))$ then there must exists an open set $ V\ni \pi_j\co\q$ in $\cu^n(\pi_j(\ns))$ such that $V\cap (\gamma)_j^{\db{n}}(\cu^n(\pi_j(\nss)))=\emptyset$. But then $\mu_{\cu^n(\pi_j(\ns))}(V)>0$ by (an argument similar to) \cite[Lemma 2.2.11]{Cand:Notes2} and $\mu_{\cu^n(\pi_j(\nss))}\co ((\gamma)_j^{\db{n}})^{-1}(V) = \mu_{\cu^n(\pi_j(\nss))}(\emptyset)=0$. Since for every $j$ we have that $\mu_{\cu^n(\pi_j(\ns))}=\mu_{\cu^n(\pi_j(\nss))}\co ((\gamma)_j^{\db{n}})^{-1}$ (which follows by composing $\mu_{\cu^n(\nss)}\co (\gamma^{\db{n}})^{-1} = \mu_{\cu^n(\ns)}$ with $(\pi_j^{\db{n}})^{-1}$ on both sides), we have a contradiction.
\end{proof}

\section{Translation between polynomial and nilspace viewpoints}\label{app:correspondence}

\noindent In this appendix our final goal is to explain in more detail how the notion of limit object for convergent sequences from \cite{HHH} relates to the definition of limit domain in Section \ref{sec:exch-limits}. The former definition is phrased using polynomials (see e.g.\ Definition 3.2 in \cite{HHH}), and thereby pertains to an approach using polynomials and related tools used previously in \cite{BTZ, BTZ2, TZ-High, T&Z-Low}, whereas the definitions in Section \ref{sec:exch-limits} involve the nilspace approach, which in the characteristic-$p$ setting was deployed in \cite{CGSS-p-hom}. Thus, towards our final goal in this appendix we also take the opportunity to explain how certain key concepts from these two approaches can be related to each other.

First it is worth recalling an important relation between the notion of polynomial map and that of a nilspace morphism, which is that the two concepts are equivalent when the nilspaces involved are \emph{group} nilspaces (see \cite[\S 2.2.2]{Cand:Notes1}). In the more specific characteristic-$p$ setting that occupies us here, let us begin by recalling from \cite[Definition 1.1]{T&Z-Low} the concept of non-classical polynomials on vector spaces over  $\mb{F}_p$.

\begin{defn}[Non-classical polynomials over $\mb{F}_p$] Let $p$ be a prime, let $n,d\ge 0$ be integers, and let $G$ be an abelian group. A function $P:\mb{F}_p^n\to G$ is a \emph{non-classical polynomial} of degree $\le d$ if $\Delta_{h_1}\cdots\Delta_{h_{d+1}} P(x)=0$, where $\Delta_h f(x):=f(x+h)-f(x)$ for any $f:\mb{F}_p^n\to G$ and any $h,x\in \mb{F}_p^n$.
\end{defn}
\noindent This notion is of particular interest when $G=\mb{T}$. The image of any such polynomial $P:\mb{F}_p^n\to \mb{T}$ is included in a coset of $(\frac{1}{p^r}\cdot \mb{Z})/\mb{Z}\subset \mb{T}$. The least $r\ge 0$ with such property is called the \emph{depth} of the polynomial $P$ (see \cite[Lemma 2.5]{HHH}).

Our first observation relating the polynomial and nilspace languages here is the following lemma.

\begin{lemma}\label{lem:poly-morph-corresp}
A function $P:\mb{F}_p^n\to\mb{T}$ is a polynomial of degree $\leq k$ and depth $r$ with $P(0)=0$ if and only if $P$ is a morphism from $\mc{D}_1(\mb{F}_p^n)$ to the group nilspace $\abph_{k,\ell}$ \textup{(}embedded in $\mb{T}$\textup{)} where $\ell=k-(r+1)(p-1)$.
\end{lemma}
\noindent To prove this we use the following description of morphisms from $\mb{Z}_p^n$ into the circle group.

\begin{proposition}\label{prop:k-poly-char}
Let $\phi\in\hom(\mc{D}_1(\mb{Z}_p^n),\mc{D}_k(\mb{T}))$ with $\phi(\mb{Z}_p^n)\ni 0$. Then, letting $\ell=\ell_{k,p}$ be the residue mod $p$ of $k$ in $[p-1]$ \textup{(}i.e.\ the integer $\ell\in [p-1]$ such that $k-\ell$ is a multiple of $p-1$\textup{)}, and viewing $\abph_{k,\ell}$ as embedded in $\mb{T}$ the natural way, we have $\phi\in\hom(\mc{D}_1(\mb{Z}_p^n),\abph_{k,\ell})$. Furthermore, if $\phi(\mb{Z}_p^n)\subset (\frac{1}{p^r}\cdot \mb{Z})/\mb{Z}$ then $\ell$ can be taken to be $\ell=k-(r+1)(p-1)$ \textup{(}so, in particular, not in $[p-1]$\textup{)}.
\end{proposition}
\begin{proof}
Suppose that $\phi(x_0)=0$. Since translation by $x_0$ is a nilspace automorphism on $\mc{D}_1(\mb{Z}_p^n)$, we can suppose without loss of generality that $x_0=0$.

We argue by induction on $k$. For $k=1$, the assumptions imply that $\phi$ is a group homomorphism $\mb{Z}_p^n\to \mb{T}$, so its image is a subset of $H_{(0)}:=(\frac{1}{p}\cdot \mb{Z})/\mb{Z}$ and $\phi$ is a morphism from $\mc{D}_1(\mb{Z}_p^n)$ to the nilspace determined by the standard degree-1 filtration $H_\bullet^{(1)}=(H_{(0)},H_{(1)},\{0\},\ldots)$, which is indeed $\abph_{1,1}$.

For $k>1$, fix any $\phi\in\hom(\mc{D}_1(\mb{Z}_p^n),\mc{D}_k(\mb{T}))$ with $\phi(0)=0$. 

Fix any elements $h_1,\ldots,h_i\in \mb{Z}_p^n$, $i\geq 1$. Note that $\partial_{h_1}\phi\in\hom(\mc{D}_1(\mb{Z}_p^n),\mc{D}_{k-1}(\mb{T}))$, so by induction $\partial_{h_1}\phi\in \hom(\mc{D}_1(\mb{Z}_p^n),\abph_{k-1,\ell'})$ where $\ell'=\ell_{k-1,p}$. By equality of this morphism set with $\poly(\mc{D}_1(\mb{Z}_p^n),\abph_{k-1,\ell'})$, we have that  $\partial_{h_i}\cdots\partial_{h_1}\phi$ takes values in the $(i-1)$-th term of the filtration $\abph_{k-1,\ell'}$, so 
\begin{equation}\label{eq:induc-p-poly}
\forall x\in \mb{Z}_p^n,\quad \partial_{h_i}\cdots\partial_{h_1}(p\phi)(x)= p \partial_{h_i}\cdots\partial_{h_1}\phi(x)\in p\cdot (\abph_{k-1,\ell'})_{(i-1)} \subset (\abph_{k-1,\ell'})_{i+(p-2)}.
\end{equation}
We have thus shown that $p\phi\in\hom(\mc{D}_1(\mb{Z}_p^n),\abph_{k-1,\ell'}^{+(p-2)})$, where $\abph_{k-1,\ell'}^{+(p-2)}$ denotes the filtration $\abph_{k-1,\ell'}$ shifted forward by $p-2$.
Note that $\abph_{k-1,\ell'}^{+(p-2)}$ is clearly a $p$-homogeneous filtration, and it has degree $k-(p-1)$ (indeed its term with index $k-(p-1)+1 =  k-(p-2)$ is $(\abph_{k-1,\ell'}^{+(p-2)})_{(k-(p-2))}=(\abph_{k-1,\ell'})_{(k)}=\{0\}$). 

We now define the following $p$-homogeneous filtration $H_\bullet$ of degree at most $k$ in $\mb{T}$:
\begin{equation}
H^{(k)}_{(i)}= \left\{\begin{array}{cc}
\frac{1}{p}\cdot (\abph_{k-1,\ell'}^{+(p-2)})_{(i)},& i\in [1,k]\\
\{0\}, & i\geq k+1,
\end{array}\right.
\end{equation}
where note that since $(\abph_{k-1,\ell'}^{+(p-2)})_{(i)}$ is $\{0\}$ for $i> k-(p-1)$, the above definition sets $H^{(k)}_{(i)}$ to be $\frac{1}{p}\cdot\{0\}=(\frac{1}{p}\cdot \mb{Z})/\mb{Z}$ for $i\in [k-(p-1)+1,k]$. Also, note that $H_\bullet$ is $\ell$-fold ergodic for the same $\ell$ for which $\abph_{k-1,\ell'}^{+(p-2)}$ is $\ell$-fold ergodic. But the latter filtration's structure implies that its ergodicity index $\ell$ satisfies $k-(p-1)-\ell = r(p-1)$ for some integer $r$, which is equivalent to $k-\ell= (r+1)(p-1)$, which means that $\ell=\ell_{k,p}$ as required. It follows that $H_\bullet=\abph_{k,\ell}$.

Thus, from \eqref{eq:induc-p-poly} we have that the derivative $\partial_{h_i}\cdots \partial_{h_1} \phi$ takes values in $H^{(k)}_{(i)}$ for each $i\in [1,k]$ and any elements $h_1,\ldots,h_i\in \mb{Z}_p^n$, and for $i>k$ this derivative vanishes since $\phi$ is a morphism into $\mc{D}_k(\mb{T})$. Hence $\phi \in \hom(\mc{D}_1(\mb{Z}_p^n),\abph_{k,\ell})$ as required.

In order to prove the last part of the proposition, note that by \cite[Theorem 3.8]{CGSS-p-hom} or \cite[Lemma 1.6]{T&Z-Low} we can write $\phi(v_1,\ldots,v_n)=\sum_{\underline{t}\in \{0,\ldots,p-1\}^n}w_{|\underline{t}|}\binom{|v_1|_p}{t_1}\cdots \binom{|v_n|_p}{t_n}$ where $\underline{t}=(t_1,\ldots,t_n)$, $|\underline{t}|:=t_1+\cdots+t_n$, $|x|_p\in \{0,\ldots,p-1\}$ is the residue class of $x\in \mb{Z}$ modulo $p$, and some coefficients $w_{|\underline{t}|}\in (\abph_{k,\ell})_{|\underline{t}|}$. If $\phi(\mb{Z}_p^n)\subset p^s \abph_{k,\ell}$ (where here we regard $\abph_{k,\ell}$ as the group $(\mb{Z}/p^{\lfloor\frac{k-\ell}{p-1}\rfloor+1}\mb{Z})$), let $w_{|\underline{t}|}^*\in \abph_{k,\ell+s(p-1)}$ be such that $p^sw_{|\underline{t}|}^* = w_{|\underline{t}|}$ for each $\underline{t}\in \{0,\ldots,p-1\}^n$ (where note that multiplying by $p^s$ is a group homomorphism from $\abph_{k,\ell+s(p-1)}$ to $\abph_{k,\ell}$, seen as abelian groups). Then the map $\phi^*(v_1,\ldots,v_n):=\sum_{\underline{t}\in \{0,\ldots,p-1\}^n}w^*_{|\underline{t}|}\binom{|v_1|_p}{t_1}\cdots \binom{|v_n|_p}{t_n}$ can be proved to be in $\hom(\mc{D}_1(\mb{Z}_p^n),\abph_{k,\ell+s(p-1)})$ and satisfies $p^s\phi^* = \phi$. The result follows.
\end{proof}

\begin{proof}[Proof of Lemma \ref{lem:poly-morph-corresp}]
For the forward implication, from the general correspondence between polynomials and morphisms \cite[\S 2.2.2]{Cand:Notes1} we deduce that $P\in \hom(\mc{D}_1(\mb{F}_p^n), \mc{D}_k(\mb{T}))$, and then Proposition \ref{prop:k-poly-char} gives us that $P\in \hom(\mc{D}_1(\mb{F}_p^n), \abph_{k,\ell})$ (and the claim about the depth $r$ is clear from its definition). The backward implication is clear by the aforementioned correspondence between morphisms and polynomials.
\end{proof} 

\noindent Non-classical polynomials were used in \cite{HHH} to define a notion of \emph{consistency} that underpins the notion of convergence for sequences of functions $(f_n:\mb{F}_p^n\to \{0,1\})_{n\in \mb{N}}$ used in that paper.\footnote{Originally in \cite{HHH} $f_n$ would take values in $\mb{F}_p^{m_n}$ for some sequence $m_n\to\infty$ as $n\to \infty$. For simplicity we have restricted ourselves to the case $m_n=n$.} For $p=2$ this notion will be shown below to coincide with the one given by Definition \ref{def:convergence-notion} (for $\Bo=\{0,1\}$). In order to explain this we first need to set up some notation. From now on, we focus on the case $p=2$. Furthermore, we assume that all non-classical polynomials $P:\mb{F}_2^n\to \mb{T}$ satisfy $0\in P(\mb{F}_2^n)$ and thus $P:\mb{F}_2^n\to (\frac{1}{2^r}\cdot \mb{Z})/\mb{Z}$ for some $r\ge 0$. This can be done without affecting the main results of this section.

We denote by $\mb{F}_2^\omega$ the set of linear forms and by $1\times \mb{F}_2^\omega:=\{v\in \mb{F}_2^\omega:v\sbr{1}=1\}\subset \mb{F}_2^\omega$ the subset of \emph{affine} linear forms (those whose first coordinate equals 1). As done previously in this paper, we abuse the notation and view $\mb{F}_2^k$ as the subset $\mb{F}_2^k\times \{0\}^{\mb{N}\setminus[k]}$ of $\mb{F}_2^\omega$. Note that any $L\in \mb{F}_2^k$ can be viewed as a linear map $L:V^k\to V$ for any $\mb{F}_2$-vector space $V$. Indeed, if $L=(\lambda_1,\ldots,\lambda_k)$ then $L(x):=\lambda_1x_1+\cdots+\lambda_kx_k$ for any $x\in V^k$.

Let us now recall the above-mentioned notion of consistency from \cite[Definition 3.3]{HHH}.

\begin{defn}[Consistency]\label{consistent}
Let $\mc{L} = \{L_1, \dots, L_m\}$ be a system of linear forms.
A sequence of elements $b_1, \dots, b_m \in \mb{T}$ is said to be {\em
  $(k,r)$-consistent with $\mc{L}$} if there exists a $(k,r)$-polynomial $P$ (i.e.\ of degree $\leq k$ and depth $r$) and a point $x$ such that $P(L_i(x))=b_i$ for every $i \in [m]$. 
\end{defn}

The corresponding version of this definition for nilspaces is then the following.
\begin{defn}[Consistency, nilspace version]\label{def:consis-nil}
Let $\mc{L} = \{L_1, \dots, L_m\}$ be a finite subset of $\mb{F}_2^\omega$.
A sequence of elements $b_1, \dots, b_m \in \mb{T}$ is $(k,r)$-consistent with $\mc{L}$ if there exists a morphism $P\in \hom(\mc{D}_1(\mb{F}_2^n),\abph_{k,\ell})$ for some $\ell=\ell_{k,r}$ taking values in $(\frac{1}{2^r}\cdot\mb{Z})/\mb{Z}$, and a point $x\in (\mb{F}_2^n)^s$ such that $P\co L_i(x)=b_i$ for every $i \in [m]$, where $s$ is such that $\mc{L}\subset \mb{F}_2^s$.
\end{defn}

\begin{corollary}\label{cor:equiva-consis}
Definitions \ref{consistent} and \ref{def:consis-nil} are equivalent.
\end{corollary}

\begin{proof}
This follows from Lemma \ref{lem:poly-morph-corresp}.
\end{proof}
\noindent Note that there is a natural extension of this definition for groups of the form $\prod_{j=1}^\infty (\frac{1}{2^{n_j}}\cdot\mb{Z})/\mb{Z}$ for any sequence $(n_j\in \mb{Z}_{\ge 0})_{j\ge 1}$.
\begin{defn}[Consistency, general groups]
Let $\mc{L} = \{L_1, \dots, L_m\}$ be a system of linear forms, let $\textbf{d}\in \mb{Z}_{>0}^{\infty}$, $\textbf{k}\in \mb{Z}_{\ge 0}^\infty$ and $G=\prod_{j=1}^\infty (\frac{1}{2^{n_j}}\cdot\mb{Z})/\mb{Z}$ for some $(n_j\in \mb{Z}_{\ge 0})_{j\ge 1}$. We say that a sequence of elements $b_1,\ldots,b_m\in G$ is consistent with $\mc{L}$ if for every $j\in \mb{N}$, $b_1\sbr{j},\ldots,b_m\sbr{j}\in  (\frac{1}{2^{n_j}}\cdot\mb{Z})/\mb{Z}\subset \mb{T}$ is $(\textbf{d}\sbr{j},\textbf{k}\sbr{j})$-consistent with $\mc{L}$.
\end{defn}

\begin{lemma}\label{lem:eq-consis}
Let $\mc{L}=\{L_1,\ldots,L_m\}\subset  \mb{F}_2^s\times \{0\}^{\mb{N}\setminus [s]}$ be a finite set of affine linear forms where $L_i=(1,L_i')$ for all $i\in [m]$ and $s\ge 1$. Let $G=\prod_{j=1}^\infty (\frac{1}{2^{n_j}}\cdot\mb{Z})/\mb{Z}$ for some $(n_j\in \mb{Z}_{\ge 0})_{j\ge 1}$, and let $\textbf{d}\in \mb{Z}_{>0}^\infty$, $\textbf{k}\in \mb{Z}_{\ge 0}^\infty$. Then the set of points $H\subset G^{\mc{L}}$ that are consistent with $\mc{L}$ forms a closed subgroup. Furthermore, there exists a 2-homogeneous filtration $G_\bullet$ on $G$ such that the group nilspace generated by $(G,G_\bullet)$ equals $\prod_{j=1}^\infty \abph_{\textbf{k}\sbr{j},\ell\sbr{j}}$ for some $\ell=\ell_{\textbf{k},\textbf{d}}\in \mb{Z}_{>0}^\infty$ and the map $T:\hom(\mc{D}_1(\mb{F}_2^{s-1}),G)\to G^{\mc{L}}$, $f\mapsto(f(L_1'),\ldots,f(L_m'))$ is a surjective homomorphism onto $H$ and thus the Haar measure on $H$ equals the pushforward of the Haar measure on $\hom(\mc{D}_1(\mb{F}_2^{s-1}),G)$ under the map $T$.
\end{lemma}

\begin{remark}
Note that this result applies in particular to the group $\mb{G}_\infty$ used in \cite[Definition 3.2]{HHH}.
\end{remark}

\begin{proof}[Proof of Lemma \ref{lem:eq-consis}]
Let $b_1,\ldots,b_m\in G$ be a sequence of elements consistent with $\mc{L}$. Fix any $j\in \mb{N}$. By Corollary \ref{cor:equiva-consis} this means that there exists $P_j\in \hom(\mc{D}_1(\mb{F}_2^n),\abph_{k\sbr{j},\ell\sbr{j}})$ and a point $x\in (\mb{F}_2^n)^s$ such that $P_j(L_i(x))=b_i\sbr{j}$ for all $i\in [m]$. Note that since the $L_i$ are affine forms, the values $L_i(x)$ are the same thing as images $A(L_i')$ where $L_i'\subset \mb{F}_2^{s-1}$ and $A\in \cu^{s-1}(\mb{F}_2^n)$ (recall that $L_i=(1,L_i')$). Hence in this case $P_j(L_i(x))=P_j(A(L_i'))$. But as $P_j\co A\in \hom(\mc{D}_1(\mb{F}_2^{s-1}),G)$, this condition reduces to saying that there exists $P_j':=P_j\co A\in \hom(\mc{D}_1(\mb{F}_2^{s-1}),G)$ such that $P_j'(L_i')=b_i\sbr{j}$ for all $i\in [m]$. As this holds for every $j\in \mb{N}$ we have proved that the map $T:\hom(\mc{D}_1(\mb{F}_2^{s-1}),G)\to G^{\mc{L}}$, $f\mapsto(f(L_1'),\ldots,f(L_m'))$ is surjective onto $H$.

Moreover, as $T$ is continuous and $\hom(\mc{D}_1(\mb{F}_2^{s-1}),G)$ is compact, we have that $H$ is a closed group as well. And the fact that $T$ is a surjective homomorphism onto $H$ implies that the Haar measure on $H$ is the pushforward of the Haar measure of the group $\hom(\mc{D}_1(\mb{F}_2^{s-1}),G)$.
\end{proof}

\end{document}